\documentclass[11pt]{article}
\usepackage{mathrsfs}
\usepackage{amssymb}
\usepackage{latexsym}
\usepackage{amsmath}
\usepackage{galois}
\usepackage{amsthm}
\usepackage{enumerate}
\usepackage{hyperref}
\usepackage{algorithm}
\usepackage{algorithmic}
\usepackage{amsfonts}
\usepackage{graphicx}
\usepackage{epstopdf}
\usepackage{amsopn}
\usepackage{enumitem}

\usepackage{booktabs}
\usepackage{multirow}
\usepackage{bm}
\usepackage{color}

\newcommand{\inp}[1]{\left\langle#1\right\rangle}
\DeclareMathOperator*{\argmin}{argmin}

\definecolor{blue}{rgb}{0,0,0.9}
\definecolor{red}{rgb}{0.9,0,0}
\definecolor{green}{rgb}{0,0.9,0}
\definecolor{lightgreen}{rgb}{0.1,0.5,0.1}
\definecolor{purple}{rgb}{0.5,0,0.5}
\hypersetup{
  colorlinks=true,
  breaklinks=true,
  urlcolor=blue,
  citecolor=lightgreen,
  linkcolor=lightgreen,
  bookmarksopen=false,
  draft=false
}

\newtheorem{theorem}{Theorem}[section]

\newtheorem{lemma}[theorem]{Lemma}

\newtheorem{proposition}[theorem]{Proposition}
\newtheorem{remark}[theorem]{Remark}

\numberwithin{equation}{section}
\author{
Dan Luo\thanks{
Key Laboratory of NSLSCS, Ministry of Education, School of Mathematical Sciences,
Nanjing Normal University, Nanjing 210023, China
(\texttt{luodmath@163.com}, \texttt{hlsun@njnu.edu.cn}, \texttt{yymath9527@163.com}).
}
\and
Hailin Sun\footnotemark[1]
\and
Lei Yang\thanks{
School of Computer Science and Engineering, Guangdong Province Key Laboratory of Computational Science,
Sun Yat-sen University, Guangzhou 510275, China
(\texttt{yanglei39@mail.sysu.edu.cn}).
}
\and
Yang You\footnotemark[1]
}

\date{}

\title{A Single-Loop Minorized Dual Decomposition Method for Nonsmooth Multi-Stage Stochastic Programming}
\usepackage[top=1in, bottom=1in, left=1.15in, right=1.15in]{geometry}

\begin{document}
\maketitle

\begin{abstract}
In this paper, we study multi-stage stochastic programming (MSP) problems with nonsmooth composite objectives. Tailored to their intrinsic stage-wise and scenario-wise structure, we develop a single-loop minorized dual decomposition method, in which each iteration constructs a minorized problem and its restricted Wolfe dual, and then performs \textit{one iteration} of the symmetric Gauss--Seidel based inexact alternating direction method of multipliers on the resulting dual problem to generate the next iterate. A key feature of the proposed optimization framework is that the resulting updates preserve the stage-wise and scenario-wise decomposable structure of the MSP problem and are suitable for parallel implementation. We establish global convergence of the generated iterates for the three-stage case and further establish the corresponding global convergence theorem for the general multi-stage setting. Numerical experiments illustrate the computational viability of the proposed framework and its favorable scaling behavior with respect to the stage-wise and scenario-wise structure.
\end{abstract}


\section{Introduction}
Multi-stage stochastic programming (MSP) \cite{shapiro2021lectures} provides a fundamental framework for sequential decision-making under uncertainty. It has been widely used in supply chain management \cite{khalilabadi2020multi,stranieri2024combining}, financial risk control \cite{mulvey2004financial,nickel2012multi}, energy system planning \cite{pereira1991multi}, and related areas. In this paper, we consider a class of nonsmooth composite MSP problems of the following form:
\begin{equation}\label{MSP}
\begin{aligned} 
\textstyle \min\limits_{x_1} ~ &  \theta_1(x_1)+h_1(x_1)+\frac{1}{N_2} \sum\limits_{k_2=1}^{N_2}\Big(
\min\limits_{x_2(\xi_{[2]}^{k_{[2]}})} \theta_{2,k_{[2]}}(x_2(\xi_{[2]}^{k_{[2]}}),\xi_2 ^{k_2}) 
+ h_{2,k_{[2]}}(x_2(\xi_{[2]}^{k_{[2]}}),\xi_2 ^{k_2}) \Big.   \\
&\Big.  +\cdots+\frac{1}{N_T} \sum\limits_{k_T=1}^{N_T} \Big( \min\limits_{x_T( \xi_{[T]}^{k_{[T]}})} \theta_{T, k_{[T]}}(x_T( \xi_{[T]}^{k_{[T]}}), \xi_{T}^{k_{T}}) + h_{T,k_{[T]}}(x_T( \xi_{[T]}^{k_{[T]}}), \xi_{T}^{k_{T}})\Big)\Big)  \vspace{-2mm} \\
\mbox{s.t.}~
&A_1x_1= b_1, \\
&B_t(\xi_{[t]}^{k_{[t]}})x_{t-1}(\xi_{[t-1]}^{k_{[t-1]}}) + A_t(\xi_{[t]}^{k_{[t]}})x_t(\xi_{[t]}^{k_{[t]}})= b_t(\xi_{[t]}^{k_{[t]}}),
~~t=2,\cdots,T,\\
&x_1 \in \mathcal{K}_1 \subset \mathbb{R}^d, 
~~x_t(\xi_{[t]}^{k_{[t]}})\in \mathcal{K}_{t, k_{[t]}} \subset \mathbb{R}^d,
~~t=2,\cdots,T. 
\end{aligned}  \vspace{-1mm}
\end{equation}
The detailed structure of problem \eqref{MSP} is as follows:   \vspace{1mm}
\begin{itemize}[leftmargin=0.65cm]
\item \textbf{Stochastic process}: This problem involves a discrete stochastic process $\big\{\xi_1, \xi_2$, $\cdots, \xi_T\big\},$ where $\xi_1$ is deterministic. We denote a scenario path up to stage $t$ as $\xi_{[t]}^{k_{[t]}} = \big(\xi_2^{k_2}, \cdots, \xi_t^{k_t}\big)$, where $k_{[t]} = \{k_2, \cdots, k_t\}$ and $k_s \in\{ 1, \cdots, N_s\}$ for $s=2, \cdots, t$.  For $t=2,\cdots,T$, given a realization $\xi_{[t-1]}^{k_{[t-1]}}$, the stage-$t$ random variable $\xi_t$ follows a discrete distribution with a support set $\big\{\xi_t^{k_{[t-1]},1},\cdots,  \xi_t^{k_{[t-1]},N_t}\big\}$ and conditional probability $P\big(\xi_t = \xi_t^{k_{[t]}} \mid \xi_{[t-1]}^{k_{[t-1]}}\big) = \frac{1}{N_t}$.
\vspace{1mm}
\item {\bf Functions, operators, sets and vectors}: 
The functions
 $\theta_1$ and $\theta_{t, k_{[t]}}$ are closed proper convex, while $h_1$ and $h_{t, k_{[t]}}$ are continuously differentiable convex functions with Lipschitz continuous gradients for $t=2,\cdots,T$.  The constraints involve linear operators 
$A_1,A_t(\xi_{[t]}^{k_{[t]}}),B_t(\xi_{[t]}^{k_{[t]}}):\mathbb{R}^d\to\mathbb{R}^{\ell}$, the closed convex sets $\mathcal{K}_1$ and $\mathcal{K}_{t,k_{[t]}}$,  
and vectors $b_1$ and $b_t(\xi_{[t]}^{k_{[t]}})$, $t=2,\cdots,T$. 
\end{itemize}
\vspace{1mm}

Problem \eqref{MSP} exhibits an intrinsic stage-wise and scenario-wise structure. The interstage linear constraints couple decisions between consecutive stages along each scenario path, whereas the objective functions and individual set constraints are separable across stages and scenarios. This coexistence of interstage coupling and scenario-wise decomposability is the main computational structure exploited in this paper. Moreover, the nonsmooth composite terms motivate decomposition methods that can preserve and utilize this stochastic multi-stage structure.

Research on MSP has attracted considerable attention, and substantial progress has been made in both mathematical theory and solution methodologies \cite{shapiro2021lectures}. A variety of classical solution approaches have been developed for MSP problems. For example, the branch-and-bound method \cite{norkin1998branch} employs randomized upper and lower bound estimates to guide the partitioning procedure; the cutting-plane method \cite{guan2009cutting} generates cutting planes for multi-stage stochastic integer programs by using inequalities valid for individual scenarios; and the progressive hedging algorithm \cite{rockafellar1991scenarios} coordinates scenario subproblems through a decomposition framework to enforce nonanticipativity.
Stochastic dual dynamic programming (SDDP), initiated by Pereira and Pinto \cite{pereira1991multi}, is another important class of methods for MSP. SDDP-type methods combine dynamic programming with cutting-plane approximations and have been extended to MSP problems with increasingly complex structural features \cite{guigues2020inexact,guigues2021inexact,zou2019stochastic}. 
For problems with high-dimensional decision variables and a relatively small number of stages, dynamic stochastic approximation methods \cite{lan2024numerical,lan2021dynamic} provide another effective approach by combining saddle-point reformulations with inexact primal-dual stochastic approximation schemes. 

Splitting, multiplier-updating, and primal-dual techniques have played an important role in large-scale structured convex optimization. Representative developments include the sequential Lagrange multiplier updating scheme for separable convex programming \cite{dai2017sequential}, the linearly convergent majorized alternating direction method of multipliers (ADMM) with indefinite proximal terms for convex composite programming \cite{zhang2020linearly}, and approximate first-order primal-dual algorithms for saddle point problems \cite{jiang2021approximate}. Symmetric Gauss--Seidel (sGS) decomposition techniques have also been incorporated into ADMM-type frameworks for multi-block convex composite optimization \cite{chen2017efficient,chen2019unified,li2018qsdpnal} and have been applied or extended in various structured optimization settings \cite{ding2023proximal,jin2025dual,lam2021semi,yang2021fast}. These works provide useful computational and analytical tools for exploiting separability and handling linear coupling in large-scale structured optimization, and motivate decomposition methods for problems with coupled multi-block structures.

Despite these developments, structure-exploiting splitting techniques for nonsmooth composite MSP remain relatively limited.  In the two-stage setting, Arp{\'o}n et al.~\cite{arpon2020admm} applied the three-block ADMM \cite{sun2015convergent} to solve two-stage stochastic programming. For the multi-stage setting considered in this paper, a natural approach is to apply the sGS inexact ADMM (sGS-iADMM) framework \cite{chen2017efficient} directly to solve a majorized formulation of problem \eqref{MSP} by introducing suitable auxiliary variables. However, this approach requires many additional auxiliary constraints, which substantially increase both the number of variables and the number of constraints and may therefore lead to significant computational overhead. To address this issue, we develop a \underline{s}ingle-\underline{l}oop \textit{\underline{m}inorized} \underline{d}ual \underline{d}ecomposition (SL-MDD) method tailored to the intrinsic stage-wise and scenario-wise structure of MSP. Specifically, motivated by the majorization treatment of the primal problem, we introduce a corresponding minorization scheme on the dual problem.

The main contributions of this paper are summarized as follows:
\begin{itemize}[leftmargin=0.65cm]
 \item We develop the SL-MDD method for solving the MSP problem \eqref{MSP}. Starting from the intrinsic stage-wise and scenario-wise structure of deterministic-equivalent nonsmooth MSP, we construct at each iteration a minorized problem, derive its restricted Wolfe dual, and then perform one sGS-iADMM sweep to update the next iterate; the resulting single-loop computational flow is illustrated in Figure~\ref{fig:flowchart} (in Section 2.1). This decomposition route leads to a computationally favorable block structure: most updates reduce to structured linear-system solves, proximal mappings, or projections, while the remaining coupled $y$-blocks are further decoupled through Lemma~\ref{lem:key}, thereby yielding efficiently computable and stage/scenario-wise parallelizable subproblems.

  \item For convergence analysis, we show that the proposed SL-MDD method admits an equivalent single-loop minorized dual inexact semi-proximal ADMM (SL-MD-isPADMM) representation after a suitable grouping of the dual variables. This equivalent representation, derived via the sGS decomposition theorem \cite[Theorem 1]{li2019block}, is used only as an analytical device rather than as the conceptual origin of the algorithm. Based on this representation, we establish global convergence for the three-stage case and then carry out the same proof pattern to obtain the corresponding global convergence theorem for the general multi-stage setting.
 \item Numerical experiments illustrate the computational viability of the proposed framework for MSP problems with a moderate number of stages and relatively large per-stage dimensions. The results further demonstrate favorable scaling behavior with respect to the stage-wise and scenario-wise structure.

\end{itemize}

The rest of this paper is organized as follows. In Section \ref{sec:alg}, we present the proposed SL-MDD method and discuss the stage-wise and scenario-wise computational structure of its updates. In Section \ref{sec:convergence}, we establish the global convergence of the generated iterates. Numerical experiments are reported in Section \ref{sec:ne} to illustrate the computational viability and scalability of the proposed framework. Finally, Section \ref{sec:conclusions} concludes the paper.

\textbf{Notation.} 
Let $G: \mathbb{R}^d \to \mathbb{R}^d$ be a self-adjoint positive semidefinite linear operator. For any vector $u\in\mathbb{R}^d$, its $G$-norm is defined by $\|u\|_G := \sqrt{\langle u, Gu \rangle}$. For a proper closed convex function $f:\mathbb{R}^d\to\mathbb{R}\cup\{+\infty\}$, its Fenchel conjugate is defined by $f^*(u) := \sup_{v}\left\{ \langle u, v \rangle - f(v)\right\}$, and its proximal mapping is defined by $\operatorname{Prox}_{f}(u) := \operatorname*{arg\,min}_{v} \left\{ f(v) + \frac{1}{2} \|v - u\|^2 \right\}$. When the quadratic term is measured in the $G$-norm with $G \succ 0$, we write $\text{Prox}^{f}_G(u) := \arg\min_{v} \left\{ f(v) + \frac{1}{2} \|v - u\|_G^2 \right\}$. Let $\mathcal{H}$ be a self-adjoint block operator and write $\mathcal{H}=\mathcal{H}_d+\mathcal{H}_u+\mathcal{H}_u^*$, where $\mathcal{H}_d$ and $\mathcal{H}_u$ denote the block-diagonal and strictly block upper-triangular parts of $\mathcal{H}$, respectively. Assume that $\mathcal{H}_d\succ0$. The sGS operator associated with $\mathcal{H}$ is defined by $\mathrm{sGS}(\mathcal{H}):=\mathcal{H}_u\,\mathcal{H}_d^{-1}\,\mathcal{H}_u^*$.

\section{The single-loop minorized dual decomposition method}
\label{sec:alg}

In this section, we present the proposed SL-MDD method for problem \eqref{MSP}. For ease of exposition, we mainly focus on the three-stage stochastic programming (TSP). The same construction and analysis extend directly to the general multi-stage setting.

\subsection{Construction for the three-stage case}
\label{sec:alg-des}
Even in the three-stage setting, the structure of MSP problems is still complex.  For simplicity, let $x_{2,i}:=x_2(\xi_2^i)$, $x_{3,ij}:=x_3(\xi_2^i,\xi_3^j)$, $A_{2,i}:=A_2(\xi_2^i)$, $A_{3,ij}:=A_3(\xi_2^i,\xi_3^j)$, $B_{2,i}:=B_2(\xi_2^i)$, $B_{3,ij}:=B_3(\xi_2^i,\xi_3^j)$,
$b_{2,i}:=b_2(\xi_2^i)$,  and $b_{3,ij}:=b_3(\xi_2^i,\xi_3^j)$ for $i=1,\cdots,N_2$ and $j=1,\cdots,N_3$. 
We also introduce the following notation for the Cartesian products:
\begin{align*}
\mathcal{K}_2 := \mathcal{K}_{2,1} \times \mathcal{K}_{2,2} \times \cdots \times \mathcal{K}_{2,N_2}, 
\quad
\mathcal{K}_3 := \mathcal{K}_{3,11} \times \mathcal{K}_{3,12} \times \cdots \times \mathcal{K}_{3,N_2N_3}. 
\end{align*}
Moreover, the vectors $x_2$, $x_3$, $b_2$ and $b_3$ are defined as follows:
\begin{align*}
x_2:=\begin{pmatrix}
x_{2,1}\\
\vdots\\
x_{2,N_2}
\end{pmatrix}, 
~x_3:=\begin{pmatrix}
x_{3,11}\\
\vdots\\
x_{3,N_2N_3}
\end{pmatrix}, 
~b_2:=\begin{pmatrix}
b_{2,1}\\
\vdots\\
b_{2,N_2}
\end{pmatrix},
~b_3:=\begin{pmatrix}
b_{3,11}\\
\vdots\\
b_{3,N_2N_3}
\end{pmatrix}.
\end{align*}
The matrices $A_2$ and $A_3$ are defined as block diagonal matrices:
\begin{equation*}
A_2:=\begin{pmatrix}
A_{2,1} & & \\
 & \ddots & \\
 & & A_{2,N_2}
\end{pmatrix},
\quad
A_3:=\begin{pmatrix}
A_{3,11} & & \\
 & \ddots & \\
 & & A_{3,N_2N_3}
\end{pmatrix}.
\end{equation*}
The matrices $B_2$ and $B_3$ are defined as follows:
\begin{equation*}
B_2 := \begin{pmatrix}
B_{2,1}\\
\vdots\\
B_{2,N_2}
\end{pmatrix},~
B_3 := \begin{pmatrix}
B_{3,1} & & \\
& \ddots & \\
& & B_{3,N_2}
\end{pmatrix}
~\text{with}~ 
B_{3,i}:=\begin{pmatrix}
B_{3,i1} \\
\vdots \\
B_{3,iN_3}
\end{pmatrix},i=1,\cdots,N_2.
\end{equation*}
Finally, the functions $\theta_2$, $\theta_3$, $h_2$ and $h_3$ are defined as:
\begin{equation}\label{theta&h}
	\begin{aligned}
		&\theta_2(x_2):=\frac{1}{N_2}\sum_{i=1}^{N_2}\theta_{2,i}(x_{2,i}), ~~\theta_3(x_3):=\frac{1}{N_2N_3}\sum_{i=1}^{N_2}\sum_{j=1}^{N_3}\theta_{3,ij}(x_{3,ij}),\\[3pt]
		&h_2(x_2):=\frac{1}{N_2}\sum_{i=1}^{N_2}h_{2,i}(x_{2,i}), ~~h_3(x_3):=\frac{1}{N_2N_3}\sum_{i=1}^{N_2}\sum_{j=1}^{N_3}h_{3,ij}(x_{3,ij}).
	\end{aligned}
\end{equation}
Using the above notations, we can reformulate the TSP problem (i.e., the three-stage case of \eqref{MSP}) as the following compact convex composite optimization problem:
\begin{equation}\label{P1}
\begin{aligned}
\min\limits_{x_t} \quad & {\sum_{t=1}^3}(\theta_t(x_t)+ \delta_{\mathcal{K}_t}(x_t)+h_t(x_t))\\
\mbox{s.t.}\quad 
&A_1x_1= b_1,\\
&B_2x_1 + A_2x_2= b_2,\\
&B_3x_2 + A_3x_3= b_3, 
\end{aligned}
\end{equation}
where $\delta_{\mathcal{K}_t}$ denotes the indicator function of $\mathcal{K}_t$, i.e., $\delta_{\mathcal{K}_t}(x)=0$ if $x\in\mathcal{K}_t$ and $\delta_{\mathcal{K}_t}(x)=\infty$ otherwise. To handle the nonsmooth terms in the objective function, we introduce auxiliary variables $u_t, s_t$ ($t=1,2,3$) and further reformulate the problem as:
\begin{equation}\label{P2}
\begin{aligned}
\min\limits_{x_{t}, u_{t}, s_{t}} 
\quad &  \sum_{t=1}^3(\theta_t(u_t)+ \delta_{\mathcal{K}_t}(s_t)+h_t(x_t))\\
\mbox{s.t.}\quad \,\,
&A_1x_1 = b_1, \\
&B_2x_1 + A_2x_2 = b_2,\\ 
&B_3x_2 + A_3x_3 = b_3, \\
&u_t=x_t,~~s_t=x_t,~~t=1,2,3.
\end{aligned}
\end{equation} 
To solve this problem, we adopt a minorization strategy to construct quadratic \emph{lower} bounds for the functions $h_1,h_2,h_3$. Specifically, choose self-adjoint positive semidefinite linear operators
$Q_1:\mathbb{R}^d\to\mathbb{R}^d$,
$Q_2:\mathbb{R}^{dN_2}\to\mathbb{R}^{dN_2}$, and
$Q_3:\mathbb{R}^{dN_2N_3}\to\mathbb{R}^{dN_2N_3}$
such that the following quadratic lower-bound condition holds: for any
$x_1,x_1'\in\mathbb{R}^d$, $x_2,x_2'\in\mathbb{R}^{dN_2}$, and
$x_3,x_3'\in\mathbb{R}^{dN_2N_3}$,
\begin{equation}\label{tilde_h}
h_t(x_t) \textstyle\geq \tilde{h}_t(x_t;x_t'):= h_t(x_t')+\langle \nabla h_t(x_t'),x_t-x_t'\rangle+\frac{1}{2}\|x_t-x_t'\|_{Q_t}^2,\quad t=1,2,3.
\end{equation}
Using this fact, we can construct a minorized problem of \eqref{P2} at points $x_1',x_2',x_3'$:
\begin{equation}\label{P3}
\begin{aligned} 
\min\limits_{x_{t}, u_{t}, s_{t}} 
\quad &  \sum_{t=1}^3(\theta_t(u_t)+ \delta_{\mathcal{K}_t}(s_t)+\tilde{h}_t(x_t;x_t'))\\
\mbox{s.t.}\quad \,\,
&A_1x_1= b_1, \\
&B_2x_1 + A_2x_2= b_2,\\ 
&B_3x_2 + A_3x_3= b_3, \\
&u_t=x_t,~~s_t=x_t,~~t=1,2,3.
\end{aligned}
\end{equation}
Let $x_{1:t}:=(x_1,\cdots, x_t)$, $u_{1:t}:=(u_1,\cdots,u_t)$, $s_{1:t}:=(s_1,\cdots, s_t)$, $y_{1:t}:=(y_1,\cdots,y_t)$, $v_{1:t}:=(v_1,\cdots,v_t)$, $z_{1:t}:=(z_1,\cdots, z_t)$ and $x'_{1:t}:=(x'_1,\cdots, x'_t)$ for $t=1,2,3$. The Lagrangian function for the minorized problem \eqref{P3} is given by
\begin{equation*}
\begin{aligned}
&L^{x'_{1:3}}(x_{1:3},u_{1:3},s_{1:3};y_{1:3},v_{1:3},z_{1:3})\\
:=~&\sum_{t=1}^3\Big(\frac{1}{2}\langle x_t, \,Q_t x_t \rangle +\langle c_t(x'_t)-A_t^{*}y_t-B_{t+1}^{*}y_{t+1}-v_t-z_t, \,x_t \rangle +\theta_t(u_t)+\langle v_t, u_t \rangle  \\
& + \delta_{\mathcal{K}_t}(s_t) + \langle z_t, s_t \rangle +\langle b_t, y_t \rangle +h_t(x_t')+\frac{1}{2}\langle x_t',Q_t x_t'\rangle -\langle \nabla h_t(x_t'),x_t' \rangle \Big),
\end{aligned}
\end{equation*}
where $B_4^*=0$, $y_4=0$, and $c_t(x'_t):=\nabla h_t(x'_t)-Q_t x'_t$ for $t=1,2,3$.
Moreover, for each $t=1,2,3$, we have
\begin{equation}\label{xt}
\begin{aligned}
& \quad \inf_{x_t}L^{x'_{1:3}}(x_{1:3},u_{1:3},s_{1:3};y_{1:3},v_{1:3},z_{1:3}) \\
&= \inf_{x_t}\Big\{ \frac{1}{2}\langle x_t, Q_t x_t\rangle+\langle c_t(x'_t)-A_t^{*}y_t-B_{t+1}^{*}y_{t+1}-v_t-z_t, \,x_t\rangle\Big\} \\
&=\begin{cases}
-\frac{1}{2}\langle w_t, Q_t w_t \rangle, &\text{if}~~c_t(x'_t)-A_t^{*}y_t-B_{t+1}^{*}y_{t+1}-v_t-z_t=-Q_t w_t, \\[1pt]
-\infty, &\text{otherwise,}\\
\end{cases} 
\end{aligned}
\end{equation}
and
\begin{equation*}
\begin{aligned}
& \inf_{u_t} L^{x'_{1:3}}(x_{1:3},u_{1:3},s_{1:3};y_{1:3},v_{1:3},z_{1:3}) = \inf_{u_t}\lbrace \theta_t(u_t)+\langle v_t, u_t\rangle \rbrace = -\theta_t^*(-v_t),\\[2pt]
& \inf_{s_t} L^{x'_{1:3}}(x_{1:3},u_{1:3},s_{1:3};y_{1:3},v_{1:3},z_{1:3}) = \inf_{s_t}\left\lbrace \delta_{\mathcal{K}_t}(s_t)+\langle z_t, s_t \rangle \right\rbrace = -\delta_{\mathcal{K}_t}^*(-z_t).
\end{aligned}
\end{equation*}
Let $\mathcal{W}_t:=Range(Q_t)$ denote the range space of $Q_t$ for $t=1,2,3$.
Thus, the restricted Wolfe dual problem (see, e.g., \cite{li2018qsdpnal}) of \eqref{P3} can be formulated as follows: 
\begin{equation}
\begin{aligned} \label{D1}
 \max\limits_{v_t,w_t,y_t,z_t}~&  \sum_{t=1}^3\Big(-\frac{1}{2}\langle w_t,Q_tw_t \rangle-\theta_t^*(-v_t)-\delta_{\mathcal{K}_t}^*(-z_t)+\langle b_t,y_t\rangle \\
& \quad +h_t(x_t')+\frac{1}{2}\langle x_t',Q_tx_t'\rangle-\langle \nabla h_t(x_t'),x_t'\rangle \Big)\\
\mbox{s.t.}\quad \,
&c_1(x'_1)-A_1^{*}y_1-B_2^{*}y_2-v_1-z_1=-Q_1w_1, &\\
&c_2(x'_2)-A_2^{*}y_2-B_3^{*}y_3-v_2-z_2=-Q_2w_2,\\ 
&c_3(x'_3)-A_3^{*}y_3-v_3-z_3=-Q_3w_3, \\
&w_t \in \mathcal{W}_t, \,t=1,\,2,\,3.
\end{aligned}
\end{equation}
Given $\sigma>0$, the augmented Lagrangian function of \eqref{D1} is
\begin{equation*}
\begin{aligned}
& \widehat{L}_{\sigma}^{x'_{1:3}}\big(v_{1:3},w_{1:3},y_{1:3},z_{1:3};x_{1:3}\big)
:= \sum_{t=1}^{3} \Big(\theta_t^*(-v_t)+\delta_{\mathcal{K}_t}^*(-z_t)+\frac{1}{2}\langle w_t,Q_tw_t\rangle   \\
& \quad 
+\delta_{\mathcal{W}_t}(w_t)-\langle b_t,y_t\rangle 
+\frac{\sigma}{2}\|A_t^{*}y_t+B_{t+1}^{*}y_{t+1}+v_t+z_t-c_t(x'_t)-Q_tw_t+\sigma^{-1}x_t\|^2 \\
& \quad -h_t(x_t')-\frac{1}{2}\langle x_t',Q_tx_t'\rangle+\langle \nabla h_t(x_t'),x_t'\rangle-\frac{1}{2\sigma} \| x_t \| ^2 \Big).
\end{aligned} 
\end{equation*}
With these preparations, we are now ready to present the  SL-MDD method for solving problem \eqref{P1} in Algorithm \ref{alg:tsp-md}.
\begin{algorithm}
\small
\caption{ A SL-MDD method for the TSP problem \eqref{P1}}
\label{alg:tsp-md}
\begin{algorithmic}
\STATE{Let $\tau \in  (0,\frac{1+\sqrt{5-4\alpha}}{2} ]$ be the step-length with $\alpha \in (0,1)$, and $\{\varepsilon_k\}_{k\geq0}$ be a summable sequence of nonnegative numbers. Choose an initial point $\varpi^0=(v^0_{1:3},w^0_{1:3},y^0_{1:3},z^0_{1:3},x^0_{1:3})$}. Set $k = 0$. 
\WHILE{the termination criterion is not met,}
\STATE{\textbf{Step 1:} Construct the minorized problem \eqref{P3} at the point $x^k_{1:3}$}.
\STATE{\textbf{Step 2:} Derive the restricted Wolfe dual problem \eqref{D1}.}
\STATE{\textbf{Step 3:} Call one sGS-iADMM iteration to solve \eqref{D1}:}
\STATE{\hspace{\algorithmicindent} \textbf{Step 3.1a:} Compute \vspace{-1mm}
     \begin{align}
       & \textstyle \tilde{w}_1^{k+1} \approx \arg\min_{w_1} \lbrace
      \widehat{L}_{\sigma}^{x^k_{1:3}} ( v_{1:3}^k, ( w_1,w_2^k, w_3^k ), y_{1:3}^k, z_{1:3}^k; x_{1:3}^k )  \rbrace , \label{eq:3-1-a-1}\\
       & \textstyle \tilde{w}_2^{k+1} \approx \arg\min_{w_2} \lbrace
      \widehat{L}_{\sigma}^{x^k_{1:3}} ( v_{1:3}^k, ( \tilde{w}_1^{k+1},w_2,w_3^k ), y_{1:3}^k, z_{1:3}^k; x_{1:3}^k )  \rbrace , \label{eq:3-1-a-2}\\
      & \textstyle \tilde{w}_3^{k+1} \approx \arg\min_{w_3} \lbrace
      \widehat{L}_{\sigma}^{x^k_{1:3}} ( v_{1:3}^k, ( \tilde{w}_1^{k+1},\tilde{w}_2^{k+1},w_3 ), y_{1:3}^k, z_{1:3}^k; x_{1:3}^k )  \rbrace. \label{eq:3-1-a-3}
     \end{align}
}
\vspace{-3ex}
\STATE{\hspace{\algorithmicindent} \textbf {Step 3.1b:} Compute \vspace{-1mm}
     \begin{align}
       & \textstyle v_1^{k+1} \approx \arg\min_{v_1}  \lbrace \widehat{L}_{\sigma}^{x^k_{1:3}} (  ( v_1, v_2^k, v_3^k  ), \tilde{w}_{1:3}^{k+1}, y_{1:3}^k, z_{1:3}^k; x_{1:3}^k  )  \rbrace,\label{eq:3-1-b-1}\\
       & \textstyle v_2^{k+1} \approx \arg\min_{v_2}  \lbrace \widehat{L}_{\sigma}^{x^k_{1:3}} (  ( v_1^{k+1}, v_2, v_3^k  ), \tilde{w}_{1:3}^{k+1}, y_{1:3}^k, z_{1:3}^k; x_{1:3}^k  )  \rbrace,\label{eq:3-1-b-2}\\
       & \textstyle v_3^{k+1} \approx \arg\min_{v_3}  \lbrace \widehat{L}_{\sigma}^{x^k_{1:3}} (  ( v_1^{k+1}, v_2^{k+1}, v_3  ), \tilde{w}_{1:3}^{k+1}, y_{1:3}^k, z_{1:3}^k; x_{1:3}^k  )   \rbrace. \label{eq:3-1-b-3}
     \end{align}
}
\vspace{-3ex}
\STATE{\hspace{\algorithmicindent} \textbf {Step 3.1c:} Compute \vspace{-1mm}
  \begin{equation*}
     \begin{aligned}
        &\textstyle w_3^{k+1} \approx  \arg\min_{w_3} \lbrace \widehat{L}_{\sigma}^{x_{1:3}^k} ( v_{1:3}^{k+1},  ( \tilde{w}_1^{k+1}, \tilde{w}_2^{k+1}, w_3  ), y_{1:3}^k, z_{1:3}^k; x_{1:3}^k  ) \rbrace,\\
        &\textstyle w_2^{k+1} \approx  \arg\min_{w_2} \lbrace \widehat{L}_{\sigma}^{x_{1:3}^k} ( v_{1:3}^{k+1},  ( \tilde{w}_1^{k+1}, w_2, w_3^{k+1}  ), y_{1:3}^k, z_{1:3}^k; x_{1:3}^k  ) \rbrace,\\
       &\textstyle w_1^{k+1} \approx  \arg\min_{w_1} \lbrace \widehat{L}_{\sigma}^{x_{1:3}^k} ( v_{1:3}^{k+1},  ( w_1, w_2^{k+1}, w_3^{k+1}  ), y_{1:3}^k, z_{1:3}^k; x_{1:3}^k  ) \rbrace.
     \end{aligned}
  \end{equation*} }
\vspace{-3ex}
\STATE{\hspace{\algorithmicindent}  \textbf {Step 3.2a:} Compute \vspace{-1mm}
\begin{align}
        &\textstyle\tilde{y}_1^{k+1} \approx \arg\min_{y_1} \{ \widehat{L}_{\sigma}^{x_{1:3}^k} ( v_{1:3}^{k+1}, w_{1:3}^{k+1},  ( y_1, y_2^k, y_3^k  ), z_{1:3}^k; x_{1:3}^k  ) \}, \label{eq:3-2-a-1} \\
       &\textstyle\tilde{y}_2^{k+1} \approx \arg\min_{y_2} \{ \widehat{L}_{\sigma}^{x_{1:3}^k} ( v_{1:3}^{k+1}, w_{1:3}^{k+1},  ( \tilde{y}_1^{k+1}, y_2 , y_3^k  ), z_{1:3}^k; x_{1:3}^k  ) \}, \label{eq:3-2-a-2} \\
      &\textstyle\tilde{y}_3^{k+1} \approx \arg\min_{y_3} \{ \widehat{L}_{\sigma}^{x_{1:3}^k} ( v_{1:3}^{k+1}, w_{1:3}^{k+1},  ( \tilde{y}_1^{k+1}, \tilde{y}_2^{k+1}, y_3  ), z_{1:3}^k; x_{1:3}^k  ) \}.\label{eq:3-2-a-3}
\end{align}
 }
\vspace{-3ex}
\STATE{ \hspace{\algorithmicindent} \textbf {Step 3.2b:} Compute \vspace{-1mm}
      \begin{align}
          &\textstyle z_1^{k+1} \approx \arg\min_{z_1}  \lbrace \widehat{L}_{\sigma}^{x_{1:3}^k} ( v_{1:3}^{k+1}, w_{1:3}^{k+1}, \tilde{y}_{1:3}^{k+1}, ( z_1, z_2^k, z_3^k ); x_{1:3}^k  )  \rbrace, \label{eq:3-2-b-1}\\
          &\textstyle z_2^{k+1} \approx \arg\min_{z_2}  \lbrace \widehat{L}_{\sigma}^{x_{1:3}^k} ( v_{1:3}^{k+1}, w_{1:3}^{k+1}, \tilde{y}_{1:3}^{k+1}, ( z_1^{k+1}, z_2, z_3^k ); x_{1:3}^k  )  \rbrace, \label{eq:3-2-b-2}\\
         &\textstyle z_3^{k+1} \approx \arg\min_{z_3}  \lbrace \widehat{L}_{\sigma}^{x_{1:3}^k} ( v_{1:3}^{k+1}, w_{1:3}^{k+1}, \tilde{y}_{1:3}^{k+1}, ( z_1^{k+1}, z_2^{k+1}, z_3 ); x_{1:3}^k  )  \rbrace.\label{eq:3-2-b-3}
      \end{align} }
\vspace{-3ex}
\STATE{\hspace{\algorithmicindent}  \textbf {Step 3.2c:} Compute \vspace{-1mm}
 \begin{equation*}
       \begin{aligned}
           &\textstyle  y_3^{k+1} \approx \arg\min_{y_3}  \{ \widehat{L}_{\sigma}^{x_{1:3}^k} ( v_{1:3}^{k+1}, w_{1:3}^{k+1},  ( \tilde{y}_1^{k+1}, \tilde{y}_2^{k+1}, y_3  ), z_{1:3}^{k+1}; x_{1:3}^k  )  \},\\  
          &\textstyle  y_2^{k+1} \approx \arg\min_{y_2}  \{ \widehat{L}_{\sigma}^{x_{1:3}^k} ( v_{1:3}^{k+1}, w_{1:3}^{k+1},  ( \tilde{y}_1^{k+1}, y_2, y_3^{k+1}  ), z_{1:3}^{k+1}; x_{1:3}^k  )  \},\\  
         &\textstyle  y_1^{k+1} \approx \arg\min_{y_1}  \{ \widehat{L}_{\sigma}^{x_{1:3}^k} ( v_{1:3}^{k+1}, w_{1:3}^{k+1},  ( y_1, y_2^{k+1}, y_3^{k+1}  ), z_{1:3}^{k+1}; x_{1:3}^k  )  \}.
       \end{aligned}
   \end{equation*} }
\vspace{-3ex}
\STATE{ \hspace{\algorithmicindent} \textbf {Step 3.3:} Compute 
\vspace{-1mm} 
\begin{equation*}
  \begin{aligned}
     &\textstyle x_1^{k+1} = x_1^k + \tau\sigma \left( A_1^*y_1^{k+1} + B_2^*y_2^{k+1} + z_1^{k+1} + v_1^{k+1}-c_1(x_1^k) -Q_1w_1^{k+1}  \right),\\
     &\textstyle x_2^{k+1} = x_2^k + \tau\sigma \left(A_2^*y_2^{k+1} + B_3^*y_3^{k+1} + z_2^{k+1} + v_2^{k+1}-c_2(x_2^k) -Q_2w_2^{k+1}  \right),\\
     &\textstyle x_3^{k+1} = x_3^k + \tau\sigma \left(A_3^*y_3^{k+1} + z_3^{k+1} + v_3^{k+1}-c_3(x_3^k) -Q_3w_3^{k+1} \right).
  \end{aligned}  
  \end{equation*}}
\ENDWHILE
\end{algorithmic}
\end{algorithm}

In Algorithm \ref{alg:tsp-md}, at the $k$-th iteration, we first construct a minorized problem \eqref{P3} at the point $x_{1:3}^k$ (\textbf{Step 1}), and derive its restricted Wolfe dual problem \eqref{D1} (\textbf{Step 2}). We then perform one sGS-iADMM iteration to approximately solve \eqref{D1} (\textbf{Step 3}). The updated $x_{1:3}^{k+1}$ is then used to construct the next minorized problem \eqref{P3} and its corresponding restricted Wolfe dual problem \eqref{D1}. Repeating this procedure produces a sequence of iterates that progressively approximates an optimal solution of the original problem \eqref{P2} (and hence \eqref{P1}).  The overall process is illustrated in the flowchart in Figure \ref{fig:flowchart}.

\begin{figure}[htbp]
\centering
\includegraphics[width=0.8\textwidth]{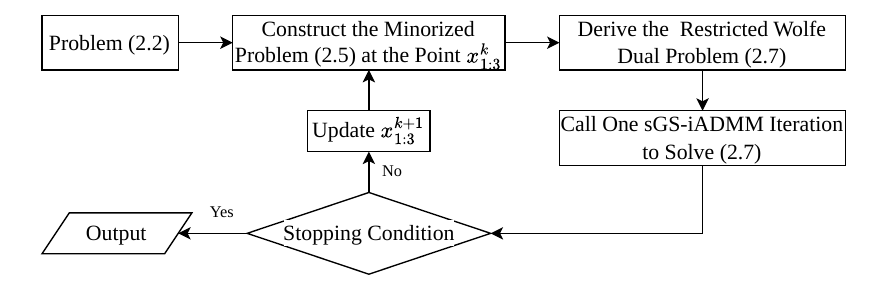}
\setlength{\abovecaptionskip}{-1pt}  
\caption{Flowchart for Solving Problem \eqref{P1}.}
\label{fig:flowchart}
\end{figure}

\begin{remark}
In Algorithm \ref{alg:tsp-md}, the symbol ``$\approx$'' indicates that the
corresponding block subproblem is solved \textit{inexactly}. Specifically, for example, the update
in \eqref{eq:3-1-a-1} means that $\tilde{w}_1^{k+1}$ is computed such that there
exists an error vector $d_{\tilde{w}_1}^k$ satisfying
\begin{equation*}
d_{\tilde{w}_1}^k \in
\partial_{\tilde{w}_1}\widehat{L}_{\sigma}^{x_{1:3}^k}
\left(v_{1:3}^k,(\tilde{w}_1^{k+1},w_2^k,w_3^k),y_{1:3}^k,z_{1:3}^k;x_{1:3}^k\right) 
\quad \mathrm{with} \quad
\|d_{\tilde{w}_1}^k\|\leq \varepsilon_k,
\end{equation*}
where $\varepsilon_k$ is a prescribed tolerance parameter. The same interpretation applies to the other subproblems. Notably, by allowing inexact minimization of the subproblems, Algorithm \ref{alg:tsp-md} provides a flexible iterative framework, which accommodates both subproblems with closed-form solutions and subproblems that must be solved only approximately by iterative methods.
\end{remark}

\subsection{Computational structure and parallel implementation of the SL-MDD method}
\label{sec:clo}
We next discuss the computational structure of one iteration of Algorithm~\ref{alg:tsp-md}. The main point is to show how the block updates can be carried out by exploiting the stage-wise and scenario-wise structure of the TSP problem. We first establish a key lemma that will be used to handle the coupled $y$-updates.

\begin{lemma}\label{lem:key}
Assume that each $V_i$ is invertible, and consider  
\begin{equation}\label{A1}
 U_i\sum_{\mu=1}^{N}U_\mu^*y_\mu + V_iy_i + c_i = 0,
\end{equation} 
where $U_i$, $V_i$, and $c_i$ $(i=1,\cdots,N)$ are given. Then, the solution for $y_i$ is given by $y_i=-V_i^{-1}(U_i\bar{b}+c_i),$ where $\bar{b}:=-\big(I+\sum_{\mu=1}^{N}U_\mu^*V_\mu^{-1}U_\mu\big)^{-1}\big(\sum_{\mu=1}^{N}U_\mu^*V_\mu^{-1}c_\mu\big)$ with $I$ being the identity matrix. 
\end{lemma}
\begin{proof}
Setting $i = \mu$ in \eqref{A1} yields 
\begin{equation*}
 U_\mu\sum_{\mu=1}^{N}U_\mu^*y_\mu + V_\mu y_\mu + c_\mu = 0.
\end{equation*}
Multiplying both sides by $U_\mu^*V_\mu^{-1}$ and summing over $\mu=1$ to $N$ gives
\begin{equation*}
 \sum_{\mu=1}^{N}U_\mu^*V_\mu^{-1}U_\mu\sum_{\mu=1}^{N}U_\mu^*y_\mu 
+ \sum_{\mu=1}^{N}U_\mu^*y_\mu + \sum_{\mu=1}^{N}U_\mu^*V_\mu^{-1}c_\mu = 0,
\end{equation*}
which implies that
\begin{equation}\label{A2}
 \bar{b}:=\sum_{\mu=1}^{N}U_\mu^*y_\mu
= -\left(I+\sum_{\mu=1}^{N}U_\mu^*V_\mu^{-1}U_\mu\right)^{-1}
\left(\sum_{\mu=1}^{N}U_\mu^*V_\mu^{-1}c_\mu\right).
\end{equation}
Substituting \eqref{A2} into \eqref{A1} immediately yields the desired result.
\end{proof}

To apply Lemma \ref{lem:key}, we assume throughout this subsection that  $A_1$, $A_{2,i}$ and $A_{3,ij}$ $(i=1,\cdots,N_2,~j=1,\cdots,N_3)$ are of full row rank, so that
$A_1A_1^*$, $A_{2,i}A_{2,i}^*$ and $A_{3,ij}A_{3,ij}^*$ are invertible.
The computational details for each step are presented below.

\vspace{1mm}
\noindent\textbf{Steps 3.1a\,\&\,3.1c:} In Step 3.1a, 
let $R_1^k := A_1^*y_1^k + \frac{1}{N_2}\sum_{i=1}^{N_2}B_{2,i}^*y_{2,i}^k + z_1^k - c_1(x_1^k)+ \sigma^{-1}x_1^k$, $R_{2,i}^k := A_{2,i}^*y_{2,i}^k + \frac{1}{N_3}\sum_{j=1}^{N_3} B_{3, ij}^*y_{3, ij}^k + z_{2,i}^k - c_{2,i}(x_{2,i}^k) + \sigma^{-1}x_{2,i}^k$, $R_{3, ij}^k := A_{3, ij}^*y_{3, ij}^k + z_{3, ij}^k - c_{3, ij}(x_{3, ij}^k) + \sigma^{-1}x_{3, ij}^k$. Then, the update in \eqref{eq:3-1-a-1} is 
 \begin{equation*}
     \begin{aligned}
       \tilde{w}_1^{k+1} 
       =& \,\argmin_{w_1\in Range(Q_1)}\Big\{\frac{1}{2}\langle w_1,Q_1w_1\rangle+\frac{\sigma}{2} \| v_1^k-Q_1w_1+ R_1^k\|^2\Big\} \\
       =& \,\big\{ w_1\in Range(Q_1) \mid Q_1(I+\sigma Q_1)w_1=\sigma Q_1(v_1^k+R_1^k) \big\}. 
     \end{aligned}
 \end{equation*} 
Moreover, \eqref{eq:3-1-a-2} is separable with respect to (w.r.t.) $w_{2,i}$, and therefore can be computed in parallel by solving $N_2$ independent subproblems, that is, for $i=1,\cdots,N_2$,
\begin{equation*}
     \begin{aligned}
       \tilde{w}_{2,i}^{k+1} =& \,\argmin_{w_{2,i}\in Range(Q_{2,i})} \Big\{ \frac{1}{2}\langle w_{2,i},Q_{2,i}w_{2,i}\rangle+\frac{\sigma}{2} \| v_{2,i}^k-Q_{2,i}w_{2,i}+ R_{2,i}^k\| ^2 \Big\} \\
       =& \,\big\{  w_{2,i}\in Range(Q_{2,i}) \mid Q_{2,i} (I+\sigma Q_{2,i})w_{2,i}=\sigma Q_{2,i}(v_{2,i}^k+R_{2,i}^k ) \big\}. 
     \end{aligned}
 \end{equation*} 
Similarly, \eqref{eq:3-1-a-3} is separable w.r.t. $w_{3, ij}$, and therefore can be computed in parallel by solving $N_2N_3$ independent subproblems, that is, for $i=1,\cdots,N_2,~j=1,\cdots,N_3$,
\begin{equation*}
\hspace{-1mm}
     \begin{aligned}
      \tilde{w}_{3,ij}^{k+1} =& \,\argmin_{w_{3,ij}\in Range(Q_{3,ij})}\Big\{ \frac{1}{2}\inp{w_{3,ij},Q_{3,ij}w_{3,ij}}+\frac{\sigma}{2} \| v_{3,ij}^k-Q_{3,ij}w_{3,ij}+ R_{3,ij}^k\| ^2 \Big\} \\
       =& \, \big\{ w_{3,ij}\in Range(Q_{3,ij}) \mid Q_{3,ij}(I+\sigma Q_{3,ij})w_{3,ij}=\sigma Q_{3,ij}(v_{3,ij}^k+R_{3,ij}^k)  \big\}. 
     \end{aligned}
\end{equation*} 
Note that the subsequent computation requires only $Q_t \tilde{w}_t^{k+1}$ and $\langle \tilde{w}_t^{k+1},Q_t \tilde{w}_t^{k+1}\rangle$, rather than $\tilde{w}_t^{k+1}$ itself. Moreover, let $\tilde{w}_t^{k+1}$ and $\bar{w}_t^{k+1}$ solve  $Q_t (I+\sigma Q_t )w_t=\sigma Q_t (v_t^k+R_t^k )$  and $ (I+\sigma Q_t )w_t=\sigma  (v_t^k+R_t^k ) $, respectively. By \cite[Proposition 1]{lam2021semi}, it holds that $Q_t \tilde{w}_t^{k+1} = Q_t \bar{w}_t^{k+1}$ and $\langle\tilde{w}_t^{k+1},Q_t \tilde{w}_t^{k+1}\rangle  = \langle \bar{w}_t^{k+1},Q_t \bar{w}_t^{k+1}\rangle $. Therefore, the required computation reduces to solving the following linear systems:
    \begin{align*}
        & (I+\sigma Q_1)w_1=\sigma (v_1^k+R_1^k),\\
        & (I+\sigma Q_{2,i})w_{2,i}=\sigma (v_{2,i}^k+R_{2,i}^k),~~i=1,\cdots,N_2,\\
        &(I+\sigma Q_{3,ij})w_{3,ij}=\sigma (v_{3,ij}^k+R_{3,ij}^k),~~i=1,\cdots,N_2,~~j=1,\cdots,N_3.
    \end{align*}
In Step 3.1c, the updates for $w_1^{k+1}$, $w_2^{k+1}$, and $w_3^{k+1}$ are computed analogously to those in Step 3.1a. Since these linear systems are independent across scenario nodes, Steps 3.1a and 3.1c can be executed in parallel with complexity $\mathcal{O}(d^2)$.

\vspace{1mm}
\noindent\textbf{Step 3.1b:} For \eqref{eq:3-1-b-1}, applying the Moreau decomposition \cite[Theorem 14.3(ii)]{bauschke2011convex} yields
\begin{equation*}
     \begin{aligned}
       v_1^{k+1} 
       =&\, \argmin_{v_1} \left\{ \theta_1^*(-v_1)+\frac{\sigma}{2}\| v_1-Q_1\tilde{w}_1^{k+1}+R_1^k\| ^2\right\} \\
       =& \, \frac{1}{\sigma}\text{Prox}_{\sigma\theta_1}(\sigma(-Q_1\tilde{w}_1^{k+1}+R_1^k))-(-Q_1\tilde{w}_1^{k+1}+R_1^k).
     \end{aligned}
\end{equation*} 
Moreover, \eqref{eq:3-1-b-2} is separable w.r.t. $v_{2,i}$, and therefore can be computed in parallel by solving $N_2$ independent subproblems, that is, for $i=1,\cdots,N_2$,
\begin{equation*}
     \begin{aligned}
       v_{2,i}^{k+1} =& \, \argmin_{v_{2,i}}\Big\{ \theta_{2,i}^*(-v_{2,i})+\frac{\sigma}{2}\| v_{2,i}-Q_{2,i}\tilde{w}_{2,i}^{k+1}+R_{2,i}^k\| ^2 \Big\} \\    
       =& \, \frac{1}{\sigma}\text{Prox}_{\sigma\theta_{2,i}} (\sigma (-Q_{2,i}\tilde{w}_{2,i}^{k+1}+R_{2,i}^k ) )- (-Q_{2,i}\tilde{w}_{2,i}^{k+1}+R_{2,i}^k ).
     \end{aligned}
 \end{equation*} 
Similarly, \eqref{eq:3-1-b-3} is separable w.r.t. $v_{3,ij}$, and therefore can be computed in parallel by solving $N_2N_3$ independent subproblems, that is, for $i=1,\cdots,N_2,~j=1,\cdots,N_3,$
\begin{equation*}
     \begin{aligned} 
      v_{3,ij}^{k+1}=& \, \argmin_{v_{3,ij}} \left\{ \theta_{3,ij}^*(-v_{3,ij})+\frac{\sigma}{2} \| v_{3,ij}-Q_{3,ij}\tilde{w}_{3,ij}^{k+1}+R_{3,ij}^k \| ^2 \right\} \\
       =&\, \frac{1}{\sigma}\text{Prox}_{\sigma\theta_{3,ij}} (\sigma (-Q_{3,ij}\tilde{w}_{3,ij}^{k+1}+R_{3,ij}^k ) )- (-Q_{3,ij}\tilde{w}_{3,ij}^{k+1}+R_{3,ij}^k). 
     \end{aligned}
 \end{equation*} 
 Therefore, the $v$-updates reduce to separable proximal mappings and can be carried out efficiently whenever the corresponding nonsmooth functions are proximal friendly.

\vspace{1mm}
\noindent\textbf{Steps 3.2a\,\&\,3.2c:} In Step 3.2a,
we define $D_1^k :=v_1^{k+1} - c_1(x_1^k) - Q_1w_1^{k+1} + \sigma^{-1}x_1^k$, $D_{2,i}^k :=v_{2,i}^{k+1} - c_{2,i}(x_{2,i}^k) - Q_{2,i}w_{2,i}^{k+1} + \sigma^{-1}x_{2,i}^k$, $D_{3,ij}^k := v_{3,ij}^{k+1} - c_{3,ij}(x_{3,ij}^k) - Q_{3,ij}w_{3,ij}^{k+1} + \sigma^{-1}x_{3,ij}^k$. Then, the update for $\tilde{y}_1^{k+1}$ in \eqref{eq:3-2-a-1} is 
\begin{equation*}
      \begin{aligned}
        \tilde{y}_1^{k+1}
       = &\, \argmin_{y_1} \Big\{ -\inp{b_1,y_1}+\frac{\sigma}{2}  \| A_1^*y_1+\frac{1}{N_2}\sum_{i=1}^{N_2}B_{2,i}^*y_{2,i}^k+z_1^k+D_1^k \| ^2  \Big\} \\
        =&\, (A_1A_1^* )^{-1} \Big(\sigma^{-1}b_1-A_1 \Big(\frac{1}{N_2}\sum_{i=1}^{N_2}B_{2,i}^*y_{2,i}^k+z_1^k+D_1^k \Big) \Big).
      \end{aligned}
\end{equation*} 
Note that the variables $y_{2,1},\dots,y_{2,N_2}$ are coupled in $\widehat{L}_{\sigma}^{x'_{1:3}}\big(v_{1:3},w_{1:3},y_{1:3},z_{1:3};x_{1:3}\big)$, which prevents the corresponding minimization problem from being separated into $N_2$ independent subproblems. Nevertheless, these variables can still be computed efficiently as follows: first, \eqref{eq:3-2-a-2} is given by
\begin{equation*}
      \begin{aligned}
      \!\!\min_{y_{2,1},\cdots,y_{2,N_2}}\!\!\left\lbrace \begin{aligned}
          & \Phi_2^k\big(y_{2,1},\cdots,y_{2,N_2}\big) 
          := 
          \frac{\sigma}{2}\|A_1^*\tilde{y}_1^{k+1}+\frac{1}{N_2}\sum_{i=1}^{N_2}B_{2,i}^*y_{2,i}+z_1^k+D_1^k\|^2  \\ 
          &\!\!+ \frac{1}{N_2}\!\sum_{i=1}^{N_2}\Big(\!-\inp{b_{2,i},y_{2,i}}  \!+\! \frac{\sigma}{2}\|A_{2,i}^*y_{2,i}
          + \frac{1}{N_3}\sum_{j=1}^{N_3}B_{3,ij}^*y_{3,ij}^k\!+z_{2,i}^k+D_{2,i}^k\|^2 \Big)
      \end{aligned}\!\!\right\rbrace
      \end{aligned}.
  \end{equation*} 
The gradient of $\Phi_2^k$ w.r.t. $y_{2,i}$ is
\begin{equation*}
      \begin{aligned}\nabla_{y_{2,i}}\Phi_2^k(y_{2,1},\cdots\!,y_{2,N_2})
      =&  \, \frac{\sigma}{N_2}B_{2,i} \Big(\frac{1}{N_2}\sum_{\mu=1}^{N_2}B_{2,\mu}^*y_{2,\mu}+A_1^*\tilde{y}_1^{k+1}+z_1^k+D_1^k \Big)-\frac{1}{N_2}b_{2,i}\\
        & \,       +\frac{\sigma}{N_2}\Big(A_{2,i} \Big(A_{2,i}^*y_{2,i}+\frac{1}{N_3}\sum_{j=1}^{N_3}B_{3,ij}^*y_{3,ij}^k+z_{2,i}^k+D_{2,i}^k \Big) \Big)\\
        =& \, \frac{\sigma}{N_2}B_{2,i} \Big(\frac{1}{N_2}\sum_{\mu=1}^{N_2}B_{2,\mu}^*y_{2,\mu} \Big)+\frac{\sigma}{N_2}A_{2,i}A_{2,i}^*y_{2,i}+\frac{\sigma}{N_2}\phi_{2,i},
      \end{aligned}
  \end{equation*} 
where $\phi_{2,i}:=B_{2,i} (A_1^*\tilde{y}_1^{k+1}+z_1^k+D_1^k )-\sigma^{-1}b_{2,i} +A_{2,i} (\frac{1}{N_3}\sum_{j=1}^{N_3}B_{3,ij}^*y_{3,ij}^k+z_{2,i}^k+D_{2,i}^k )$. Then, applying Lemma \ref{lem:key} yields 
\begin{equation*}
\tilde{y}_{2,i}^{k+1} =
-\big(A_{2,i}A_{2,i}^*\big)^{-1}\Big(\frac{1}{N_2}B_{2,i}\tilde{b}_2^k+\phi_{2,i}\Big),
\end{equation*}
where $\tilde{b}_2^k\!:=\!-\big(I+\frac{1}{N_2}\sum_{\mu=1}^{N_2}B_{2,\mu}^*(A_{2,\mu}A_{2,\mu}^* )^{-1}B_{2,\mu}\big)^{-1}\!\big(\sum_{\mu=1}^{N_2}B_{2,\mu}^*(A_{2,\mu}A_{2,\mu}^* )^{-1}\!\phi_{2,\mu}\big)$. Similarly, for each $i=1,\cdots,N_2$, the variables $y_{3,i1},y_{3,i2},\cdots,y_{3,iN_3}$ can be obtained by solving
\begin{equation*}
\min_{y_{3,i1},\cdots,y_{3,iN_3}}\left\lbrace 
\begin{aligned}
& \frac{\sigma}{2}\|A_{2,i}^*\tilde{y}_{2,i}^{k+1}+\frac{1}{N_3}\sum_{j=1}^{N_3}B_{3,ij}^*y_{3,ij}+z_{2,i}^k+D_{2,i}^k\|^2 \\[3pt]
& +\frac{1}{N_3}\sum_{j=1}^{N_3}\left( -\inp{b_{3,ij},y_{3,ij}}+\frac{\sigma}{2}\|A_{3,ij}^*y_{3,ij}+z_{3,ij}^k+D_{3,ij}^k\|^2 \right)
\end{aligned}\right\rbrace.
\end{equation*}
By applying the same argument as above, for \eqref{eq:3-2-a-3}, we obtain
\begin{equation*}
\tilde{y}_{3,ij}^{k+1}
=-\big(A_{3,ij}A_{3,ij}^*\big)^{-1}\big(\frac{1}{N_3}B_{3,ij}\tilde{b}_{3,i}^k+\phi_{3,ij}\big),
\end{equation*}
where $\phi_{3,ij}:=B_{3,ij}\big(A_{2,i}^*\tilde{y}_{2,i}^{k+1}+z_{2,i}^k+D_{2,i}^k\big)-\sigma^{-1}b_{3,ij}+A_{3,ij}\big(z_{3,ij}^k+D_{3,ij}^k\big)$ and $\tilde{b}_{3,i}^k := - \big(I+\frac{1}{N_3}\sum_{\mu=1}^{N_3}B_{3,i\mu}^*(A_{3,i\mu}A_{3,i\mu}^*)^{-1}B_{3,i\mu}\big)^{-1}\big(\sum_{\mu=1}^{N_3}B_{3,i\mu}^*(A_{3,i\mu}A_{3,i\mu}^*)^{-1}\phi_{3,i\mu}\big)$.

The updates of $y_1^{k+1}$, $y_2^{k+1}$, and $y_3^{k+1}$ in Step~3.2c are performed analogously to those in Step~3.2a. The $y$-subproblem contains the main coupling induced by the interstage constraints between adjacent stages, but this coupling follows the block structure of the scenario tree. By Lemma~\ref{lem:key}, the coupled $y$-subproblem can be reduced to smaller linear systems organized by stages and scenarios, so that the update can be implemented without solving a large system over all stages and scenarios. In particular, the third-stage update decomposes into \(N_3\) independent subproblems per second-stage node, with the resulting \(N_2\) groups processed in parallel. Therefore, the parallel computational costs of Step~3.2a and Step~3.2c are $\mathcal{O}\big((1+N_2+N_3)\ell\max \{\ell,d\}\big)$ and $\mathcal{O}\big((N_2+N_3)\ell d+(1+N_2+N_3)\ell\max \{\ell,d\}\big)$, respectively.

\vspace{1mm}
\noindent\textbf{Step 3.2b:} For \eqref{eq:3-2-b-1}, we have 
\begin{equation*}
      \begin{aligned}
        z_1^{k+1} =&\,\argmin_{z_1} \Big\{ \delta_{\mathcal{K}_1}^*(-z_1)
        + \frac{\sigma}{2} \| A_1^*\tilde{y}_1^{k+1}
        + \frac{1}{N_2}\sum_{i=1}^{N_2}B_{2,i}^*\tilde{y}_{2,i}^{k+1}+z_1+D_1^k \|^2 \Big\} \\
        =&\,\frac{1}{\sigma}\Pi_{\mathcal{K}_1}\Big(\sigma \Big(A_1^*\tilde{y}_1^{k+1}+\frac{1}{N_2}\sum_{i=1}^{N_2}B_{2,i}^*\tilde{y}_{2,i}^{k+1}+D_1^k\Big)\Big)\\
        &\quad -\Big(A_1^*\tilde{y}_1^{k+1}+\frac{1}{N_2}\sum_{i=1}^{N_2}B_{2,i}^*\tilde{y}_{2,i}^{k+1}+D_1^k \Big).
      \end{aligned}
\end{equation*} 
Moreover, \eqref{eq:3-2-b-2} is separable w.r.t. $z_{2,i}$, and therefore can be computed in parallel by solving $N_2$ independent subproblems, that is, for $ i=1,2,\cdots,N_2,$
\begin{equation*}
     \begin{aligned}
       z_{2,i}^{k+1}
       =& \, \argmin_{z_{2,i}} \Big\{ 
       \delta_{\mathcal{K}_{2,i}}^*(-z_{2,i})
           +\frac{\sigma}{2} \| A_{2,i}^*\tilde{y}_{2,i}^{k+1}+\frac{1}{N_3}\sum_{j=1}^{N_3}B_{3,ij}^*\tilde{y}_{3,ij}^{k+1}+z_{2,i}+D_{2,i}^k \| ^2
        \Big\} \\ 
       =&\, \frac{1}{\sigma}\Pi_{\mathcal{K}_{2,i}} \Big(\sigma \Big(A_{2,i}^*\tilde{y}_{2,i}^{k+1}+\frac{1}{N_3}\sum_{j=1}^{N_3}B_{3,ij}^*\tilde{y}_{3,ij}^{k+1}+D_{2,i}^k \Big) \Big) \\
       &\quad - \Big(A_{2,i}^*\tilde{y}_{2,i}^{k+1}+\frac{1}{N_3}\sum_{j=1}^{N_3}B_{3,ij}^*\tilde{y}_{3,ij}^{k+1}+D_{2,i}^k \Big).
     \end{aligned}
 \end{equation*} 
Finally, \eqref{eq:3-2-b-3} is separable w.r.t. $z_{3,ij}$, and therefore can be computed in parallel by solving $N_2N_3$ independent subproblems, that is, for $i=1,2,\cdots,N_2,~j=1,2,\cdots,N_3,$
\begin{equation*}
     \begin{aligned}
       z_{3,ij}^{k+1}
        =&  \, \argmin_{z_{3,ij}}  \left\{ 
        \delta_{\mathcal{K}_{3,ij}}^*(-z_{3,ij})+\frac{\sigma}{2} \| A_{3,ij}^*\tilde{y}_{3,ij}^{k+1}+z_{3,ij}+D_{3,ij}^k \|^2  \right\} \\
        =& \, \frac{1}{\sigma}\Pi_{\mathcal{K}_{3,ij}}\big(\sigma  \big(A_{3,ij}^*\tilde{y}_{3,ij}^{k+1}+D_{3,ij}^k\big) \big) 
        - \left(A_{3,ij}^*\tilde{y}_{3,ij}^{k+1}+D_{3,ij}^k \right).
     \end{aligned}
\end{equation*} 
Therefore, the $z$-updates reduce to separable projection-type computations and can be carried out efficiently whenever the projections onto the corresponding sets are easy to compute.

In summary, the $w$-, $v$-, and $z$-updates in Algorithm~\ref{alg:tsp-md} have separable or structured forms, while the coupled $y$-updates can be handled by Lemma~\ref{lem:key}. Hence, the proposed SL-MDD method preserves the stage-wise and scenario-wise structure of the TSP problem and is well suited for parallel implementation.
\subsection{Extension to the general multi-stage case}

Based on the developments in the preceding sections, we are now ready to present an algorithm for solving the MSP problem \eqref{MSP}. Following the framework in Section \ref{sec:alg-des}, the corresponding minorized problem takes the form
\begin{equation}
\begin{aligned}\label{MSP2}
\min\limits_{x_{t}, s_{t}, u_{t}} 
\quad &\sum_{t=1}^T\left(\theta_t(u_t)+ \delta_{\mathcal{K}_t}(s_t)+\tilde{h}_t(x_t;x_t')\right) \\
\mbox{s.t.}\quad \,
&A_1x_1= b_1, \\
&B_tx_{t-1} + A_tx_t= b_t, ~~t=2,\cdots,T,\\ 
&u_t=x_t,~~s_t=x_t,~~t=1,\cdots,T.
\end{aligned}
\end{equation}
Here, the functions $\theta_t$ and $\tilde{h}_t$ are defined analogously to their three-stage counterparts in \eqref{theta&h} and \eqref{tilde_h}, with the definitions extended naturally to the general $T$-stage setting. The restricted Wolfe dual problem of \eqref{MSP2} is given by
\begin{equation}\label{MSP3}
\begin{aligned}
\max\limits_{v_t,w_t,y_t,z_t} 
& \sum_{t=1}^{T}\Big(-\frac{1}{2}\inp{w_t,Q_tw_t}-\theta_t^*(-v_t)-\delta_{\mathcal{K}_t}^*(-z_t)+\inp{b_t,y_t} \\
&\qquad + h_t(x_t') + \frac{1}{2}\inp{x_t',Q_tx_t'}-\inp{\nabla h_t(x_t'),x_t'} \Big) \\[2pt]
\mbox{s.t.}\quad
&c_t(x'_t)-A_t^{*}y_t-B_{t+1}^{*}y_{t+1}-v_t-z_t=-Q_tw_t, \\[2pt]
&w_t \in \mathcal{W}_t,~~t=1,\cdots,T,
\end{aligned}
\end{equation}
where we set $B^*_{T+1}=0$ and $y_{T+1}=0$ for notational consistency. The associated augmented Lagrangian function is
\begin{equation*}
\begin{aligned}
& \widehat{L}_{\sigma}^{x'_{1:T}}\big(v_{1:T},w_{1:T},y_{1:T},z_{1:T};x_{1:T}\big)
:= \sum_{t=1}^{T} \Big(\theta_t^*(-v_t)+\delta_{\mathcal{K}_t}^*(-z_t)+\frac{1}{2}\inp{w_t,Q_tw_t} \\
&\textstyle 
\quad +\delta_{\mathcal{W}_t}(w_t)-\inp{b_t,y_t}
+\frac{\sigma}{2}\|A_t^{*}y_t+B_{t+1}^{*}y_{t+1}+v_t+z_t-c_t(x'_t)-Q_tw_t+\sigma^{-1}x_t\|^2 \\
&
\quad -h_t(x_t')-\frac{1}{2}\inp{x_t',Q_tx_t'}+\inp{\nabla h_t(x_t'),x_t'}
-\frac{1}{2\sigma} \|x_t\| ^2 \Big).
\end{aligned}
\end{equation*}
The resulting SL-MDD method for problem \eqref{MSP} is outlined in Algorithm~\ref{alg:msp}.

\begin{algorithm}
\small
\caption{A SL-MDD method for the MSP problem \eqref{MSP}}
\label{alg:msp}
\begin{algorithmic}
\STATE{Let $\tau \in  (0,\frac{1+\sqrt{5-4\alpha}}{2} ]$ be the step-length with $\alpha \in (0,1)$, and $\{\varepsilon_k\}_{k\geq0}$ be a summable sequence of nonnegative numbers. Choose an initial point $\varpi^0=(v^0_{1:T},w^0_{1:T},y^0_{1:T},z^0_{1:T},x^0_{1:T})$. Set $k = 0$.}
\WHILE{the termination criterion is not met,}
\STATE{\textbf{Step 1:} Construct the minorized problem \eqref{MSP2} at the point $x_{1:T}^k$.}
\STATE{\textbf{Step 2:} Derive the restricted Wolfe dual problem} \eqref{MSP3}.
\STATE{\textbf{Step 3:} Call one sGS-iADMM iteration to solve \eqref{MSP3}:}
\STATE{\hspace{\algorithmicindent} \textbf{Step 3.1a:}
For $t=1,\ldots,T,$ compute}
\begin{equation*}
\begin{aligned}
\textstyle  \tilde{w}_t^{k+1} \approx \argmin_{w_t} \lbrace
\widehat{L}_{\sigma}^{x^k_{1:T}} ( v_{1:T}^k, ( \tilde{w}_{\leq t-1}^{k+1}, w_t, w_{\geq t+1}^k ), y_{1:T}^k, z_{1:T}^k;x_{1:T}^k )  \rbrace.
\end{aligned}
\end{equation*}
\vspace{-1.5ex}
\STATE{\hspace{\algorithmicindent} \textbf{Step 3.1b:}
For $t=1,\ldots,T,$ compute}
\begin{equation*}
\begin{aligned}
  \textstyle v_t^{k+1} \approx \argmin_{v_t}  \lbrace \widehat{L}_{\sigma}^{x^k_{1:T}} (  ( v_{\leq t-1}^{k+1}, v_t, v_{\geq t+1}^k  ), \tilde{w}_{1:T}^{k+1}, y_{1:T}^k, z_{1:T}^k; x_{1:T}^k  )  \rbrace.
\end{aligned}
\end{equation*}
\vspace{-1.5ex}
\STATE{\hspace{\algorithmicindent} \textbf{Step 3.1c:}
For $t=T,\ldots,1,$ compute}
\begin{equation*}
\begin{aligned}
  \textstyle w_t^{k+1} \approx  \argmin_{w_t} \lbrace \widehat{L}_{\sigma}^{x^k_{1:T}} ( v_{1:T}^{k+1},  ( \tilde{w}_{\leq t-1}^{k+1}, w_t, w_{\geq t+1}^{k+1}  ), y_{1:T}^k, z_{1:T}^k; x_{1:T}^k  ) \rbrace.
\end{aligned}
\end{equation*}
\vspace{-1.5ex}
\STATE{\hspace{\algorithmicindent} \textbf{Step 3.2a:}
For $t=1,\ldots,T,$ compute}
\begin{equation*}
\begin{aligned}
   \textstyle\tilde{y}_t^{k+1} \approx \argmin_{y_t} \{ \widehat{L}_{\sigma}^{x^k_{1:T}} ( v_{1:T}^{k+1}, w_{1:T}^{k+1}, ( \tilde{y}_{\leq t-1}^{k+1}, y_t, y_{\geq t+1}^k  ), z_{1:T}^k; x_{1:T}^k  ) \}.
\end{aligned}
\end{equation*}
\vspace{-1.5ex}
\STATE{\hspace{\algorithmicindent} \textbf{Step 3.2b:}
For $t=1,\ldots,T,$ compute}
\begin{equation*}
\begin{aligned}
  \textstyle z_t^{k+1} \approx \argmin_{z_t}  \lbrace \widehat{L}_{\sigma}^{x^k_{1:T}} ( v_{1:T}^{k+1}, w_{1:T}^{k+1}, \tilde{y}_{1:T}^{k+1}, ( z_{\leq t-1}^{k+1}, z_t, z_{\geq t+1}^k ); x_{1:T}^k  )  \rbrace.
\end{aligned}
\end{equation*}
\vspace{-1.5ex}
\STATE{\hspace{\algorithmicindent} \textbf{Step 3.2c:}
For $t=T,\ldots,1,$ compute}
\begin{equation*}
\begin{aligned}
 \textstyle  y_t^{k+1} \approx \argmin_{y_t}  \{ \widehat{L}_{\sigma}^{x^k_{1:T}} ( v_{1:T}^{k+1}, w_{1:T}^{k+1}, ( \tilde{y}_{\leq t-1}^{k+1}, y_t, y_{\geq t+1}^{k+1}  ), z_{1:T}^{k+1}; x_{1:T}^k  )  \}.
\end{aligned}
\end{equation*}
\vspace{-1.5ex}
\STATE{\hspace{\algorithmicindent} \textbf{Step 3.3:} For $t=1,\ldots,T,$ compute}
\begin{equation*}
\begin{aligned}
   ~~~&x_t^{k+1} = x_t^k + \tau\sigma(A_t^*y_t^{k+1} + B_{t+1}^*y_{t+1}^{k+1} + z_t^{k+1} + v_t^{k+1} - c_t(x_t^{k}) - Q_tw_t^{k+1}). \\
  \end{aligned}   
  \end{equation*}
\ENDWHILE
\end{algorithmic}
\end{algorithm}


\section{Convergence analysis}\label{sec:convergence}
In this section, we establish the global convergence of the proposed SL-MDD method in Algorithm~\ref{alg:tsp-md} for the TSP problem \eqref{P1}. For convergence analysis, we introduce an equivalent SL-MD-isPADMM, and show that it is equivalent to the proposed SL-MDD method after a suitable grouping of the dual variables. Consequently, the convergence of the SL-MDD method follows from that of the SL-MD-isPADMM. 
This equivalent method is used only as an analytical device, while the implementable algorithmic framework remains the SL-MDD method developed in Section~\ref{sec:alg}. We present the detailed convergence analysis for the three-stage case and then extend the same proof pattern to the general multi-stage setting.

We first introduce the equivalent SL-MD-isPADMM for the TSP problem \eqref{P1}. 
Let $m:=(v_{1:3};w_{1:3})$, $n:=(z_{1:3};y_{1:3})$, $x':=(x'_1;x'_2;x'_3)$, and $x:=(x_1;x_2;x_3)$. 
Then the restricted Wolfe dual problem \eqref{D1} can be equivalently rewritten as
\begin{equation}\label{D2}
\begin{aligned}
\min_{m,n}\quad &p(m)+f(m; x')+q(n)+g(n) \\
\mbox{s.t.}\quad
&\mathcal{F}^*m+\mathcal{G}^*n=c(x'),
\end{aligned}
\end{equation}
where 
\begin{equation*}
\begin{aligned}
&p(m):=\sum_{t=1}^3\theta_t^*(-v_t), 
~~q(n):=\sum_{t=1}^3\delta_{\mathcal{K}_t}^*(-z_t),
~~g(n):=-\sum_{t=1}^3\langle b_t,y_t\rangle,\\[1pt]
&f(m;x'):=\sum_{t=1}^3\Big(\frac{1}{2}\langle w_t,Q_tw_t \rangle+\delta_{\mathcal{W}_t}(w_t)-h_t(x_t')-\frac{1}{2}\langle x_t',Q_tx_t'\rangle+\langle \nabla h_t(x_t'),x_t'\rangle\Big),\\
&\setlength{\arraycolsep}{2pt}
\mathcal{F}^*\!:=\!\begin{pmatrix}
I & 0 & 0 & -Q_1 & 0 & 0\\
0 & I & 0 & 0 & -Q_2 & 0\\
0 & 0 & I & 0 & 0 & -Q_3
\end{pmatrix}\!,~
\mathcal{G}^*\!:=\!\begin{pmatrix}
I & 0 & 0 & A_1^* & B_2^* & 0\\
0 & I & 0 & 0 & A_2^* & B_3^*\\
0 & 0 & I & 0 & 0 & A_3^*
\end{pmatrix}\!,~
c(x')\!:=\!\begin{pmatrix}
c_1(x'_1)\\
c_2(x'_2)\\
c_3(x'_3)
\end{pmatrix}\!.
\end{aligned}
\end{equation*}
The augmented Lagrangian function associated with problem \eqref{D2} is
\begin{equation*}
\begin{aligned}
\widehat{L}_{\sigma}^{x'}(m,n;x):= \,&  p(m)+f(m; x')+q(n)+g(n)-\frac{1}{2\sigma}\|x\|^2\\
& +\frac{\sigma}{2}\|\mathcal{F}^*m+\mathcal{G}^*n-c(x')+\sigma^{-1}x\|^2.
\end{aligned}
\end{equation*}
Furthermore, let $\mathcal{S}$ and $\mathcal{T}$ be positive semidefinite linear operators defined by
\begin{align*}
	\mathcal{S}:=sGS(\widetilde{\mathcal{M}}),\quad \mathcal{T}:=sGS(\widetilde{\mathcal{N}}),
\end{align*}
where $\widetilde{\mathcal{M}}:=\widetilde{Q}+\sigma\mathcal{F}\mathcal{F}^*$ and $\widetilde{\mathcal{N}}:=\sigma\mathcal{G}\mathcal{G}^*$, with $Q:= \text{diag}(Q_1,Q_2,Q_3)$ and $\widetilde{Q}:= \text{diag}(0,Q)$. Define 
$\mathcal{M}:=\widetilde{\mathcal{M}}+\mathcal{S},
    ~\mathcal{N}:=\widetilde{\mathcal{N}}+\mathcal{T}.$
By the sGS decomposition theorem \cite[Theorem 1]{li2019block}, it follows that $\mathcal{M}\succ0$ and $\mathcal{N}\succ0$. 

We are now ready to present the equivalent SL-MD-isPADMM for the TSP problem \eqref{P1} in Algorithm \ref{alg:tsp-p}.
\begin{algorithm}
\small
\caption{An equivalent SL-MD-isPADMM for the TSP problem \eqref{P1}}
\label{alg:tsp-p}
\begin{algorithmic}
\STATE{Let $\tau \in (0,\frac{1+\sqrt{5-4\alpha}}{2}]$ be the step-length with $\alpha \in (0,1)$, and $\{\varepsilon_k\}_{k\geq0}$ be a summable sequence of nonnegative numbers. Set $k = 0$.}
\WHILE{the termination criterion is not met,}
\STATE{\textbf{Step 1:} Construct the minorized problem \eqref{P3} at the point $x^k$.}
\STATE{\textbf{Step 2:} Derive the restricted Wolfe dual problem \eqref{D2}.}
\STATE{\textbf{Step 3:} Call one isPADMM iteration to solve \eqref{D2}:}
\STATE{\hspace{\algorithmicindent} \textbf{Step 3.1:} Compute  
\vspace{-1ex}
\begin{equation}\label{m}
\begin{aligned}
m^{k+1} \approx \bar{m}^{k+1} := \,\argmin_{m}\Big\{\widehat{L}_{\sigma}^{x^k}(m,n^k;x^k)+\frac{1}{2}\|m-m^k\|_{\mathcal{S}}^2\Big\},
\end{aligned}   
\vspace{-1ex}
\end{equation} 
\hspace{\algorithmicindent} such that there exists a vector $d_m^k$ satisfying $\|\mathcal{M}^{-\frac{1}{2}}d_m^k\| \leq\varepsilon_k$ and  
\begin{equation*}
\begin{aligned}
d_m^k \in \partial_m\widehat{L}_{\sigma}^{x^k}(m^{k+1},n^k;x^k)+\mathcal{S}(m^{k+1}-m^k).
\end{aligned}  
\end{equation*}}

\STATE{\hspace{\algorithmicindent} \textbf{Step 3.2:} Compute   
\begin{equation}
\begin{aligned}\label{n}
 n^{k+1} \approx \bar{n}^{k+1} := \,\argmin_{n}\Big\{\widehat{L}_{\sigma}^{x^k}(\bar{m}^{k+1},n;x^k)+\frac{1}{2}\|n-n^k\|_{\mathcal{T}}^2\Big\},
\end{aligned} 
 \vspace{-1ex}
\end{equation}
\hspace{\algorithmicindent} such that there exists a vector $d_n^k$ satisfying $\|\mathcal{N}^{-\frac{1}{2}}d_n^k\| \leq\varepsilon_k$ and 
\begin{equation*}
\begin{aligned}
 d_n^k \in \partial_n\widehat{L}_{\sigma}^{x^k}(m^{k+1},n^{k+1};x^k)+\mathcal{T}(n^{k+1}-n^k).
\end{aligned} 
\end{equation*} }
  
\STATE{\hspace{\algorithmicindent} \textbf{Step 3.3:} Compute~
$x^{k+1} = x^k + \tau\sigma\left(\mathcal{F}^*m^{k+1} + \mathcal{G}^*n^{k+1} - c(x^k)\right)$. 
}
\ENDWHILE
\end{algorithmic}
\end{algorithm}


\begin{theorem}\label{3.3}
The equivalent SL-MD-isPADMM in Algorithm~\ref{alg:tsp-p} is equivalent to the proposed SL-MDD method in Algorithm~\ref{alg:tsp-md}.
\end{theorem}
\begin{proof}
The conclusion follows from the sGS decomposition theorem 
\cite[Theorem~1]{li2019block}. 
The grouping $m=(v_{1:3},w_{1:3})$ and $n=(z_{1:3},y_{1:3})$ is induced by the dual-side coupling structure of the restricted Wolfe dual. 
With the choices $\mathcal S=\operatorname{sGS}(\widetilde{\mathcal M})$ and $\mathcal T=\operatorname{sGS}(\widetilde{\mathcal N})$, the grouped $m$-update in Algorithm~\ref{alg:tsp-p} is equivalent to the $(w,v,w)$-sweep in Algorithm~\ref{alg:tsp-md}, while the grouped $n$-update is equivalent to the $(y,z,y)$-sweep. 
The blockwise inexactness conditions in Algorithm~\ref{alg:tsp-md} are collected into the inexactness criteria in Algorithm~\ref{alg:tsp-p}. 
Therefore, the two algorithms generate identical iterates, and the equivalence follows.
\end{proof}

By Theorem \ref{3.3}, it suffices to analyze the convergence of Algorithm \ref{alg:tsp-p}; the convergence of Algorithm \ref{alg:tsp-md} then follows immediately.
We next introduce the essential notations and definitions for the convergence analysis.
Define $\Delta_m^k:= m^k-m^{k+1}$, $\Delta_n^k:= n^k-n^{k+1}$, $r^k:=\mathcal{F}^*m^k+\mathcal{G}^*n^k-c (x^{k-1} )$, $\tilde{x}^{k+1}:=x^k+\sigma r^{k+1}$, $ R (m^{k+1},n^k ):=\mathcal{F}^*m^{k+1}+\mathcal{G}^*n^k-c (x^{k} )$.
Moreover, let $ \gamma:=\max \{1-\tau, 1-\tau^{-1}\} $, $\beta:=
\begin{cases}
1-\alpha\tau, & \tau\in(0,1],\\[0.3em]
1-\tau+(1-\alpha)\tau^{-1}, & \tau\in(1,\infty),
\end{cases}
$ with $\alpha \in (0,1)$. Given that $\tau \in (0,\frac{1+\sqrt{5-4\alpha}}{2} ]$, we have $0\leq \gamma<1$ and $\beta>0$. We further define $ \mathcal{A}:=\mathcal{T}+\min  \{\tau, 1+\tau-\tau^{2} \}(1-\alpha) \sigma \mathcal{GG}^{*}.$ For $\varpi:=(m,n,x)$, define
\begin{equation*}
    \begin{aligned}
       \phi_{k}(\varpi):=&~\textstyle\frac{1}{\tau \sigma} \|x-x^{k} \|^{2}+ \|m-m^{k} \|_{\widetilde{Q}+\mathcal{S}}^{2}+ \|n-n^{k} \|_{\mathcal{N}}^{2}+(1-\alpha) \gamma\sigma \|r^{k} \|^{2}\\
       &+(1-\alpha) \|\Delta_{n}^{k-1} \|_{\mathcal{T}}^{2},\\
		  \bar{\phi}_{k}(\varpi):=&~\textstyle \frac{1}{\tau \sigma} \|x-\bar{x}^{k} \|^{2}+ \|m-\bar{m}^{k} \|_{\widetilde{Q}+\mathcal{S}}^{2}+ \|n-\bar{n}^{k} \|_{\mathcal{N}}^{2}+(1-\alpha) \gamma\sigma \|\bar{r}^{k} \|^{2}\\
		 &+(1-\alpha) \|\bar{n}^{k}-n^{k-1} \|_{\mathcal{T}}^{2}.
	\end{aligned}
\end{equation*}
Assume that both the primal problem \eqref{P1} and its dual have nonempty interior feasible regions. Then optimal solutions exist, and they satisfy the following Karush--Kuhn--Tucker (KKT) conditions:
 \begin{equation}\label{KKT}
 \begin{cases}
 \begin{aligned}
   &A_1x_1 = b_1,~B_2x_1 + A_2x_2= b_2,~B_3x_2 + A_3x_3 = b_3,\\
   &\nabla h_1(x_1)-A_1^{*}y_1-B_2^{*}y_2-v_1-z_1=0, \\
   &\nabla h_2(x_2)-A_2^{*}y_2-B_3^{*}y_3-v_2-z_2=0,\\ 
   &\nabla h_3(x_3)-A_3^{*}y_3-v_3-z_3=0,\\ 
   &0 \in \partial\theta_t^*(-v_t)-x_t \Leftrightarrow -v_t=\text{Prox}_{\theta_t^*}(x_t-v_t) \Leftrightarrow
   x_t=\text{Prox}_{\theta_t}(x_t-v_t),\\
   &0 \in \partial\delta_{\mathcal{K}_t}^*(-z_t)-x_t \Leftrightarrow -z_t=\text{Prox}_{\delta_{\mathcal{K}_t}^*}(x_t-z_t) \Leftrightarrow x_t=\Pi_{\mathcal{K}_t}(x_t-z_t),\\
 \end{aligned}
\end{cases}
\end{equation}
where $\Pi_{\mathcal{K}_t}$ denote the projection on $\mathcal{K}_t$, $t=1, 2, 3$. Define $\hat{b}:=(0;b)$ and  $\nabla h(x):=(\nabla h_1(x_1)$; $\nabla h_2(x_2);\nabla h_3(x_3))$. Moreover, by \eqref{xt}, we have $Q_tw_t=Q_tx_t$ for $t=1,2,3$. Since $\partial_m f(m;x')$ is independent of $x'$, we write $\partial_m f(m)$. Therefore, the KKT system \eqref{KKT} can be written compactly as
\begin{equation}\label{KKT2}
\begin{cases}
    \begin{aligned}
   &0 \in \partial p(m)+ \partial_m f(m) +\mathcal{F}x,\\
   &0 \in \partial q(n)+ \nabla g(n) +\mathcal{G}x,\\
   &\mathcal{F}^*m+\mathcal{G}^*n=\nabla h(x)-Qx.
 \end{aligned}
\end{cases}
\end{equation}

We now turn to the convergence analysis of the SL-MD-isPADMM in Algorithm \ref{alg:tsp-p}. We begin with a basic lemma for the subsequent proposition.

\begin{lemma}\label{lem:3.1}
Let  $\{\varpi^k\}$ be the sequence generated by the SL-MD-isPADMM in Algorithm \ref{alg:tsp-p}. For any $\alpha\in(0,1)$ and $k\ge 1$, the following inequality holds:
\begin{equation}
\begin{aligned}\label{5.1}
&~ (1-\tau)\sigma \| r^{k+1} \| ^2+(1-\alpha)\sigma \| R(m^{k+1},n^k) \| ^{2}+2(1-\alpha)\langle d_n^{k-1}-d_n^k,\Delta_n^{k}\rangle\\
\geq & ~ (1-\alpha)\gamma \sigma (  \| r^{k+1} \|^{2} - \| r^{k} \| ^{2} ) +(1-\alpha) (  \| \Delta_{n}^{k} \| _{\mathcal{T}}^{2}- \| \Delta_{n}^{k-1} \| _{\mathcal{T}}^{2} )\\
 & ~ +(1-\alpha)\min \lbrace \tau,1+\tau-\tau^2 \rbrace \sigma \| \mathcal{G^*}\Delta_{n}^{k} \| ^{2}+\beta\sigma \| r^{k+1} \|^{2}.
\end{aligned}
\end{equation}
\end{lemma} 
\begin{proof}
First, note that $\sigma r^{k+1}=\tilde{x}^{k+1}-\tilde{x}^k+(1-\tau) \sigma r^k$. Using the identity $ R (m^{k+1}, n^k) =r^{k+1}+\mathcal{G}^* \Delta_n^k$, we obtain
\begin{equation}
\begin{aligned}\label{5.2}
& ~   (1-\tau) \sigma \|r^{k+1} \|^2+(1-\alpha) \sigma \|R(m^{k+1}, n^k) \|^2 \\
= &~   (1-\tau) \sigma \|r^{k+1} \|^2+(1-\alpha) \sigma \|r^{k+1} \|^2+(1-\alpha) \sigma \| \mathcal{G}^* \Delta_n^k \|^2+2(1-\alpha) \langle\sigma r^{k+1}, \mathcal{G}^* \Delta_n^k \rangle \\
= &~   (2-\tau-\alpha) \sigma \|r^{k+1} \|^2+(1-\alpha) \sigma  \| \mathcal{G}^*\Delta_n^k  \|^2 + 2(1-\tau)(1-\alpha) \sigma\langle r^k,\mathcal{G}^* \Delta_n^k\rangle \\
&~  +2(1-\alpha)\langle \tilde{x}^{k+1}-\tilde{x}^k, \mathcal{G}^* \Delta _n^k\rangle.
\end{aligned}
\end{equation}
Next, by Step 3.2 of Algorithm \ref{alg:tsp-p}, for any $k \geq 0$,
\begin{equation}\label{5.3}
  d_n^k+\hat{b}-\mathcal{G} \tilde{x}^{k+1}+\mathcal{T} \Delta_n^k \in \partial q (n^{k+1} ).\end{equation}
Since $\partial q$ is maximally monotone, for any $k \geq 1$, we have
$$
  \langle d_n^k-d_n^{k-1}-\mathcal{G} (\tilde{x}^{k+1}-\tilde{x}^k ),-\Delta_n^k \rangle+ \langle \mathcal{T} (\Delta_n^k-\Delta_n^{k-1} ),-\Delta_n^k \rangle \geq 0.
$$
Hence, 
\begin{equation}
\begin{aligned}\label{5.4}
 \langle \tilde{x}^{k+1}-\tilde{x}^k, \mathcal{G}^* \Delta_n^k \rangle- \langle d_n^k-d_n^{k-1}, \Delta_n^k \rangle  \geq &~ \|\Delta_n^k \|_{\mathcal{T}}^2- \langle\Delta_n^{k-1}, \Delta_n^k \rangle_{\mathcal{T}}  \\
\geq &~  \frac{1}{2} \|\Delta_n^k \|_{\mathcal{T}}^2-\frac{1}{2} \|\Delta_n^{k-1} \|_{\mathcal{T}}^2.
\end{aligned}
\end{equation}
On the other hand, by the Cauchy--Schwarz inequality,
$$
 \tau \| \mathcal{G}^* \Delta_n^k \|^2+\tau^{-1} \|r^k \|^2 \geq 2 \langle r^k, \mathcal{G}^* \Delta_n^k \rangle \geq- \|\mathcal{G}^* \Delta_n^k \|^2- \|r^k \|^2.
$$
Therefore,
\begin{equation}\label{5.5}
\textstyle 2(1-\tau) \sigma \langle r^k, \mathcal{G}^* \Delta_n^k \rangle \geq \min \{\tau-1,\tau-\tau^2\}\sigma\|\mathcal{G}^* \Delta_n^k\|^2-\max \{1-\tau,1-\tau^{-1}\}\sigma \|r^{k} \|^2.
\end{equation}
Substituting \eqref{5.4} and \eqref{5.5} into \eqref{5.2} yields \eqref{5.1}. This completes the proof.
\end{proof}

We next establish a key inequality for the global convergence analysis.

\begin{proposition}\label{prop:descent}
Suppose that the solution set $\bar{\Omega}$ to the KKT system \eqref{KKT} of problem \eqref{P1} is nonempty and $P_{x}^{k}-Q \preceq \frac{2 \alpha}{\sigma} I$ with $\alpha \in (0,1)$. Let $\{\varpi^k\}:= \{ (m^k,n^k,x^k ) \}$ be generated by the SL-MD-isPADMM in Algorithm \ref{alg:tsp-p}. Then, for any $\bar{\varpi}:=(\bar{m}, \bar{n}, \bar{x}) \in \bar{\Omega}$ and $k \geq 1$,
\begin{equation}\label{5.6}
\begin{aligned}
& 2(1-\alpha) \langle d_{n}^{k}-d_{n}^{k-1}, \Delta_{n}^{k} \rangle-2 \langle d_{m}^{k}, m^{k+1}-\bar{m} \rangle-2 \langle d_n^k, n^{k+1}-\bar{n}\ \rangle\\
&  + \|\Delta_{m}^{k} \|_{\mathcal{S}}^{2}+ \|\Delta_{n}^{k} \|_{\mathcal{A}}^{2}+\beta \sigma \|r^{k+1} \|^{2}
\leq \phi_{k}(\bar{\varpi})-\phi_{k+1}(\bar{\varpi}).
\end{aligned}
\end{equation}
\end{proposition} 
\begin{proof}
For any given $(m, n, x)$, define $m_{e}:=m-\bar{m},~n_{e}:=n-\bar{n}, ~x_{e}:=x-\bar{x}$. Note that
\begin{equation}\label{5.7}
\begin{aligned}
  x^{k}+\sigma R (m^{k+1}, n^{k} )=\tilde{x}^{k+1}+\sigma \mathcal{G}^{*} \Delta_n^k.
\end{aligned}
\end{equation}
From Step 3.1 of Algorithm \ref{alg:tsp-p}, we have 
\begin{equation*}
\begin{aligned}
  d_{m}^{k}-\mathcal{F} (\tilde{x}^{k+1}+\sigma \mathcal{G}^{*} \Delta_{n}^{k} )-\widetilde{Q} m^{k+1}+\mathcal{S} \Delta_{m}^{k} \in \partial p (m^{k+1} ).
\end{aligned}
\end{equation*}
Since the function $p$ is convex, we obtain
\begin{equation}\label{5.9}
\begin{aligned}
  p(\bar{m})+\langle d_{m}^{k}-\mathcal{F} (\tilde{x}^{k+1}+\sigma \mathcal{G}^{*} \Delta_{n}^k )-\widetilde{Q} m^{k+1}+\mathcal{S} \Delta_{m}^{k}, m_e^{k+1}\rangle \geq p (m^{k+1} ).
\end{aligned}
\end{equation}
Similarly, using \eqref{5.3} and the convexity of $q$, we have 
\begin{equation}\label{5.10}
 q(\bar{n})+\langle d_{n}^{k}+\hat{b}-\mathcal{G} \tilde{x}^{k+1}+\mathcal{T} \Delta_{n}^k, n_e^{k+1} \rangle \geq q(n^{k+1}).
\end{equation}
By the convexity of $f$, $g$, $p$ and $q$,
together with KKT conditions \eqref{KKT2}, this yields 
\begin{align}
&  p (m^{k+1} )-p(\bar{m})+ \langle\mathcal{F}\bar{x}, m_{e}^{k+1} \rangle \geq \frac{1}{2} \|\bar{m} \|_{\widetilde{Q}}^{2}-\frac{1}{2} \| m^{k+1} \|_{\widetilde{Q}}^{2},  \label{5.11a}\\
&  q (n^{k+1} )-q(\bar{n})+ \langle\mathcal{G}\bar{x}, n_{e}^{k+1} \rangle \geq \langle\hat{b}, n_{e}^{k+1} \rangle.\label{5.11b}
\end{align}
Summing \eqref{5.9}, \eqref{5.10},  \eqref{5.11a} and \eqref{5.11b} together, we get
\begin{equation}\label{5.13}
\begin{aligned}
&   \langle d_{m}^{k}, m_{e}^{k+1} \rangle +  \langle d_{n}^{k}, n_{e}^{k+1} \rangle -  \langle\tilde{x}_{e}^{k+1}, \mathcal{F}^{*} m_{e}^{k+1}+\mathcal{G}^{*} n_{e}^{k+1} \rangle\\
&  -\sigma \langle\mathcal{G}^{*} \Delta_{n}^{k}, \mathcal{F}^{*}m_{e}^{k+1} \rangle
+  \langle\Delta_{m}^{k}, m_{e}^{k+1} \rangle _{\mathcal{S}} +  \langle\Delta_{n}^{k}, n_{e}^{k+1} \rangle_{\mathcal{T}}\\
\geq ~&  \frac{1}{2}\|\bar{m}\|_{\widetilde{Q}}^{2}-\frac{1}{2} \|m^{k+1} \|_{\widetilde{Q}}^{2}+ \langle m^{k+1}, m_e^{k+1} \rangle_{\widetilde{Q}}\\
=~& \frac{1}{2} \|m_{e}^{k+1}\|_{\widetilde{Q}}^{2}.
\end{aligned}
\end{equation}

We next estimate the left-hand side of \eqref{5.13}.
Observe that
\begin{equation}\label{5.14}
\begin{aligned}
&   \langle m_{e}^{k+1}, \Delta_{m}^{k} \rangle_{\mathcal{S}}=\frac{1}{2} \|m_{e}^{k} \|_{\mathcal{S}}^{2}-\frac{1}{2} \|m_{e}^{k+1} \|_{\mathcal{S}}^{2}-\frac{1}{2} \|\Delta_{m}^{k} \|_{\mathcal{S}}^{2}, \\
& \langle n_{e}^{k+1}, \Delta_{n}^{k} \rangle_{\mathcal{T}}=\frac{1}{2} \|n_{e}^{k} \|_{\mathcal{T}}^{2}-\frac{1}{2} \|n_{e}^{k+1} \|_{\mathcal{T}}^{2}-\frac{1}{2} \|\Delta_{n}^{k} \|_{\mathcal{T}}^{2}.
\end{aligned}
\end{equation}
Using the definition of $r^{k+1}$ and the identity
$
\mathcal F^*\bar m+\mathcal G^*\bar n=\nabla h(\bar x)-Q\bar x,
$
we have
\begin{equation}\label{5.15}
r^{k+1}
=
\mathcal{F}^{*} m_{e}^{k+1}
+
\mathcal{G}^{*} n_e^{k+1}
+
\nabla h(\bar{x})-\nabla h (x^{k})
+
Q x_{e}^{k}.
\end{equation}
Since $\nabla h$ is Lipschitz continuous, Clarke's mean value theorem
\cite[Proposition 2.6.5]{clarke1990optimization} gives a self-adjoint
linear operator $P_x^k \in \operatorname{conv}\big(\partial^2 h([\bar x,x^k])\big)$ such that
\begin{equation}\label{5.16}
\nabla h(x^{k})-\nabla h(\bar{x}) = P_x^{k} x_e^{k},
\end{equation}
where $\operatorname{conv}\big(\partial^2 h([\bar x,x^k])\big)$ denotes the convex
hull of the generalized Hessians of $h$ at all points in $[\bar x,x^k]$. 
The quadratic lower-bound condition \eqref{tilde_h} allows this operator to be taken with \(Q\preceq P_x^k\). Substituting \eqref{5.16} into \eqref{5.15}, we obtain
\[
r^{k+1}
=
\mathcal F^*m_e^{k+1}
+
\mathcal G^*n_e^{k+1}
-
(P_x^k-Q)x_e^k.
\]
From \eqref{5.15} and \eqref{5.16}, we get
\begin{equation}\label{5.17}
\begin{aligned}
& \langle\tilde{x}_e^{k+1},\mathcal{F}^*m_e^{k+1}+\mathcal{G}^* n_e^{k+1}\rangle+\sigma \langle \mathcal{G}^* \Delta_n^{k},\mathcal{F}^*m_e^{k+1}\rangle\\
=~&  \langle\tilde{x}_{e}^{k+1}, r^{k+1}- (Q-P_{x}^{k} ) x_{e}^{k} \rangle+\sigma \langle \mathcal{G}^{*} \Delta_{n}^{k}, r^{k+1}-\mathcal{G}^{*} n_{e}^{k+1}- (Q-P_{x}^{k} ) x_{e}^{k} \rangle \\
=~&  \langle\tilde{x}_{e}^{k+1}, r^{k+1} \rangle+\sigma \langle \mathcal{G}^{*} \Delta_{n}^{k}, r^{k+1}-\mathcal{G}^{*} n_{e}^{k+1} \rangle - \langle\tilde{x}_{e}^{k+1}, (Q-P_{x}^{k} ) x_{e}^{k} \rangle\\
~&-\sigma \langle \mathcal{G}^{*} \Delta_{n}^{k}, (Q-P_{x}^{k} ) x_{e}^{k} \rangle.
\end{aligned}
\end{equation}
By the definition of $\tilde{x}^{k+1}$, we obtain
\begin{equation}\label{5.18}
\begin{aligned}
  \langle \tilde{x}_{e}^{k+1}, r^{k+1} \rangle & = \frac{1}{\tau \sigma} \langle x^{k+1}-x^{k}, x_{e}^{k} \rangle+\sigma \|r^{k+1} \|^{2} \\
&   =\frac{1}{2 \tau \sigma} ( \|x_{e}^{k+1} \|^{2}- \|x_{e}^{k} \|^{2} )+\frac{(2-\tau) \sigma}{2} \|r^{k+1} \|^{2},
\end{aligned}
\end{equation}
where the last equality follows from the relation $x^{k+1}-x^k=\tau \sigma r^{k+1}$.  
Note also that
\begin{equation}\label{5.19}
\begin{aligned}
&  \langle \mathcal{G}^{*} \Delta_{n}^{k}, r^{k+1}-\mathcal{G}^{*} n_{e}^{k+1} \rangle \\
=~&  \langle\mathcal{G}^{*} \Delta_{n}^{k}, r^{k+1} \rangle- \langle\mathcal{G}^{*} n_{e}^{k}-\mathcal{G}^{*} n_{e}^{k+1}, \mathcal{G}^{*} n_{e}^{k+1} \rangle\\
=~&  \frac{1}{2} (\|\mathcal{G}^{*} (n^{k}-n^{k+1} )+\mathcal{F}^{*} m^{k+1}+\mathcal{G}^{*} n^{k+1}-c (x^{k} ) \|^{2}- \|\mathcal{G}^{*} n_{e}^{k}-\mathcal{G}^{*} n_{e}^{k+1} \|^{2}\!-\! \|r^{k+1} \|^{2} )\\
~&  -\frac{1}{2} ( \|\mathcal{G}^{*} n_{e}^{k} \|^{2}- \|\mathcal{G}^{*} n_{e}^{k}-\mathcal{G}^{*} n_{e}^{k+1}\|^{2}- \|\mathcal{G}^{*} n_{e}^{k+1} \|^{2} ) \\
=~&  \frac{1}{2} ( \|\mathcal{G}^{*} n_{e}^{k+1} \|^{2}- \|\mathcal{G}^{*} n_{e}^{k} \|^{2}+ \|R (m^{k+1}, n^{k} ) \|^{2}- \|r^{k+1} \|^{2} ).
\end{aligned}
\end{equation}
By using \eqref{5.7}, we achieve
\begin{equation}\label{5.20}
\begin{aligned}
&~ - \langle\tilde{x}_{e}^{k+1}, (Q-P_{x}^{k} ) x_{e}^{k} \rangle-\sigma \langle \mathcal{G}^{*} \Delta_{n}^{k}, (Q-P_{x}^{k} ) x_{e}^{k} \rangle\\
=&~ - \langle x_{e}^{k}+\sigma R (m^{k+1}, n^{k} ), (Q-P_{x}^{k} ) x_{e}^{k} \rangle \\
\geq &~  \|x_{e}^{k}\|_{P_x^{k}-Q}^{2}-\frac{\alpha \sigma}{2} \|R (m^{k+1}, n^{k}) \|^{2}-\frac{\sigma}{2 \alpha} \|x_{e}^{k} \|_{(Q-P_x^k)^{2}}^{2} \\
 =&~   \|x_{e}^{k} \|_{P_x^k-Q-\frac{\sigma}{2 \alpha} ( P_{x}^{k}-Q )^{2}}^{2}-\frac{\alpha \sigma}{2} \|R (m^{k+1}, n^{k} ) \|^{2}\\
\geq &~  -\frac{\alpha \sigma}{2} \|R (m^{k+1}, n^{k} ) \|^{2},
\end{aligned}
\end{equation}
where the last inequality follows from
$0 \preceq Q \preceq P_x^{k} \preceq Q+\frac{2\alpha}{\sigma}I$, 
which is a consequence of \(Q\preceq P_x^k\) and the assumption
\(P_x^k-Q\preceq \frac{2\alpha}{\sigma}I\).
Substituting \eqref{5.18}, \eqref{5.19}, and \eqref{5.20} into \eqref{5.17} gives\begin{equation}\label{5.21}
\begin{aligned}
& \langle \tilde{x}_e^{k+1},\mathcal{F}^*m_e^{k+1}+\mathcal{G}^* n_e^{k+1} \rangle+\sigma \langle \mathcal{G}^* \Delta_n^{k},\mathcal{F}^*m_e^{k+1}\rangle\\
\geq~ &  \frac{1}{2 \tau \sigma} ( \|x_{e}^{k+1} \|^{2}- \|x_{e}^{k} \|^{2} )
+\frac{\sigma}{2} ( \|\mathcal{G}^{*}n_{e}^{k+1} \|^{2}- \|\mathcal{G}^{*} n_{e}^{k} \|^{2} )
+\frac{(1-\tau) \sigma}{2} \|r^{k+1} \|^{2} \\
~&  +\frac{(1-\alpha) \sigma}{2} \|R (m^{k+1}, n^{k} ) \|^{2}.
\end{aligned}
\end{equation}

By substituting \eqref{5.14} and \eqref{5.21}  into  \eqref{5.13}, it follows that 
\begin{equation*}
\begin{aligned}
&~  \langle d_{m}^{k}, m_{e}^{k+1} \rangle+ \langle d_{n}^k, n_{e}^{k+1} \rangle
+\frac{1}{2 \tau\sigma} ( \|x_{e}^{k} \|^{2}- \|x_{e}^{k+1} \|^{2} )
+\frac{\sigma}{2} ( \|\mathcal{G}^{*}n_{e}^{k} \|^{2}
- \| \mathcal{G}^{*}n_{e}^{k+1} \|^{2} )\\
&~ +\frac{1}{2} ( \|m_{e}^{k} \|_{\widetilde{Q}+\mathcal{S}}^{2}+ \|n_{e}^{k} \|_{\mathcal{T}}^{2} )-\frac{1}{2} ( \|m_{e}^{k+1} \|_{\widetilde{Q}+\mathcal{S}}^{2}+ \|n_{e}^{k+1} \|_{\mathcal{T}}^{2} ) \\
\geq &~ \frac{1}{2} \|\Delta_{m}^{k} \|_{\mathcal{S}}^{2}+\frac{1}{2} \|\Delta_{n}^{k} \|_{\mathcal{T}}^{2}+\frac{(1-\tau) \sigma}{2} \|r^{k+1} \|^{2}+\frac{(1-\alpha) \sigma}{2} \|R (m^{k+1}, n^{k} ) \|^{2}.
\end{aligned}
\end{equation*}
Combining Lemma \ref{lem:3.1} with the definitions of $\phi_{k}(\varpi)$, $\gamma$, and $\mathcal{A}$, we conclude that \eqref{5.6} holds for all $k \geq 1$. This completes the proof.
 \end{proof}

The next proposition bounds the errors between the approximate and exact solutions of the subproblems. Since our inexactness criterion follows the imsPADMM framework in \cite{chen2017efficient}, the proof is analogous to that of \cite[Proposition 3.1]{chen2017efficient} and is omitted.

\begin{proposition}\label{inexact}
Let $ \{\varpi^k \}$ be the sequence generated by the SL-MD-isPADMM in Algorithm \ref{alg:tsp-p}, and $ \{\bar{m}^k \}$, $ \{\bar{n}^k \}$ be the sequence defined in \eqref{m} and \eqref{n}. Then, for any $k \geq 0$, we have $\|m^{k+1}-\bar{m}^{k+1}\|_{\mathcal{M}} \leq \varepsilon_{k}$ and $\|n^{k+1}-\bar{n}^{k+1}\|_{\mathcal{N}} \leq (1+\sigma\|\mathcal{N}^{-\frac{1}{2}}\mathcal{G}\mathcal{F}^*\mathcal{M}^{-\frac{1}{2}}\|)\varepsilon_{k}$.
\end{proposition}

We are now ready to state the convergence theorem for the SL-MD-isPADMM in Algorithm \ref{alg:tsp-p}.

\begin{theorem}\label{thm:conv}
Suppose the solution set $\bar{\Omega}$ of the KKT conditions \eqref{KKT} associated with problem \eqref{P1} is nonempty and $P_{x}^{k}-Q \preceq \frac{2 \alpha}{\sigma} I$ with $\alpha \in (0,1)$. Let $ \{\varpi^{k} \}$ be generated by the SL-MD-isPADMM in Algorithm \ref{alg:tsp-p}. Then, the sequence $ \{\varpi^{k} \}$ converges to a point in $\bar{\Omega}$. 
\end{theorem}
\begin{proof}
Let $\rho :=\min  \{\tau, 1+\tau-\tau^{2} \} \in (0,1]$. Then
$$
\mathcal{A}=\mathcal{T}+\rho(1-\alpha) \sigma \mathcal{GG}^{*}=\rho(1-\alpha) (\mathcal{T}+\sigma \mathcal{GG}^{*} )+(1-\rho(1-\alpha))\mathcal{T} \succ 0,
$$
and hence $\mathcal{A}^{-1}$ exists. Moreover,
\begin{equation}\label{5.22}
\begin{aligned}
&~ \|\Delta_{n}^{k} \|_{\mathcal{A}}^{2}+2(1-\alpha) \langle d_{n}^{k}-d_{n}^{k-1}, \Delta_{n}^{k} \rangle\\
=&~  \|\Delta_{n}^{k}+(1-\alpha) \mathcal{A}^{-1} (d_{n}^{k}-d_{n}^{k-1} ) \|_{\mathcal{A}}^{2}-(1-\alpha)^{2} \|d_{n}^{k}-d_{n}^{k-1} \|_{\mathcal{A}^{-1}}^{2}.
\end{aligned}
\end{equation}
Replacing $m^{k+1}$ and $n^{k+1}$ in \eqref{5.6} by $\bar{m}^{k+1}$ and $\bar{n}^{k+1}$, respectively, and using \eqref{5.22}, we obtain
\begin{equation}\label{5.24}
\begin{aligned}
&~  \phi_{k}(\bar{\varpi})-\bar{\phi}_{k+1}(\bar{\varpi})+(1-\alpha)^{2} \|d_n^{k-1} \|_{\mathcal{A}^{-1}}^{2}\\ 
\geq &~  \|n^{k}-\bar{n}^{k+1}-(1-\alpha) \mathcal{A}^{-1} d_n^{k-1} \|_{\mathcal{A}}^{2}+ \|m^{k}-\bar{m}^{k+1} \|_{\mathcal{S}}^{2}+\beta \sigma \|\bar{r}^{k+1} \|^{2}.
\end{aligned}
\end{equation}
Define $\xi^{k}:= \big( \frac{1}{\sqrt{\tau \sigma}}  x_{e}^{k}, \sqrt{\smash[b]{\widetilde{Q}}+\mathcal{S}}
 m_{e}^{k}, \sqrt{\mathcal{N}} n_{e}^{k}, \sqrt{(1-\alpha) \gamma \sigma} r^{k}, \sqrt{(1-\alpha) \mathcal{T}} \Delta_{n}^{k-1} \big)$ and $\bar{\xi}^{k}:= \big(\frac{1}{\sqrt{\tau \sigma}} \bar{x}_{e}^{k}$, $\sqrt{\smash[b]{\widetilde{Q}}+\mathcal{S}}
 \bar{m}_{e}^{k}, \sqrt{\mathcal{N}} \bar{n}_{e}^{k}, \sqrt{(1-\alpha) \gamma \sigma} \bar{r}^{k}, \sqrt{(1-\alpha) \mathcal{T}} (n^{k-1}-\bar{n}^{k} ) \big)$. Then $ \|\xi^{k} \|^{2}=\phi_{k}(\bar{\varpi})$ and $ \|\bar{\xi}^{k} \|^{2}=\bar{\phi}_{k}(\bar{\varpi})$. It follows from these identities and \eqref{5.24} that $ \|\bar{\xi}^{k+1} \|^{2} \leq \|\xi^{k} \|^{2}+(1-\alpha)^{2}\| \mathcal{A}^{-\frac{1}{2}} d_{n}^{k-1} \|^{2}$, hence  $ \|\bar{\xi}^{k+1} \| \leq  \|\xi^{k} \|+(1-\alpha) \|\mathcal{A}^{-\frac{1}{2}} d_{n}^{k-1} \|$. Consequently,\begin{equation}\label{5.25}
   \|\xi^{k+1} \| \leq \|\xi^{k} \|+(1-\alpha) \|\mathcal{A}^{-\frac{1}{2}}d_{n}^{k-1} \|+\|\bar{\xi}^{k+1}-\xi^{k+1} \|. 
\end{equation}
Now, we estimate $ \|\bar{\xi}^{k+1}-\xi^{k+1} \|$ in \eqref{5.25}. Since $\tau \in  ( 0, \frac{1+\sqrt{5-4\alpha}}{2}  ]  $ and $\alpha \in(0,1)$, we have
$$
\begin{aligned}
&~ \frac{1}{\tau \sigma} \|\bar{x}^{k+1}-x^{k+1} \|^{2}+(1-\alpha)\gamma \sigma \|\bar{r}^{k+1}-r^{k+1} \|^{2} \\
= &~   (\tau+(1-\alpha)\gamma) \sigma \|\bar{r}^{k+1}-r^{k+1} \|^{2} \\
\leq &~  3 \sigma \|\mathcal{F}^{*} (\bar{m}^{k+1}-m^{k+1} )+\mathcal{G}^{*} (\bar{n}^{k+1}-n^{k+1} ) \|^{2} \\
\leq &~  6 \|\bar{m}^{k+1}-m^{k+1} \|_{\sigma \mathcal{FF}^{*}}^{2}+6 \|\bar{n}^{k+1}-n^{k+1} \|_{\sigma \mathcal{GG}^{*}}^{2}.
\end{aligned}
$$
Together with Proposition \ref{inexact}, this implies
\begin{equation}\label{5.27}
\begin{aligned}
   \|\bar{\xi}^{k+1}-\xi^{k+1} \|^{2}
\leq &~  \|\bar{m}^{k+1}-m^{k+1} \|_{\widetilde{Q}+\mathcal{S}}^{2}+ \|\bar{n}^{k+1}-n^{k+1} \|_{\mathcal{N}}^{2}+ \|n^{k+1}-\bar{n}^{k+1} \|^{2}_{(1-\alpha) \mathcal{T}}\\
&~  +6 \|\bar{m}^{k+1}-m^{k+1} \|_{\sigma \mathcal{FF}^{*}}^{2}+6 \|\bar{n}^{k+1}-n^{k+1}\|_{\sigma \mathcal{GG}^{*}}^{2}\\
\leq &~  7 ( \|\bar{m}^{k+1}-m^{k+1} \|_{\mathcal{M}}^{2}+ \|\bar{n}^{k+1}-n^{k+1} \|_{\mathcal{N}}^{2} )=\kappa^{2} \varepsilon_{k}^{2},
\end{aligned}
\end{equation}
where 
$\kappa:=\sqrt{\smash[b]{7 \bigl(1+ (1+\sigma \|\mathcal{N}^{-\frac{1}{2}}\mathcal{GF}^{*}\mathcal{M}^{-\frac{1}{2}}\| )^{2} \bigr)}}.$ 
From Algorithm \ref{alg:tsp-p}, we have that
$
\|\mathcal{A}^{-\frac{1}{2}}d_{n}^{k}\| \leq \| \mathcal{A}^{-\frac{1}{2}} \mathcal{N}^{\frac{1}{2}} \| \varepsilon_{k}.
$
Combining above inequalities with \eqref{5.25}, we obtain
\begin{equation}\label{5.28}
\begin{aligned}
\|\xi^{k+1}\|
&\leq \|\xi^{k}\|+\kappa \varepsilon_{k}+(1-\alpha)\|\mathcal{A}^{-\frac{1}{2}}\mathcal{N}^{\frac{1}{2}}\|\varepsilon_{k-1} \leq \|\xi^{1}\|+\bigl(\kappa+(1-\alpha)\|\mathcal{A}^{-\frac{1}{2}}\mathcal{N}^{\frac{1}{2}}\|\bigr)\mathcal{E},
\end{aligned}
\end{equation}
where
$
\mathcal{E}:=\sum_{k=0}^{\infty}\varepsilon_k$ and
$\mathcal{E}':=\sum_{k=1}^{\infty}\varepsilon_k^2$.
Therefore, $\{\xi^k\}$ is bounded. It then follows from \eqref{5.27} that $\{\bar{\xi}^k\}$ is bounded as well. By the definition of $\xi^k$, the sequences $\{n^k\}$, $\{x^k\}$, $\{r^k\}$ and $\{\sqrt{\smash[b]{\widetilde{Q}}+\mathcal{S}}\,m^k\}$ are all bounded. Moreover, by the definition of $r^k$, the sequence $\{\mathcal{F}^*m^k\}$ is bounded. Together with $\mathcal{M} = \widetilde{Q}+\mathcal{S} + \sigma\mathcal{FF}^*\succ 0$, this further implies that $\{m^k\}$ is bounded.
Next, by \eqref{5.24}, \eqref{5.27} and \eqref{5.28}, we derive
\begin{equation}\label{5.29}
\begin{aligned}
&~  \sum_{k=1}^{\infty} ( \|\bar{m}^{k+1}-m^{k} \|_{\mathcal{S}}^{2}+\beta \sigma\|\bar{r}^{k+1} \|^{2}+ \|\bar{n}^{k+1}-n^{k}+(1-\alpha) \mathcal{A}^{-1}d_n^{k-1} \|_{\mathcal{A}}^2 )\\
\leq &~  \sum_{k=1}^{\infty} (\phi_{k}(\bar{\varpi})-\phi_{k+1}(\bar{\varpi})+\phi_{k+1}(\bar{\varpi})-\bar{\phi}_{k+1}(\bar{\varpi})+(1-\alpha)^{2} \|d_{n}^{k-1} \|_{\mathcal{A}^{-1}}^{2} )\\
\leq &~  \phi_{1}(\bar{\varpi})+\sum_{k=1}^{\infty} \|\xi^{k+1}-\bar{\xi}^{k+1} \|( \|\xi^{k+1}\|+ \|\bar{\xi}^{k+1} \| )+ \|\mathcal{A}^{-\frac{1}{2}} \mathcal{N}^{\frac{1}{2}} \|^{2} \mathcal{E}'\\
\leq
&~  \phi_{1}(\bar{\varpi})+ \|\mathcal{A}^{-\frac{1}{2}} \mathcal{N}^{\frac{1}{2}} \|^{2} \mathcal{E}^{\prime}+\kappa \max _{k \geq 1} \{ \|\xi^{k+1} \|+ \|\bar{\xi}^{k+1} \| \} \mathcal{E}<\infty,
\end{aligned}
\end{equation}
where the second inequality is because $\phi_{k+1}(\bar{\varpi})-\bar{\phi}_{k+1}(\bar{\varpi})= \|\xi^{k+1}\|^2-\|\bar{\xi}^{k+1} \|^2
\leq \|\xi^{k+1}-\bar{\xi}^{k+1} \| ( \|\xi^{k+1} \|+ \|\bar{\xi}^{k+1} \| )$. Hence, \eqref{5.29} implies that
$\|\bar{m}^{k+1}-m^{k} \|_{\mathcal{S}}^{2}  \rightarrow 0$, $ \| \bar{n}^{k+1}-n^{k}+(1-\alpha) \mathcal{A}^{-1} d_n^{k-1} \|_{\mathcal{A}}^{2}  \rightarrow 0$ and $  \|\bar{r}^{k+1}\|^{2} \rightarrow 0$ as $k \rightarrow \infty$. Therefore, $ \mathcal{S}(\bar{m}^{k+1}-m^{k})  \rightarrow 0$.
Since $\mathcal A\succ0$, we further obtain $ \bar{n}^{k+1}-n^{k}  \rightarrow 0$. By  Proposition \ref{inexact}, we also have $ \bar{m}^{k}-m^{k}  \rightarrow 0$ and $ \bar{n}^{k}-n^{k}  \rightarrow 0$ . Consequently, $ \bar{r}^{k}-r^{k}  \rightarrow 0$ and $ r^{k}  \rightarrow 0$, and hence $\mathcal{S}\Delta_{m}^{k}  \rightarrow 0$ and $ \Delta_{n}^{k}  \rightarrow 0$ as $k\to\infty$. Since $\{(m^{k+1}, n^{k+1}$, $x^{k+1})\}$ is bounded, there exists a subsequence $\{(m^{k_i+1}, n^{k_i+1}, x^{k_i+1})\}$ converging to $(m^{\infty}, n^{\infty}, x^{\infty})$.
From Steps 3.1 and 3.2 of Algorithm \ref{alg:tsp-p}, together with the definitions of 
$\tilde{x}^{k+1}$ and $x^{k+1}$, we have, for any $k \geq 0$,
\begin{equation}\label{5.30}
\begin{cases}
d_{m}^{k}-(\tau-1) \sigma \mathcal{F} r^{k+1}-\sigma \mathcal{FG}^{*} \Delta_{n}^{k}+\mathcal{S} \Delta_m^k \in \partial p (m^{k+1} )+\partial_m f (m^{k+1} )+\mathcal{F} x^{k+1},\\
d_{n}^{k}-(\tau-1) \sigma \mathcal{G} r^{k+1}+\mathcal{T} \Delta_{n}^{k} \in \partial q (n^{k+1} )+\nabla g (n^{k+1} )+\mathcal{G} x^{k+1}.
\end{cases}
\end{equation}
Taking limits along the subsequence $\{k_i\}$ as $i \rightarrow \infty$ in \eqref {5.30} yields
$$
0 \in \partial p (m^{\infty} )+\partial_m f (m^{\infty} )+\mathcal{F} x^{\infty}, \quad 0\in \partial q (n^{\infty} )+\nabla g (n^{\infty} )+\mathcal{G} x^{\infty}.
$$
Moreover, since $r^k \to 0$ and $x^{k+1}-x^k=\tau\sigma r^{k+1},$
we have $x^{k_i}-x^{k_i+1}\to 0$. Together with $x^{k_i+1}\to x^\infty$, this implies that $x^{k_i}\to x^\infty$. Hence, by
$r^{k_i+1}=\mathcal{F}^*m^{k_i+1}+\mathcal{G}^*n^{k_i+1}-c(x^{k_i}) \to 0
$ and the definition $c(x^k)$, we obtain
\(
\mathcal{F}^*m^{\infty}+\mathcal{G}^*n^{\infty}
=
\nabla h(x^{\infty})-Qx^{\infty}.
\)
Therefore, $(m^{\infty},n^{\infty},x^{\infty})$ satisfies the KKT system \eqref{KKT2}. Since \eqref{KKT} and \eqref{KKT2} are equivalent, we conclude that $(m^{\infty}, n^{\infty}, x^{\infty}) \in \bar{\Omega}.$

We next show that the whole sequence converges.  
Since $r^k \to 0$, $\Delta_n^k \to 0$ and subsequence $(m^{k_i+1},n^{k_i+1},x^{k_i+1}) \to (m^{\infty},n^{\infty},x^{\infty})\in\bar{\Omega}$, it follows from the definition of $\xi^k$ that $\|\xi^{k_i+1}\| \to 0$ as $i \to \infty$. 
Moreover, by \eqref{5.25} and the estimates below it, there exists a summable sequence $\{\eta_k\}$ such that
$\|\xi^{k+1}\| \le \|\xi^k\|+\eta_k .$ 
Hence $\{\|\xi^k\|\}$ is convergent. Since a subsequence $\{\|\xi^{k_i+1}\|\}$ converges to $0$, we must have
$\lim_{k\to\infty}\|\xi^k\|=0.$
 By the definition of $ \{\xi^{k} \}$, we conclude that $\lim _{k\rightarrow \infty} x^{k}=x^{\infty}$, $\lim_{k \rightarrow \infty} n^{k}=n^{\infty}$, and $\lim_{k \rightarrow \infty}  (\widetilde{Q}+\mathcal{S}) m^{k}=(\widetilde{Q}+\mathcal{S}) m^{\infty}$. Given that $\lim _{k \rightarrow \infty} r^{k}=0$, we have $ \mathcal{F}^{*} m^{k+1}  \rightarrow \mathcal{F}^{*} m^{\infty}$ as $k \rightarrow \infty$, which, together with $\mathcal{M} \succ 0$ implies that $\lim _{k \rightarrow \infty} m^{k}=m^{\infty}$. This completes the proof.

\end{proof}

We finally give the convergence theorem for the general MSP problem \eqref{MSP}. This result is obtained by carrying out the same proof pattern after replacing the three-stage grouping by the corresponding (T)-stage grouping, and it makes explicit that the convergence guarantee applies to Algorithm~\ref{alg:msp}.

\begin{theorem}
\label{thm:conv-msp}
Suppose that the solution set $\bar{\Omega}_T$ of the KKT system associated with \eqref{MSP} is nonempty. 
Assume further that
\(P_x^k-Q \preceq \frac{2\alpha}{\sigma}I\)
with $\alpha\in(0,1)$, where \(Q=\operatorname{diag}(Q_1,\ldots,Q_T)\). Let $\{\bar{\varpi}^k\}$ be the sequence generated by Algorithm~\ref{alg:msp}. 
Then $\{\bar{\varpi}^k\}$ converges to a point in $\bar{\Omega}_T$.
\end{theorem}
\begin{proof}
We only indicate the modifications needed for the multi-stage case, since the argument is identical after replacing the three-stage grouping by the \(T\)-stage grouping.
For the MSP problem \eqref{MSP}, the dual variables are grouped as
\(m=(v_{1:T},w_{1:T})\) and \(n=(z_{1:T},y_{1:T})\). 
The restricted Wolfe dual of the corresponding minorized subproblem admits the
same grouped structure as in the three-stage case. The only changes are the
index ranges, the dimensions of the grouped variables, and the concrete forms of
the associated sGS operators. These operators are the natural \(T\)-stage
extensions of \(\widetilde{\mathcal M}\) and \(\widetilde{\mathcal N}\), and preserve
the scenario-tree block structure. Thus, the sGS decomposition theorem gives the block sweeps in
Algorithm~\ref{alg:msp}, and the blockwise errors satisfy the required grouped
inexactness conditions. 
Therefore, Lemma~\ref{lem:3.1}, Proposition~\ref{prop:descent}, and Proposition~\ref{inexact} remain valid after replacing the three-stage block operators by their \(T\)-stage counterparts. 
The convergence assertion follows from the same argument as in Theorem~\ref{thm:conv}.
\end{proof}

\section{Numerical experiments}\label{sec:ne}

In this section, we conduct numerical experiments to evaluate the performance and scalability of the proposed SL-MDD method. 
For comparison, we also implement a primal sGS-iADMM method (P-sGS-iADMM), which applies the inexact sGS majorized ADMM directly to the primal reformulation with auxiliary variables introduced in Section~\ref{sec:alg}, rather than to the minorized restricted Wolfe dual used by the proposed SL-MDD method. 
Notably, both algorithms feature efficiently parallelizable subproblems.

\subsection{Portfolio optimization with a quadratic utility function}\label{sec:exp1}

We consider a multi-stage portfolio optimization problem involving $\tilde{d}$ risky assets, following the modeling framework of \cite{dantzig1993multi,lan2021dynamic}. The problem is formulated as follows:
\begin{equation}\label{MSPe}
\begin{aligned}
\min_{x^1} \quad & -u\big(\sum_{i=1}^{\tilde{d}} x_i^1\big)
- \frac{1}{N_2} \sum_{k_2=1}^{N_2}\Big( \min\limits_{x^2(\xi_{[2]}^{k_{[2]}})} u\big(\sum_{i=1}^{\tilde{d}} x_i^2(\xi_{[2]}^{k_{[2]}})\big) - \cdots  \\
&  -\frac{1}{N_T} \sum_{k_T=1}^{N_T} \min\limits_{x^T(\xi_{[T]}^{k_{[T]}})} u\big(\sum_{i=1}^{\tilde{d}} x_i^T(\xi_{[T]}^{k_{[T]}})\big) \Big) \\
\text{s.t.} \quad
& \sum_{i=1}^{\tilde{d}} x_i^0 = w_0, \; \textstyle \sum_{i=1}^{\tilde{d}} b_i^{t,+}(\xi_{[t]}^{k_{[t]}}) = \sum_{i=1}^{\tilde{d}} b_i^{t,-}(\xi_{[t]}^{k_{[t]}}), \\
&  x_i^t(\xi_{[t]}^{k_{[t]}}) = R_i^t(\xi_{[t]}^{k_{[t]}})
\big(x_i^{t-1}(\xi_{[t-1]}^{k_{[t-1]}}) + b_i^{t,+}(\xi_{[t]}^{k_{[t]}}) - b_i^{t,-}(\xi_{[t]}^{k_{[t]}})\big), \\
&  x_i^0 \geq 0,\quad x_i^t(\xi_{[t]}^{k_{[t]}}) \geq -\bar{p},\quad
0 \leq b_i^{t,+}(\xi_{[t]}^{k_{[t]}}) \leq \bar{b}^+,\quad
0 \leq b_i^{t,-}(\xi_{[t]}^{k_{[t]}}) \leq \bar{b}^- ,
\end{aligned}
\end{equation}
where $t = 1, \cdots, T$ and $i = 1, \cdots, \tilde d$. Here, $x_i^0$ and $w_0$ denote the initial holding of asset $i$ and the initial wealth, respectively. For each stage $t$ and scenario node $\xi_{[t]}^{k_{[t]}}$, $x_i^t(\xi_{[t]}^{k_{[t]}})$ denotes the holding of asset $i$, while $b_{i}^{t,+}(\xi_{[t]}^{k_{[t]}})$ and $b_{i}^{t,-}(\xi_{[t]}^{k_{[t]}})$ denote the corresponding purchase and sale amounts. The random return of asset $i$ at stage $t$ under scenario node $\xi_{[t]}^{k_{[t]}}$ is denoted by $R_i^t(\xi_{[t]}^{k_{[t]}})$. Scenario realizations are sampled sequentially across stages to form a discrete scenario tree with $N_t = 40$ for each stage. The random variables are generated as $\xi_2 \sim \mathcal{U}[0,1]$ and $\xi_t \sim \mathcal{U}[\xi_{t-1},\xi_{t-1}+1]$ for $t \geq 3$. The investor's risk preference is described by the concave quadratic utility function
\(
u(W_t)=W_t-\lambda W_t^2,
\)
where $W_t$ denotes the total wealth at stage $t$, with
\(
W_t(\xi_{[t]}^{k_{[t]}})=\sum_{i=1}^{\tilde d} x_i^t(\xi_{[t]}^{k_{[t]}}),
\)
and $\lambda=\frac{1}{3w_0}$; see \cite{pedersen2003utility}. Short selling is allowed subject to the lower bound $x_i^t(\xi_{[t]}^{k_{[t]}})\ge -\bar p$, and the trading bounds $\bar b^+$ and $\bar b^-$ are imposed to reflect practical market liquidity constraints.

\begin{remark}
The solution accuracy of the proposed SL-MDD method is evaluated by the relative KKT residual
\[
\eta_{\mathrm{KKT}} = \max_{1 \le t \le T} \{ \eta_{P_t}, \eta_{D_t}, \eta_{Q_t}, \eta_{\theta_t}, \eta_{\mathcal{K}_t} \},
\]
where the individual residuals are defined for $t=1,\ldots,T$, with $x_0=0$, $B_{T+1}=0$, and $y_{T+1}=0$, as follows:
\[
\begin{aligned}
\eta_{P_t} &= \frac{\| B_tx_{t-1} + A_tx_t - b_t \|}{1 + \| b_t \|}, \quad
\eta_{D_t} = \frac{\| A_t^*y_t + B_{t+1}^*y_{t+1} + v_t + z_t - Q_tw_t - c_t(x_t') \|}{1 + \| c_t(x_t') \|}, \\
\eta_{Q_t} &= \frac{\| Q_tw_t - Q_tx_t\|}{1 + \| Q_t \|}, \quad
\eta_{\theta_t} = \frac{\| x_t - \mathrm{Prox}_{\theta_t}(x_t - v_t) \|}{1 + \| x_t \| + \| v_t \|}, \quad
\eta_{\mathcal{K}_t} = \frac{\| x_t - \Pi_{\mathcal{K}_t}(x_t - z_t) \|}{1 + \| x_t \| + \| z_t \|}.
\end{aligned}
\]
The relative duality gap is defined by
\[
\eta_{\mathrm{gap}} =
\frac{|\mathrm{obj}_P-\mathrm{obj}_D|}{1+|\mathrm{obj}_P|+|\mathrm{obj}_D|},
\]
where
\[
\mathrm{obj}_P := \sum_{t=1}^{T} \big(\theta_t(x_t)+h_t(x_t)+\delta_{\mathcal{K}_t}(x_t)\big),
\]
and
\[
\begin{aligned}
\mathrm{obj}_D := \sum_{t=1}^{T} \Big(&-\frac{1}{2}\langle w_t,Q_tw_t\rangle
-\theta_t^*(-v_t)-\delta_{\mathcal{K}_t}^*(-z_t)+\langle b_t,y_t\rangle \\
&+h_t(x_t')+\frac{1}{2}\langle x_t',Q_tx_t'\rangle
-\langle \nabla h_t(x_t'),x_t'\rangle\Big).
\end{aligned}
\]
The SL-MDD method terminates when $\max\{\eta_{\mathrm{KKT}},\eta_{\mathrm{gap}}\}\le 10^{-3}$. 
For P-sGS-iADMM, we use its standard primal-dual residual stopping criterion with tolerance $10^{-3}$ and adaptively update the penalty parameter every 50 iterations by balancing the primal and dual residuals.
\end{remark}

We test both P-sGS-iADMM and the SL-MDD method on several instances of problem \eqref{MSPe} to examine their behavior under different dimensions and stage numbers. Let $w_0 = 5$ and we consider three groups of instances: the first group (Instances 1-3) fixes the number of stages at $T = 3$ and varies the number of assets $\tilde{d} \in \{5, 200, 400\}$; the second group (Instances 4-6) fixes $T = 4$ and uses the same variation in assets $\tilde{d} \in \{5, 200, 400\}$; the third group (Instances 1, 4 and 7) fixes $\tilde{d} = 5$ and varies $T \in \{3, 4, 5\}$. 
Parallel runtimes are estimated assuming a parallel computing environment with 1000 processing cores.

\begin{figure}[htbp]
  \centering
  \begin{minipage}{0.32\textwidth}
    \centering
    \includegraphics[width=\linewidth, height=0.8\linewidth]{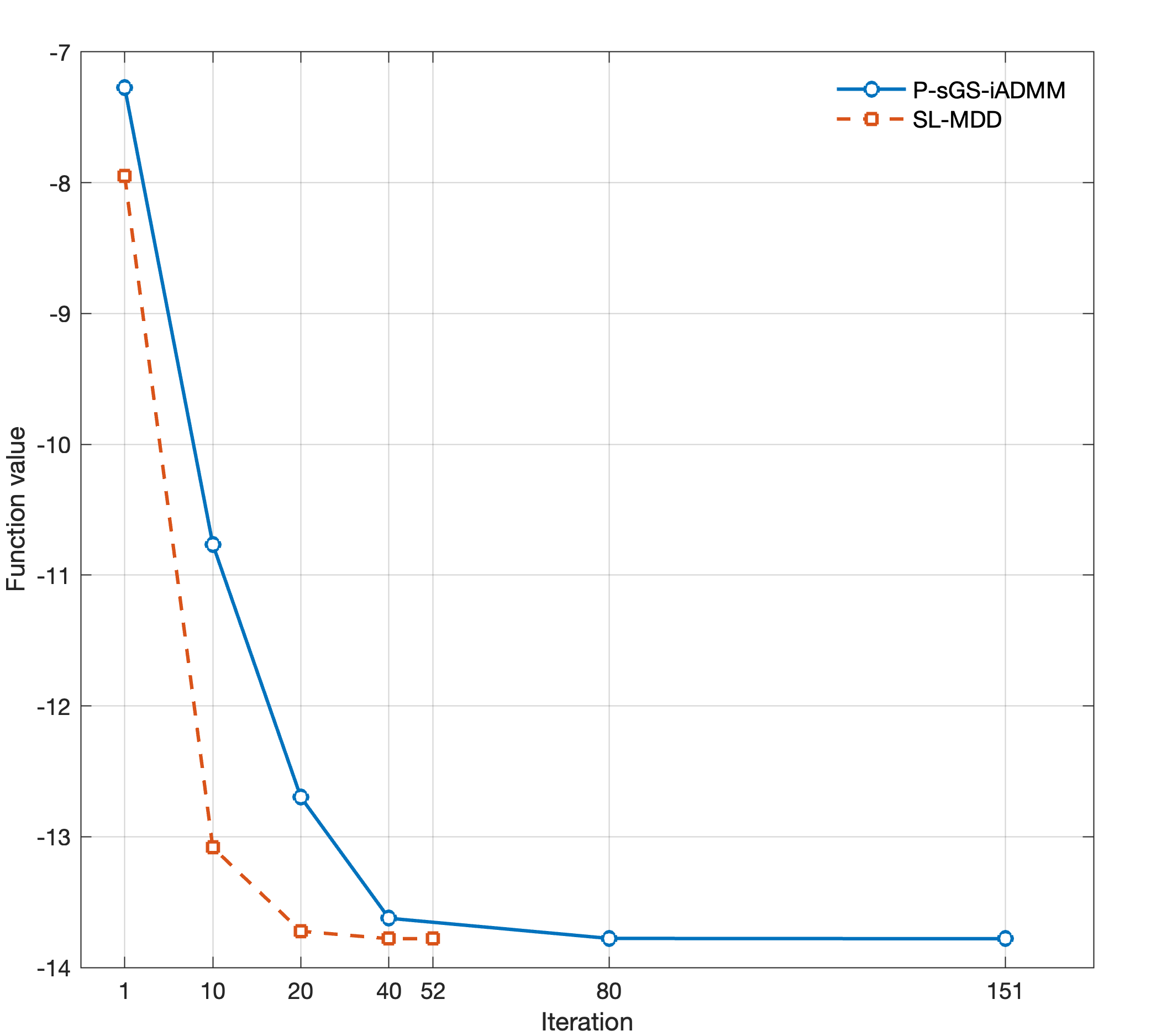}
    \\[2pt]
    \footnotesize(a) Function value vs iteration
  \end{minipage}
  \hfill
  \begin{minipage}{0.32\textwidth}
    \centering
    \includegraphics[width=\linewidth, height=0.8\linewidth]{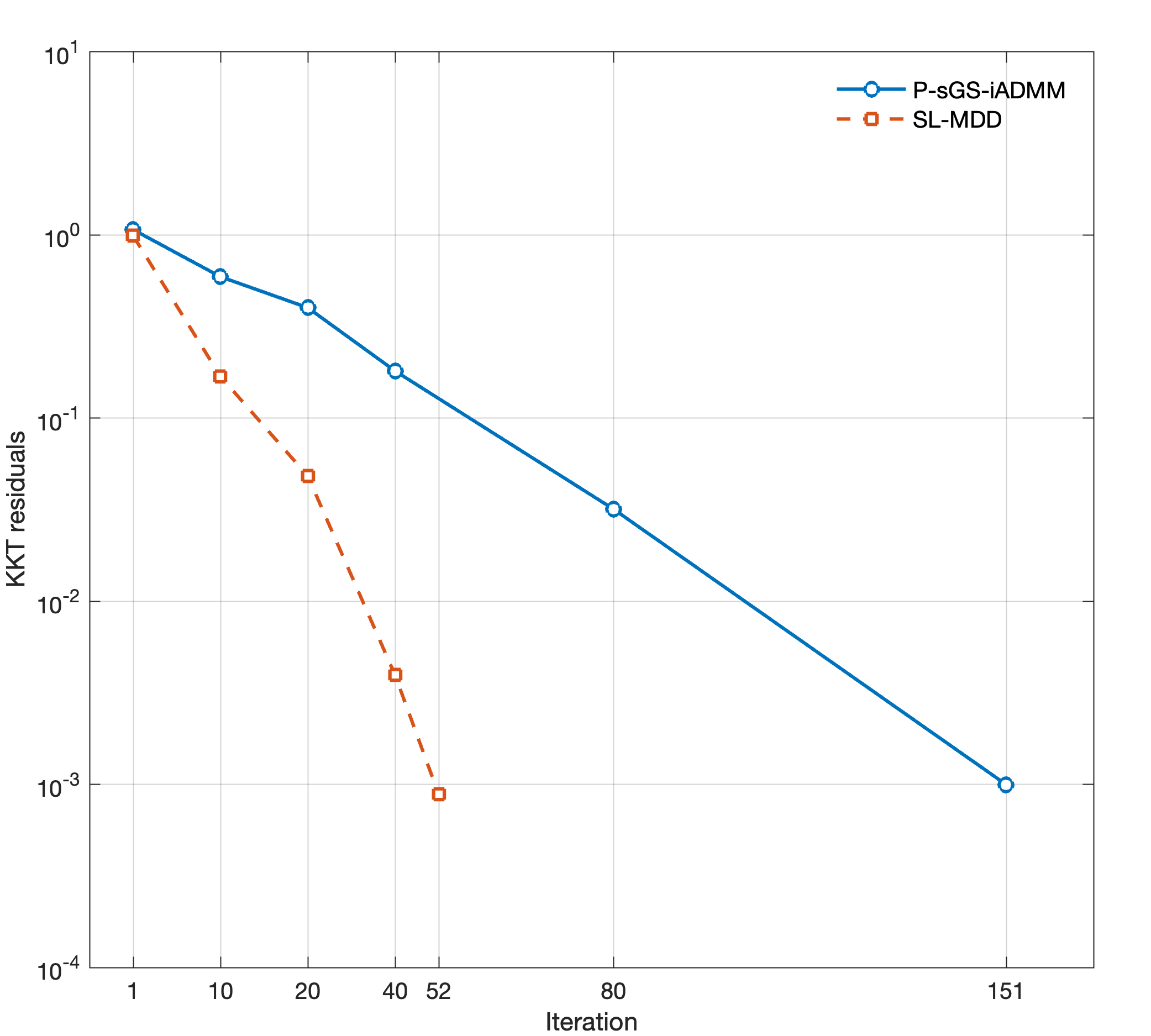}
    \\[2pt]
    \footnotesize(b) KKT residuals vs iteration
    \end{minipage}
  \hfill
  \begin{minipage}{0.32\textwidth}
    \centering
    \includegraphics[width=\linewidth, height=0.8\linewidth]{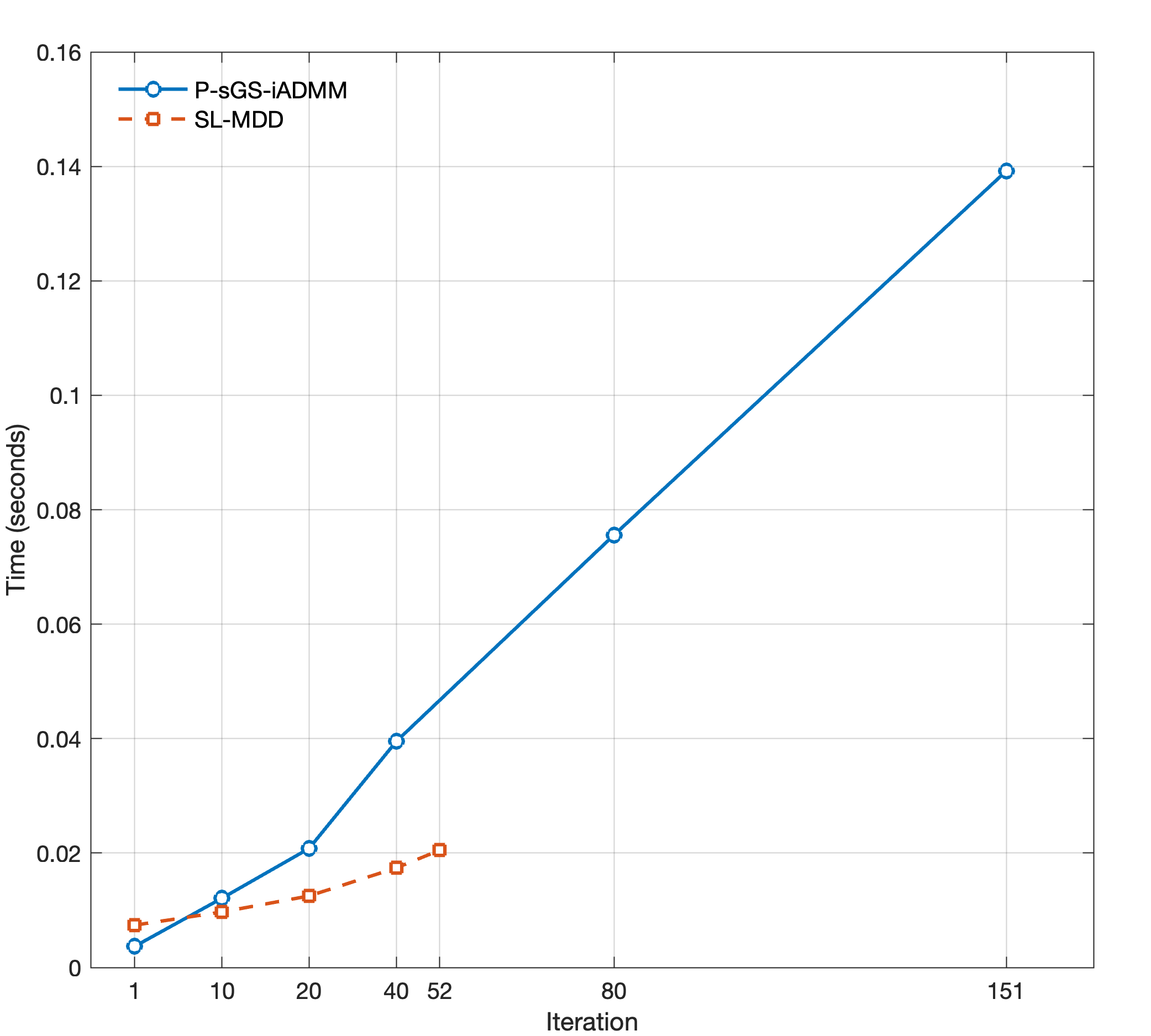}
    \\[2pt]
    \footnotesize(c) Parallel time vs iteration
  \end{minipage}
  \caption{Instance 1 (Experiment 1): $\tilde{d}=5, T=3$}
  \label{fig:Instance1}
\end{figure}
\begin{figure}[htbp]
  \centering
  \begin{minipage}{0.32\textwidth}
    \centering
    \includegraphics[width=\linewidth, height=0.8\linewidth]{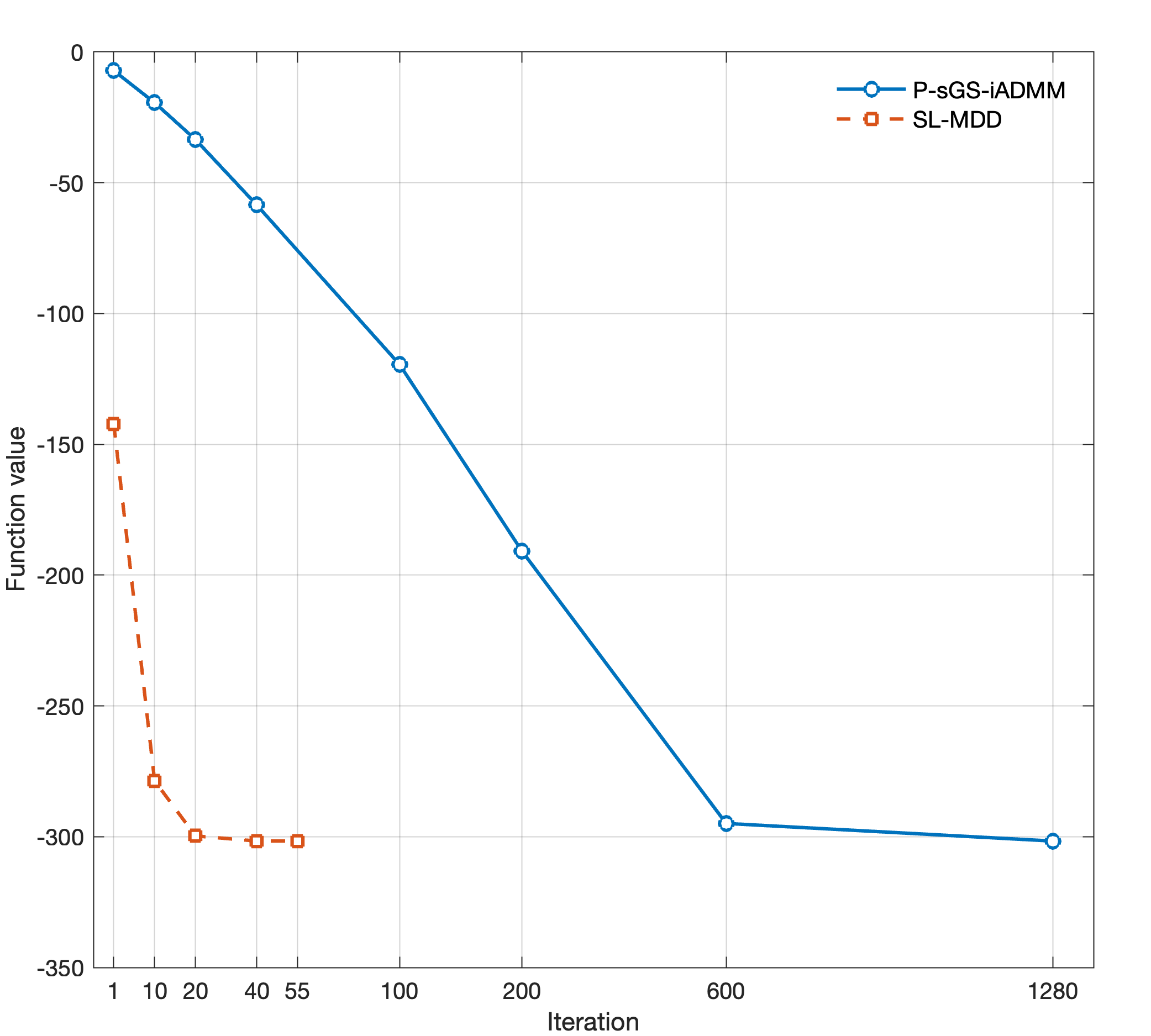}
    \\[2pt]
    \footnotesize(a) Function value vs iteration
  \end{minipage}
  \hfill
  \begin{minipage}{0.32\textwidth}
    \centering
    \includegraphics[width=\linewidth, height=0.8\linewidth]{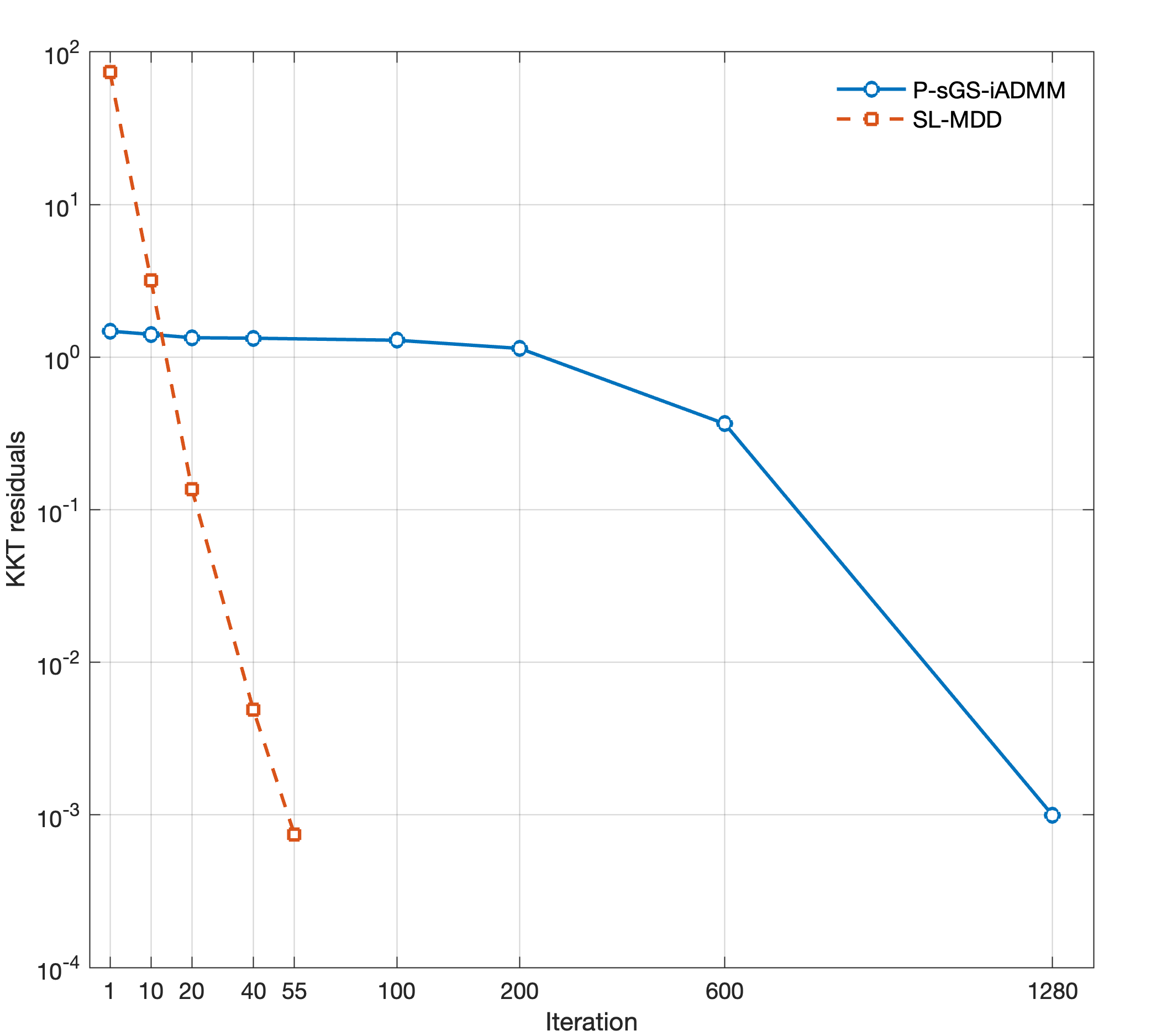}
    \\[2pt]
    \footnotesize(b) KKT residuals vs iteration
      \end{minipage}
  \hfill
  \begin{minipage}{0.32\textwidth}
    \centering
    \includegraphics[width=\linewidth, height=0.8\linewidth]{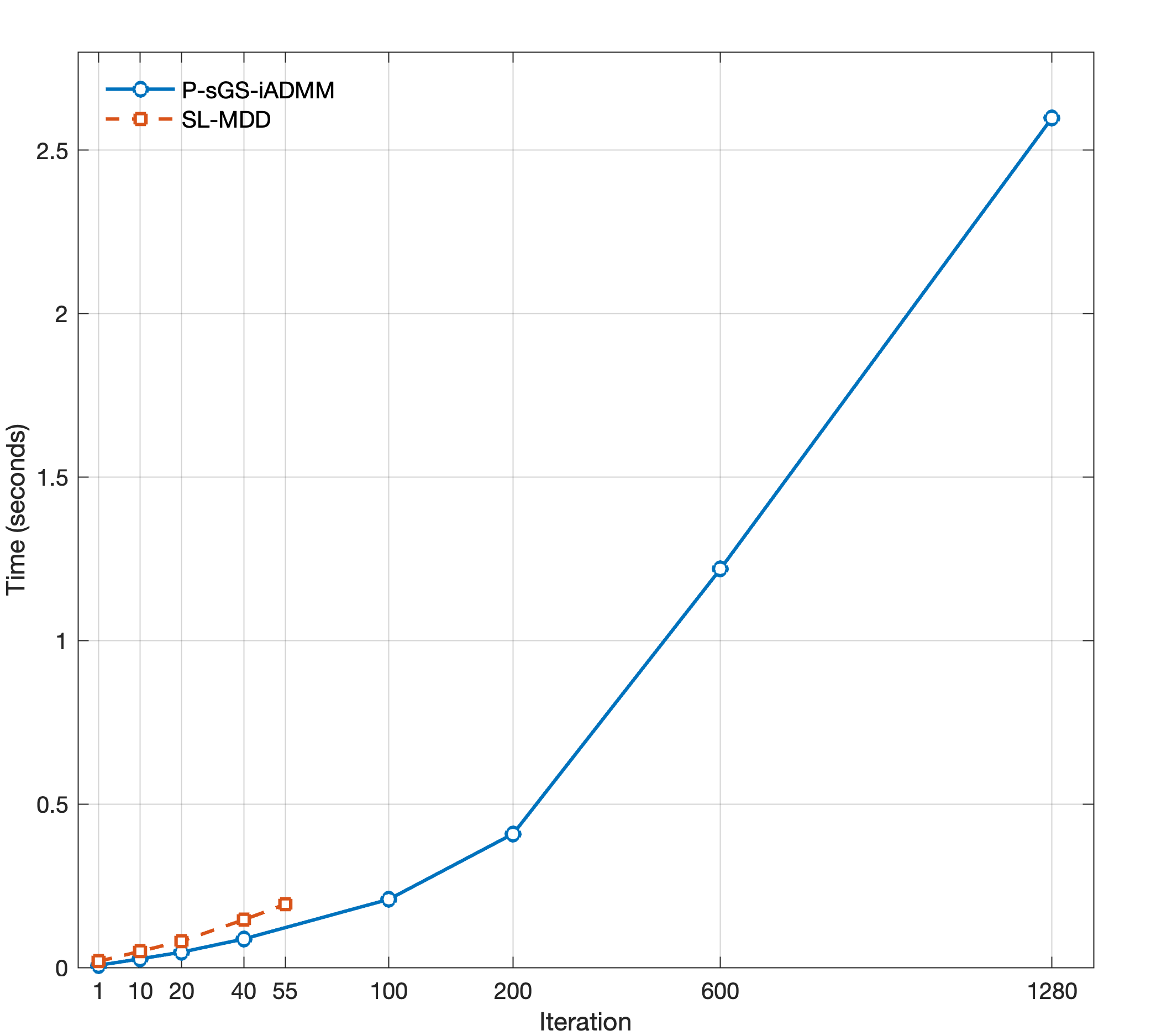}
    \\[4pt]
    \footnotesize(c) Parallel time vs iteration
      \end{minipage}
  \caption{Instance 2 (Experiment 1): $\tilde{d}=200, T=3$}
  \label{fig:Instance2}
\end{figure}
\begin{figure}[htbp]
  \centering
  \begin{minipage}{0.32\textwidth}
    \centering
    \includegraphics[width=\linewidth, height=0.8\linewidth]{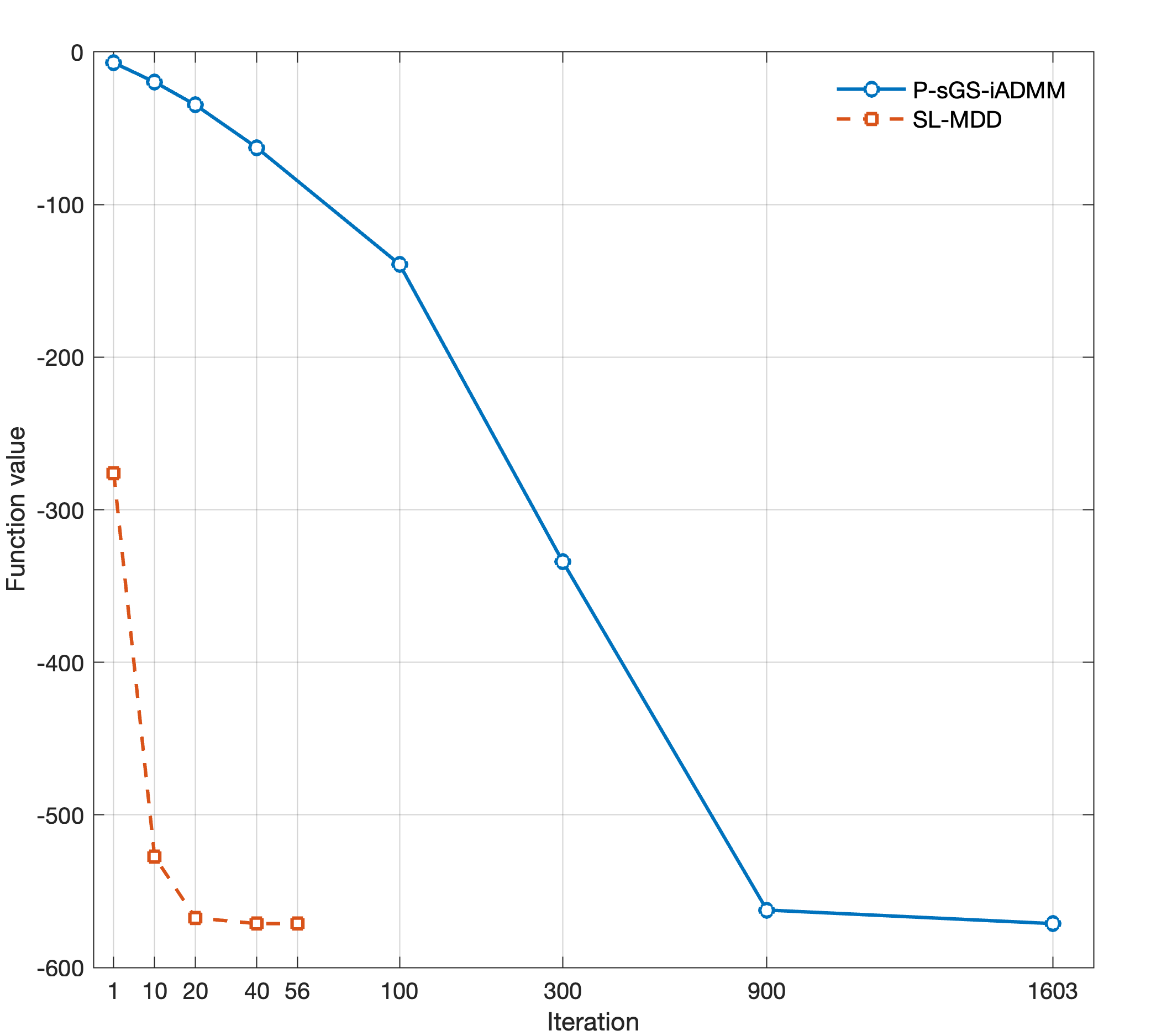}
    \\[2pt]
    \footnotesize(a) Function value vs iteration
  \end{minipage}
  \hfill
  \begin{minipage}{0.32\textwidth}
    \centering
    \includegraphics[width=\linewidth, height=0.8\linewidth]{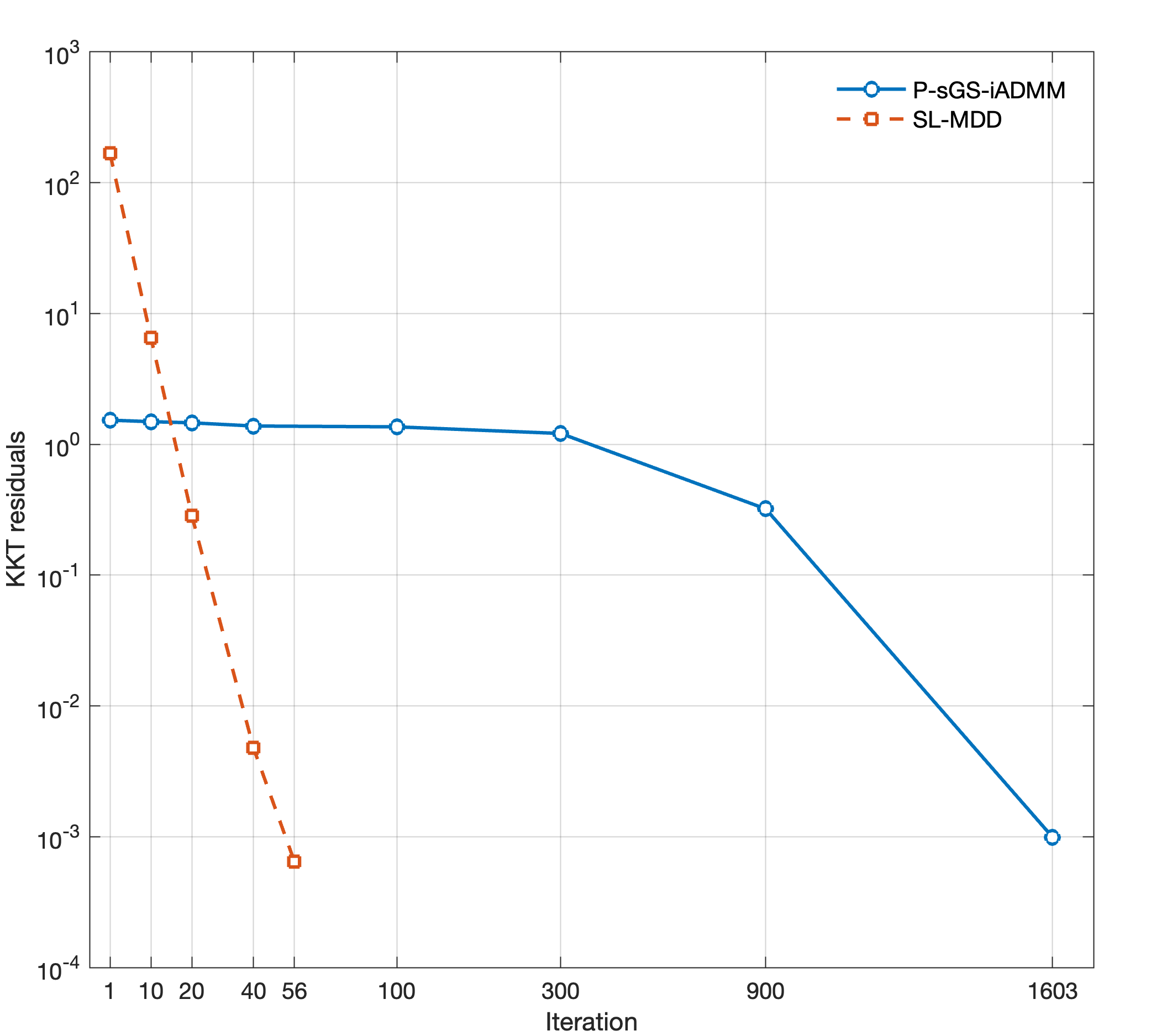}
    \\[4pt]
    \footnotesize(b) KKT residuals vs iteration
    \end{minipage}
  \hfill
  \begin{minipage}{0.32\textwidth}
    \centering
    \includegraphics[width=\linewidth, height=0.8\linewidth]{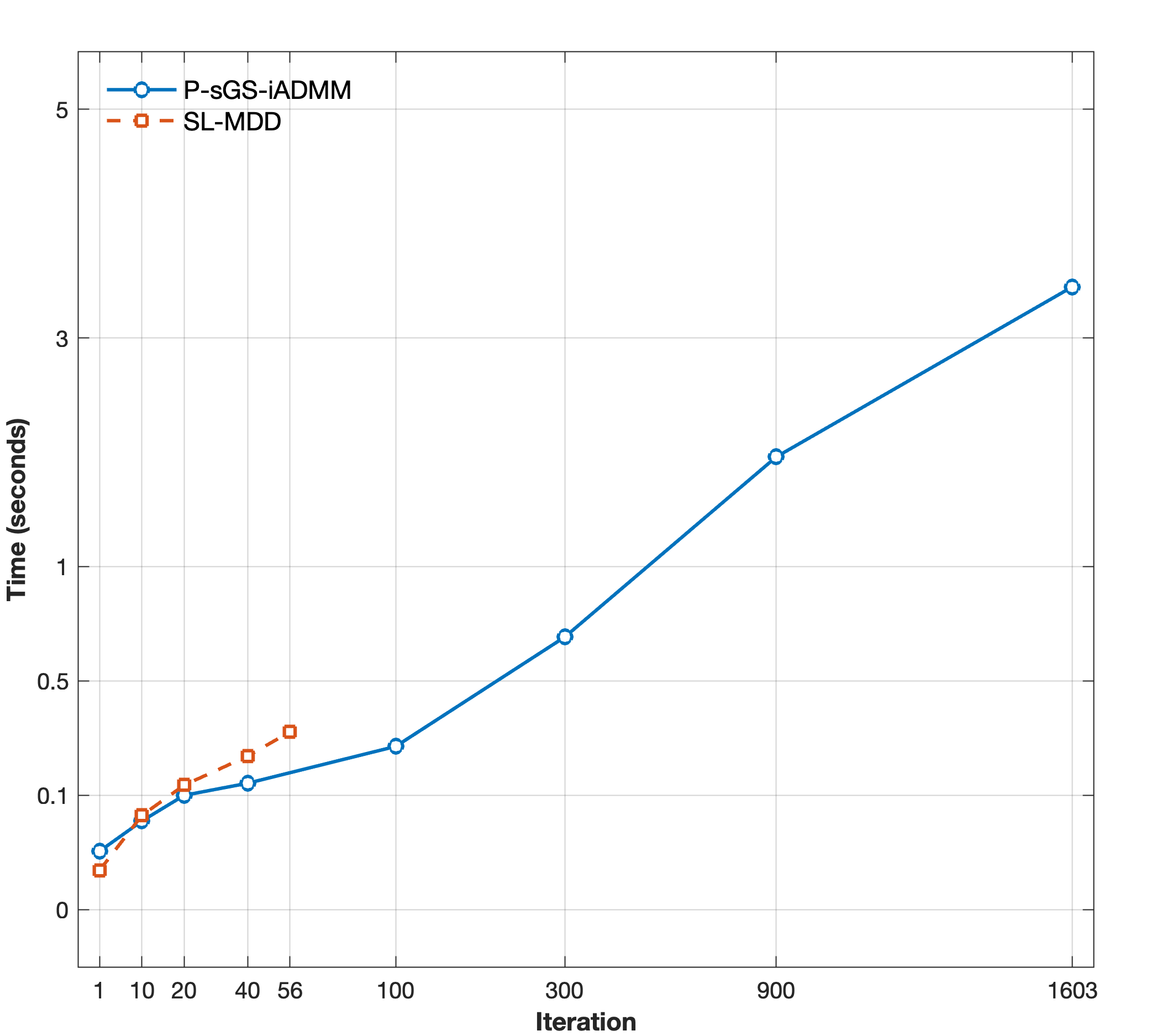}
    \\[2pt]
    \footnotesize(c) Parallel time vs iteration
  \end{minipage}
  \caption{Instance 3 (Experiment 1): $\tilde{d}=400, T=3$}
  \label{fig:Instance3}
\end{figure}
\begin{table}[htbp]
\footnotesize
\caption{Performance comparison on Experiment 1, Group 1}\label{tab:group1}
\centering
\begin{tabular}{ccccccccc}
\toprule
\multirow{2}{*}{Instances} & \multirow{2}{*}{$T$} & \multirow{2}{*}{$\tilde{d}$} & \multicolumn{2}{c}{Iter} & \multicolumn{2}{c}{Time (s)} & \multicolumn{2}{c}{Parallel Time (s)} \\
\cmidrule(r){4-5} \cmidrule(r){6-7} \cmidrule(r){8-9}
 & & & D & P & D & P & D & P \\
\midrule
        1 & 3 & 5    & 52    & 151   & 1.53      & 2.66     & 0.02      & 0.14      \\
        2 & 3 & 200  & 55    & 1280  & 5.73    & 146.97     & 0.19      & 2.60     \\
        3 & 3 & 400  & 56    & 1603  & 9.64   & 387.20    & 0.32     & 3.45     \\
  \bottomrule
\end{tabular}
\end{table}

\begin{figure}[htbp]
  \centering
  \begin{minipage}{0.32\textwidth}
    \centering
    \includegraphics[width=\linewidth, height=0.8\linewidth]{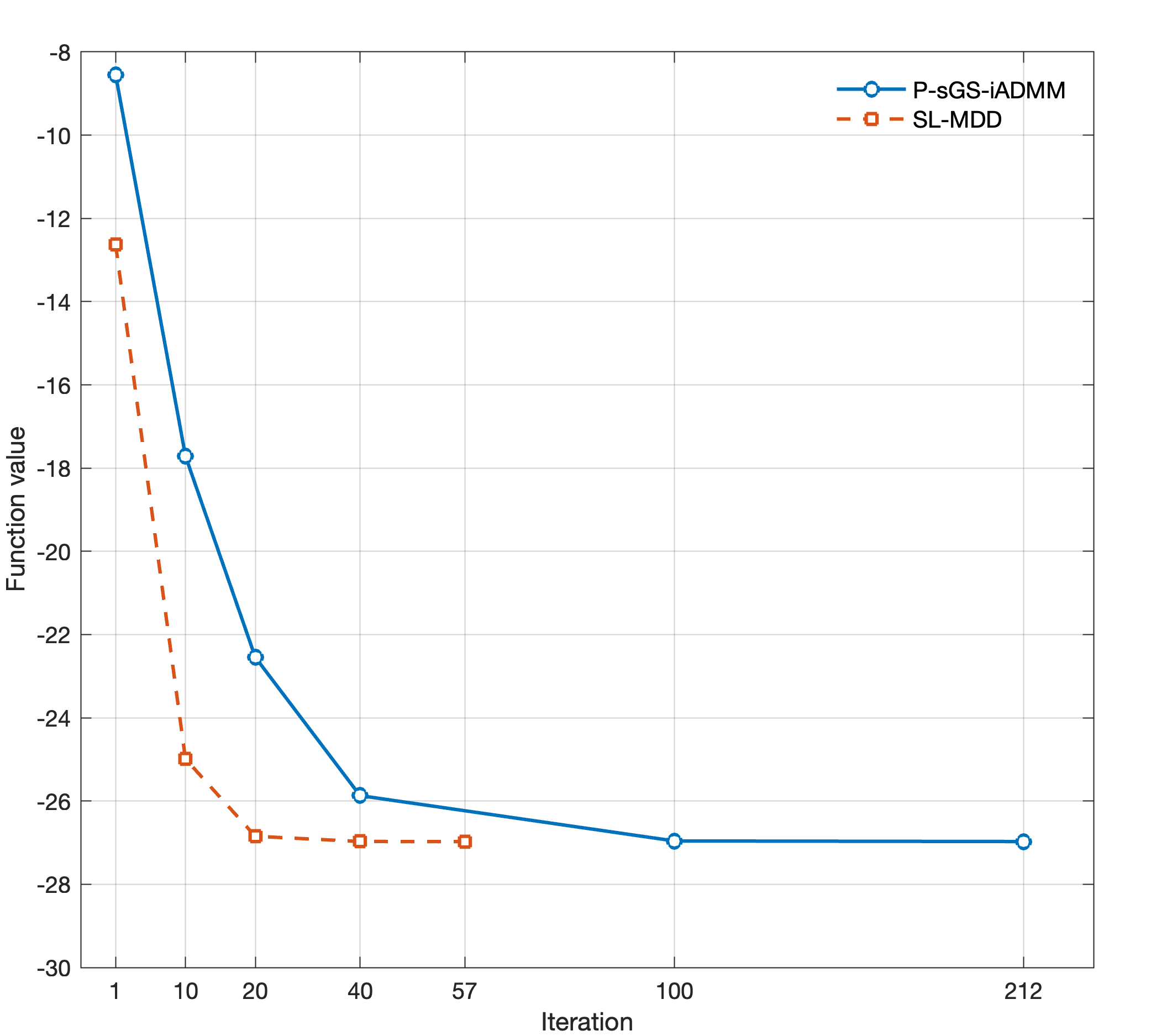}
    \\[2pt]
    \footnotesize(a) Function value vs iteration
  \end{minipage}
  \hfill
  \begin{minipage}{0.32\textwidth}
    \centering
    \includegraphics[width=\linewidth, height=0.8\linewidth]{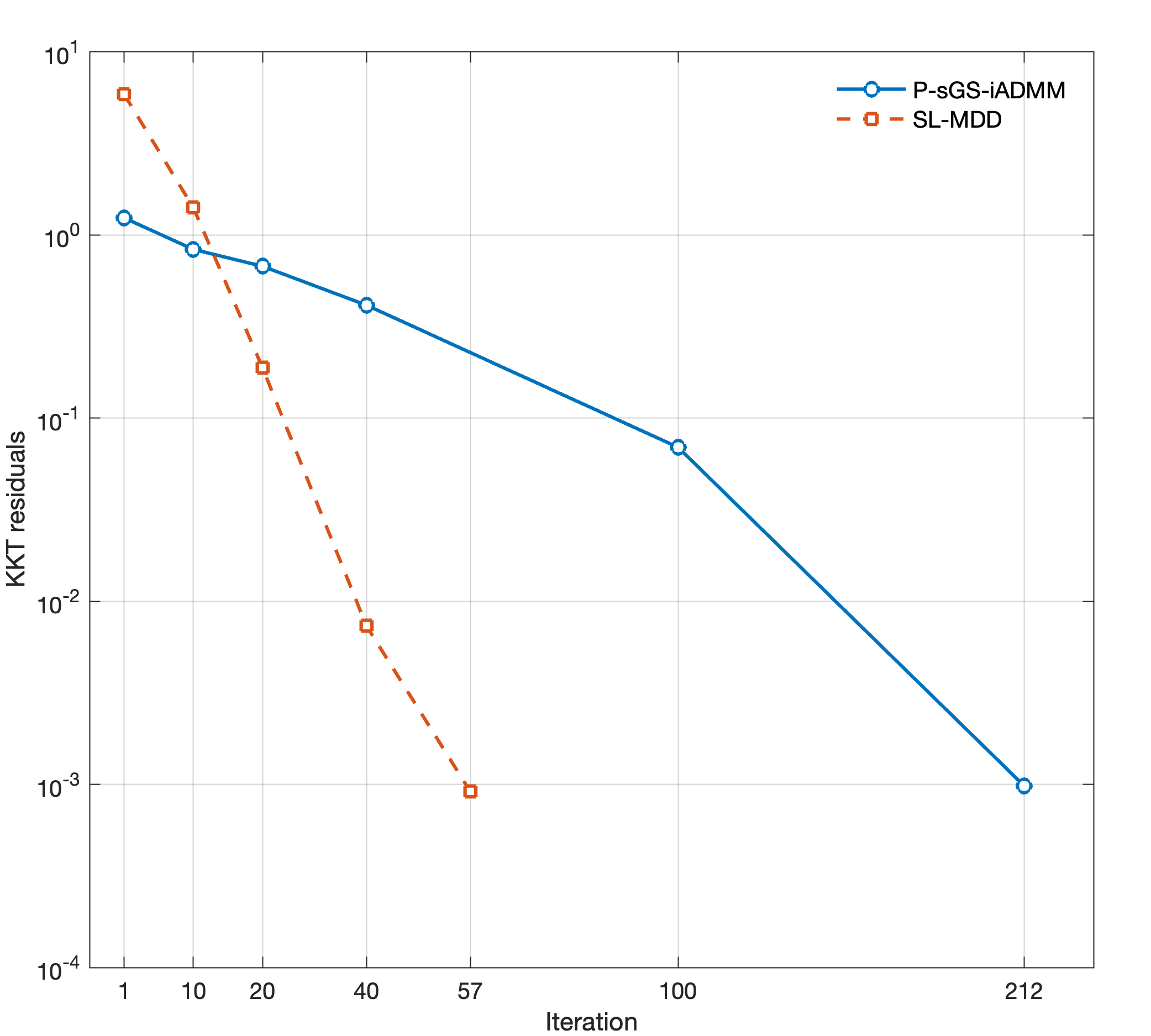}
    \\[2pt]
    \footnotesize(b) KKT residuals vs iteration
  \end{minipage}
  \hfill
  \begin{minipage}{0.32\textwidth}
    \centering
    \includegraphics[width=\linewidth, height=0.8\linewidth]{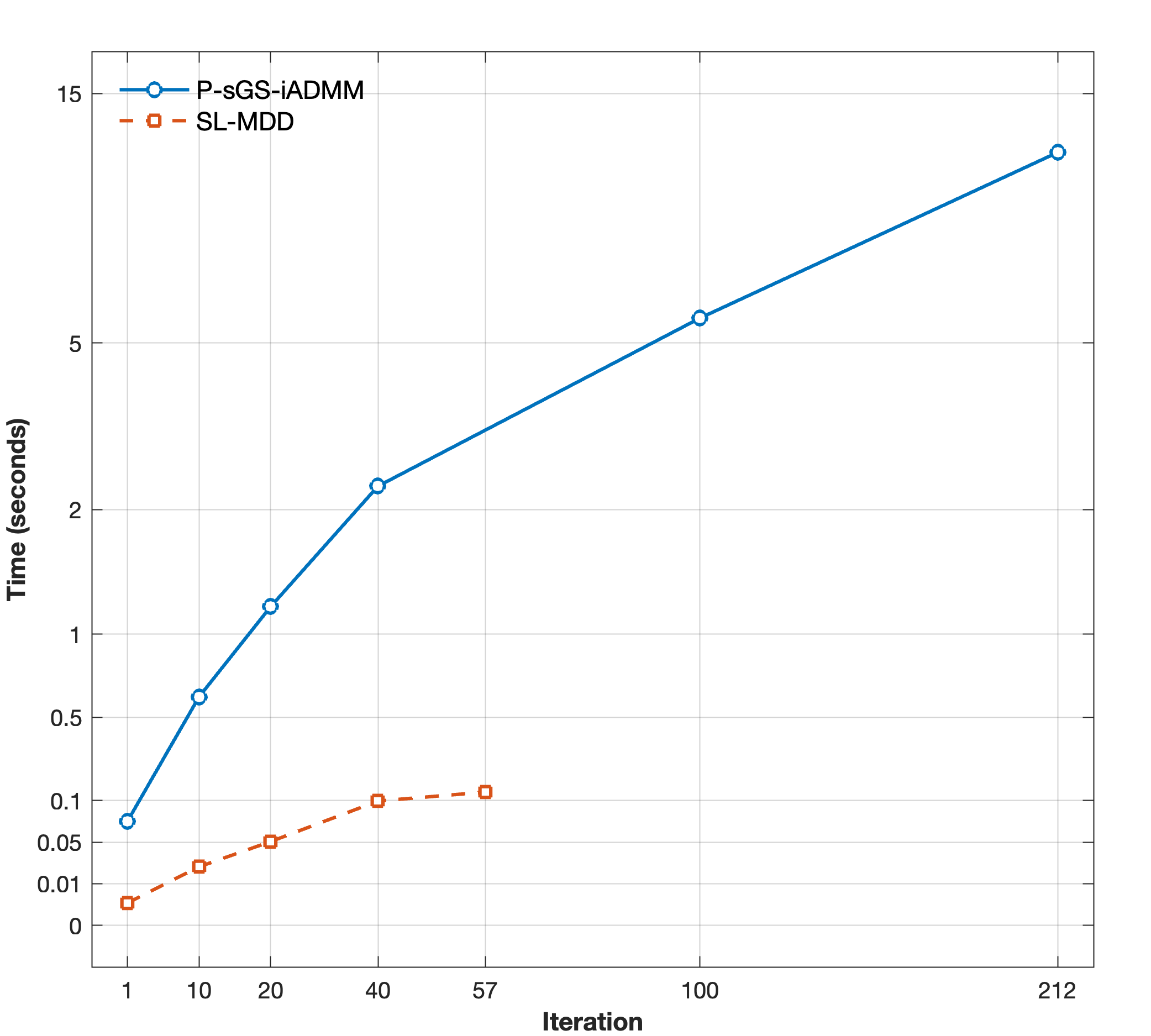}
    \\[2pt]
    \footnotesize(c) Parallel time vs iteration
  \end{minipage}
  \caption{Instance 4 (Experiment 1): $\tilde{d}=5, T=4$}
  \label{fig:Instance4}
\end{figure}

\begin{figure}[htbp]
  \centering
  \begin{minipage}{0.32\textwidth}
    \centering
    \includegraphics[width=\linewidth, height=0.8\linewidth]{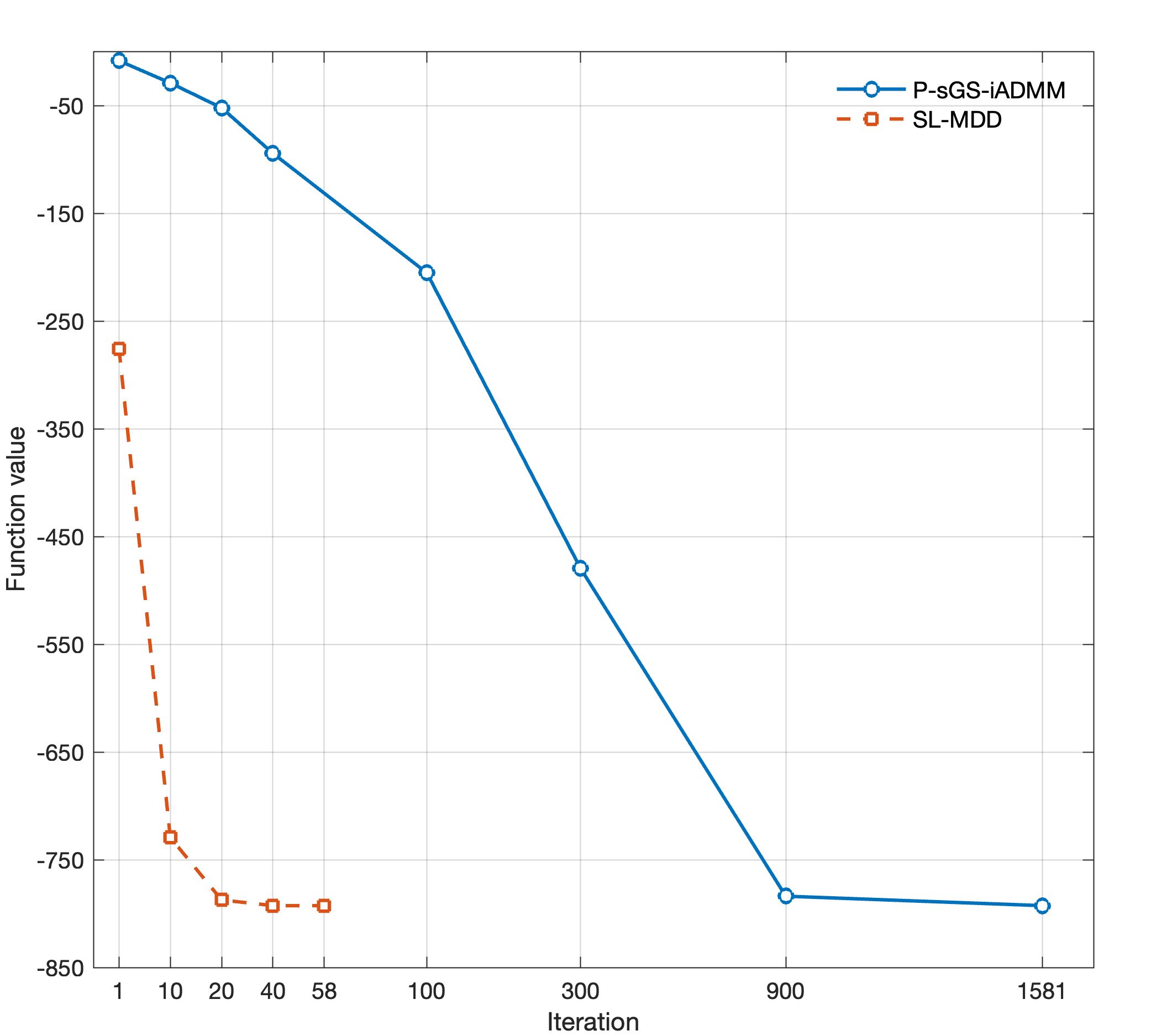}
    \\[2pt]
    \footnotesize(a) Function value vs iteration
  \end{minipage}
  \hfill
  \begin{minipage}{0.32\textwidth}
    \centering
    \includegraphics[width=\linewidth, height=0.8\linewidth]{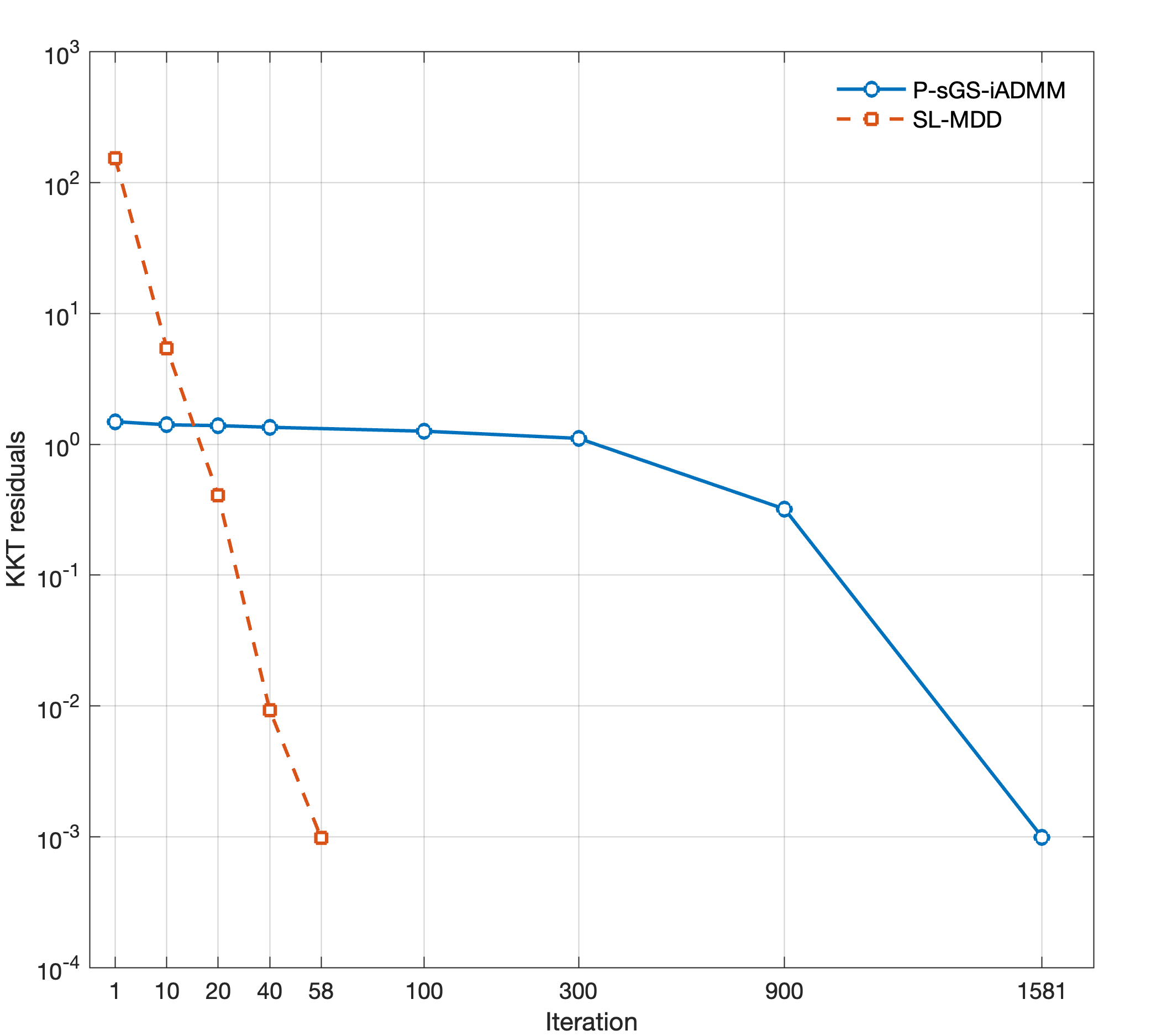}
    \\[2pt]
    \footnotesize(b) KKT residuals vs iteration
  \end{minipage}
  \hfill
  \begin{minipage}{0.32\textwidth}
    \centering
    \includegraphics[width=\linewidth, height=0.8\linewidth]{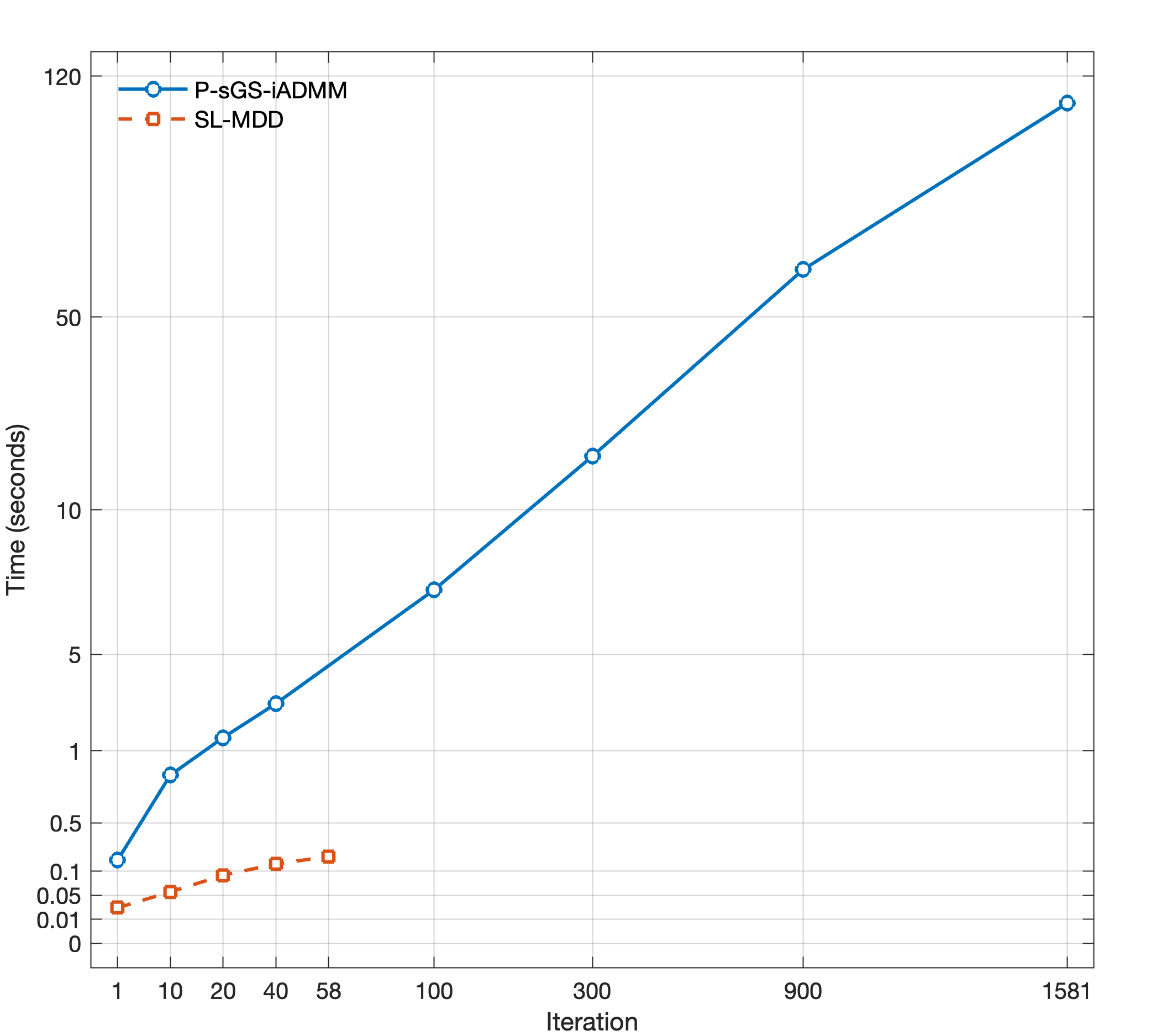}
    \\[2pt]
    \footnotesize(c) Parallel time vs iteration
  \end{minipage}
  \caption{Instance 5 (Experiment 1): $\tilde{d}=200, T=4$}
  \label{fig:Instance5}
\end{figure}
\begin{figure}[htbp]
  \centering
  \begin{minipage}{0.32\textwidth}
    \centering
    \includegraphics[width=\linewidth, height=0.8\linewidth]{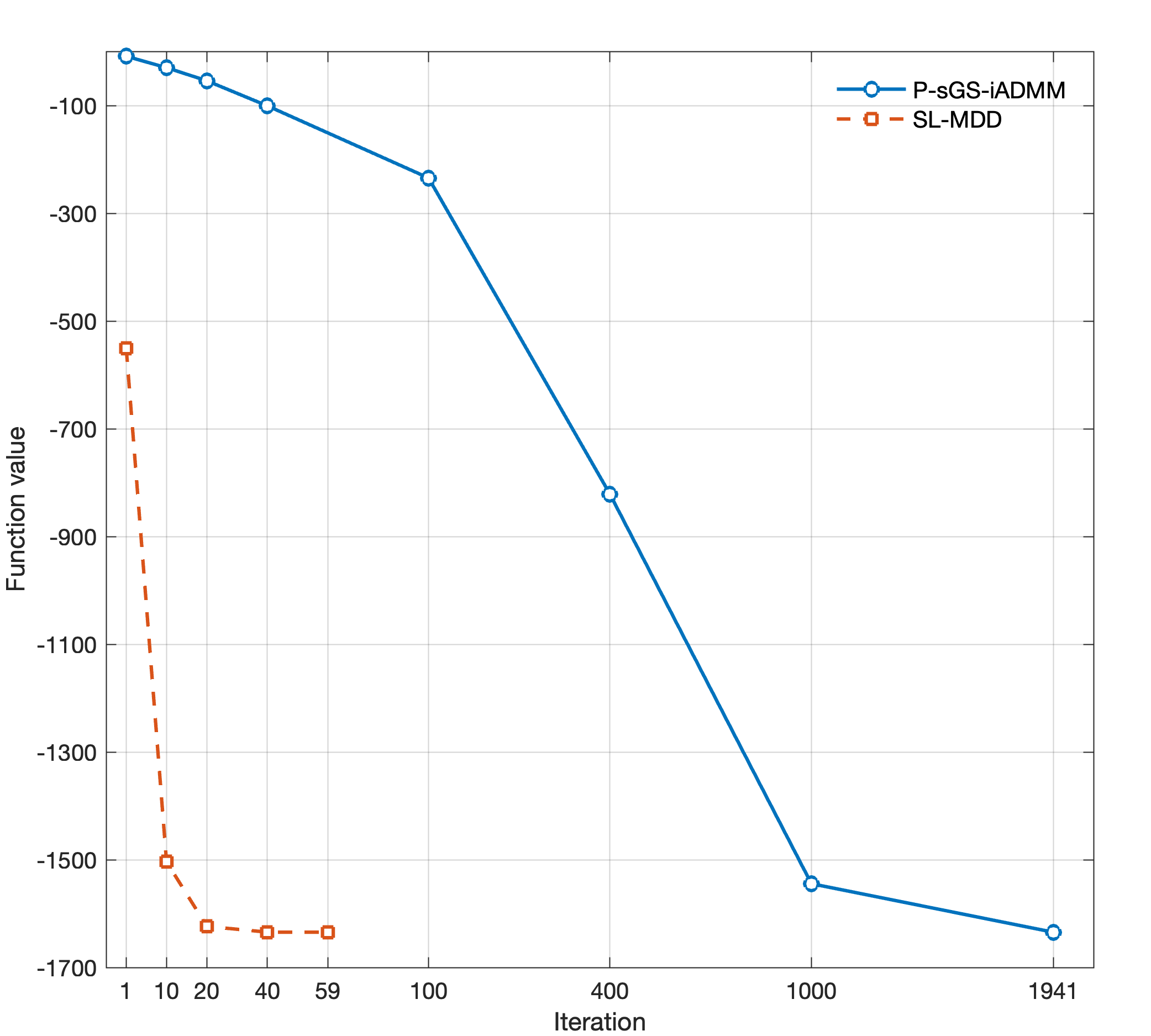}
    \\[2pt]
    \footnotesize(a) Function value vs iteration
  \end{minipage}
  \hfill
  \begin{minipage}{0.32\textwidth}
    \centering
    \includegraphics[width=\linewidth, height=0.8\linewidth]{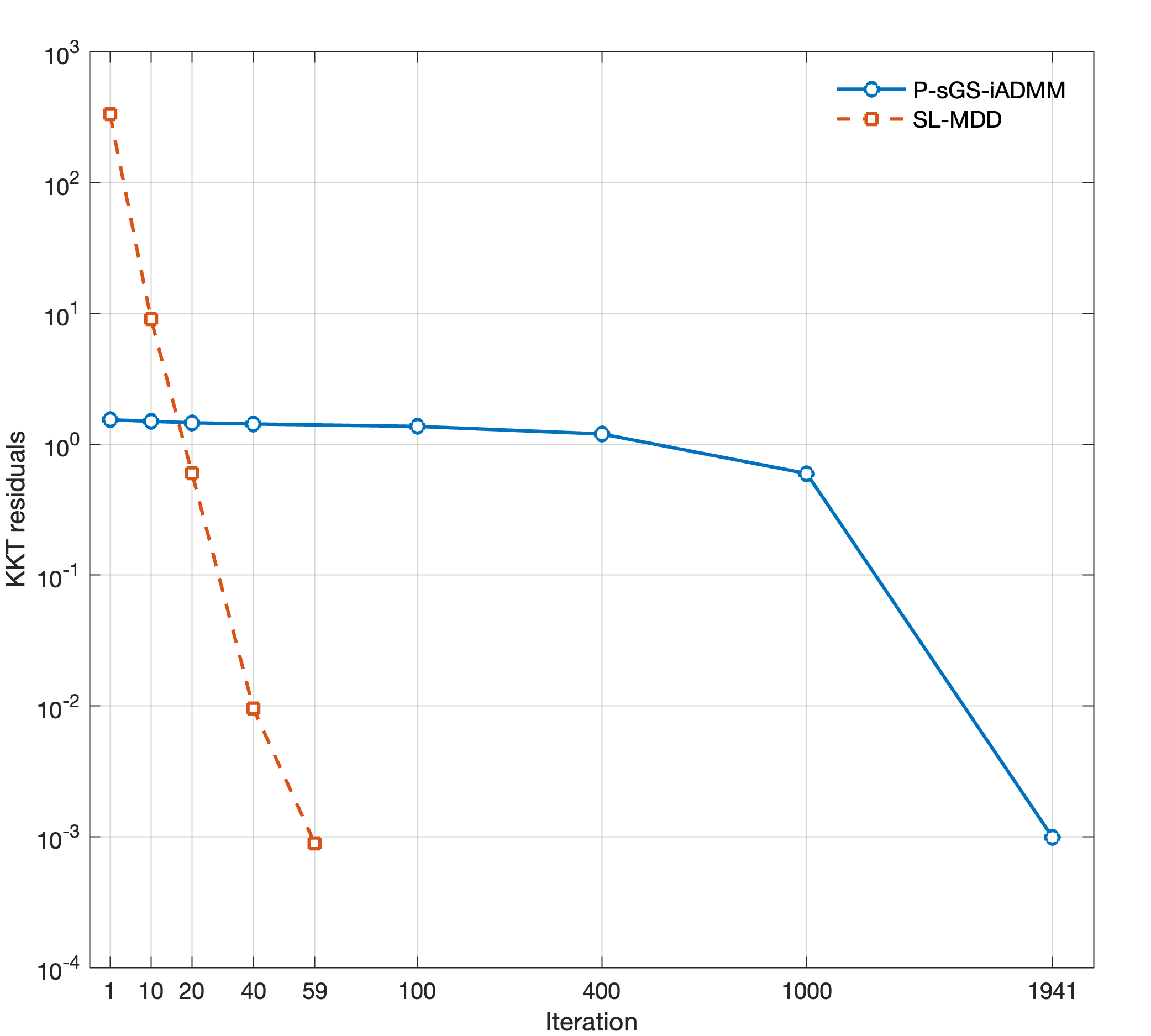}
    \\[2pt]
    \footnotesize(b) KKT residuals vs iteration
  \end{minipage}
  \hfill
  \begin{minipage}{0.32\textwidth}
    \centering
    \includegraphics[width=\linewidth, height=0.8\linewidth]{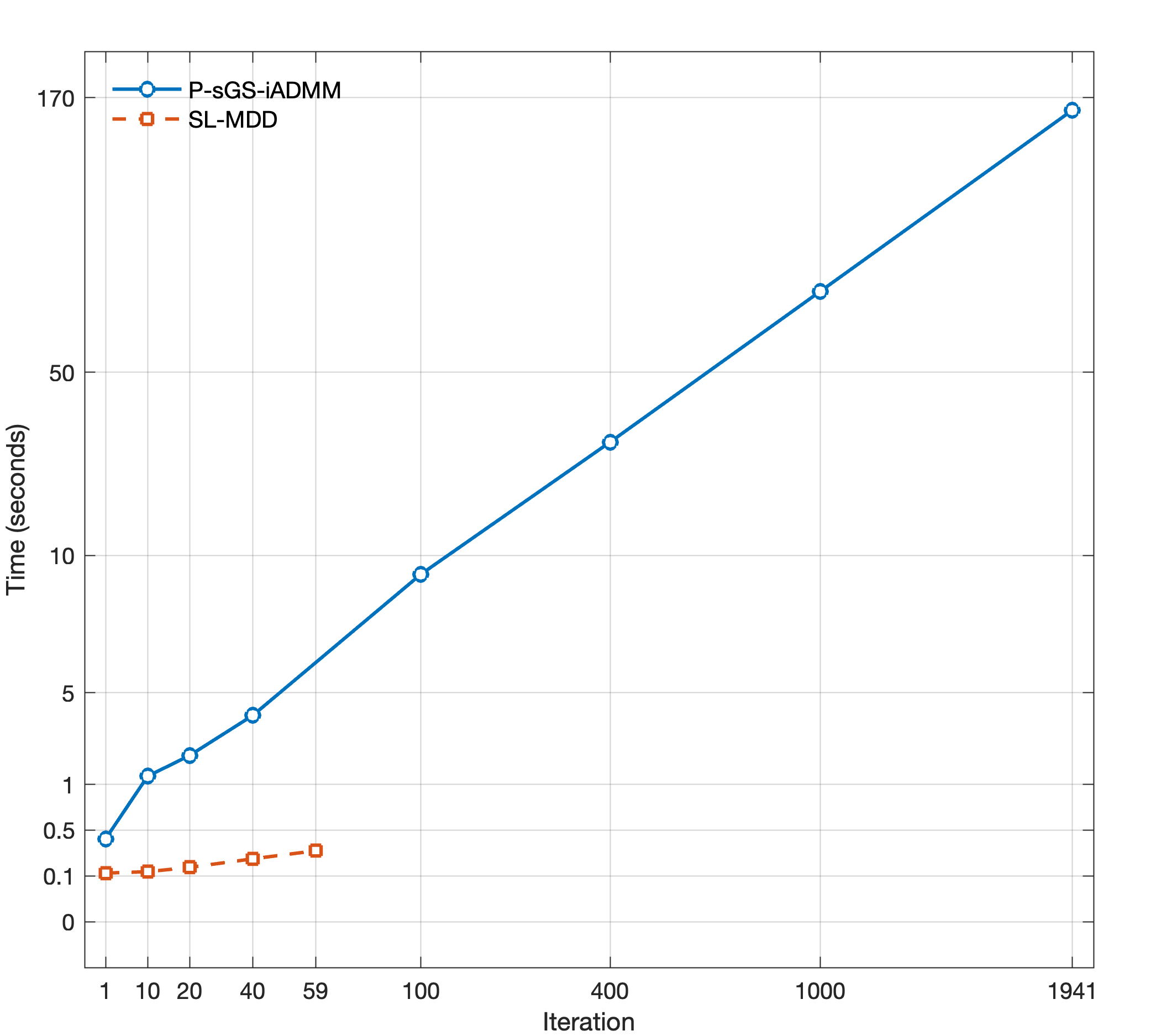}
    \\[2pt]
    \footnotesize(c) Parallel time vs iteration
  \end{minipage}
  \caption{Instance 6 (Experiment 1): $\tilde{d}=400, T=4$}
  \label{fig:Instance6}
\end{figure}

\begin{table}[htbp]
\footnotesize
\caption{Performance comparison on Experiment 1, Group 2}\label{tab:group2}
\centering
\begin{tabular}{ccccccccc}
\toprule
\multirow{2}{*}{Instances} & \multirow{2}{*}{$T$} & \multirow{2}{*}{$\tilde{d}$} & \multicolumn{2}{c}{Iter} & \multicolumn{2}{c}{Time (s)} & \multicolumn{2}{c}{Parallel Time (s)} \\
\cmidrule(r){4-5} \cmidrule(r){6-7} \cmidrule(r){8-9}
 & & & D & P & D & P & D & P \\
\midrule
        4 & 4 & 5    & 57    & 212   & 136.36   & 356.14   & 0.14     & 12.64      \\
        5 & 4 & 200  & 58    & 1581  & 281.59  & 8940.11 & 0.22     & 112.11    \\
        6 & 4 & 400  & 59    & 1941  & 466.25 & 21630.31 & 0.42   & 164.61  \\
    \bottomrule
\end{tabular}
\end{table}

\begin{figure}[htbp]
  \centering
  \begin{minipage}{0.32\textwidth}
    \centering
    \includegraphics[width=\linewidth, height=0.8\linewidth]{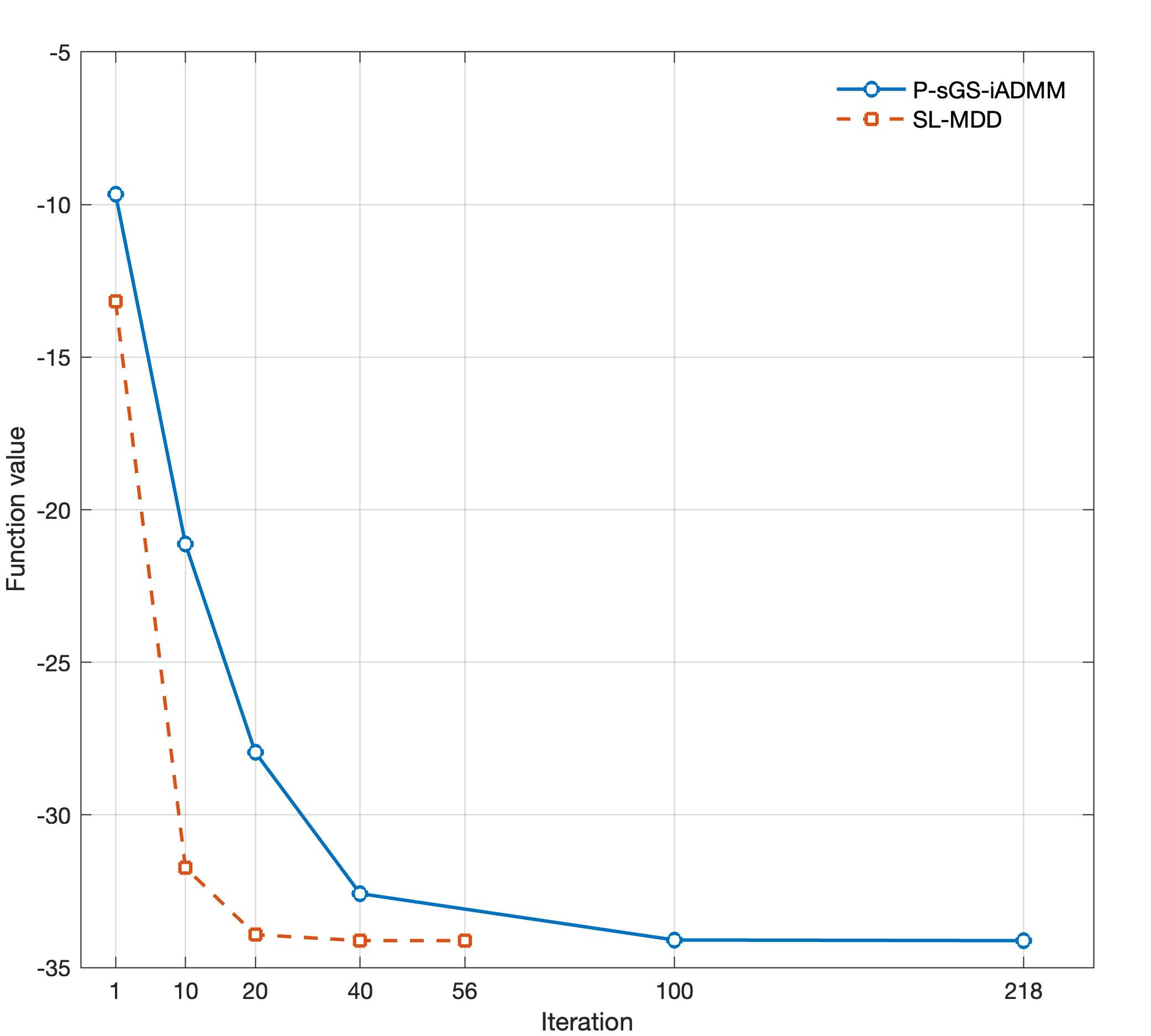}
    \\[2pt]
    \footnotesize(a) Function value vs iteration
              \end{minipage}
  \hfill
  \begin{minipage}{0.32\textwidth}
    \centering
    \includegraphics[width=\linewidth, height=0.8\linewidth]{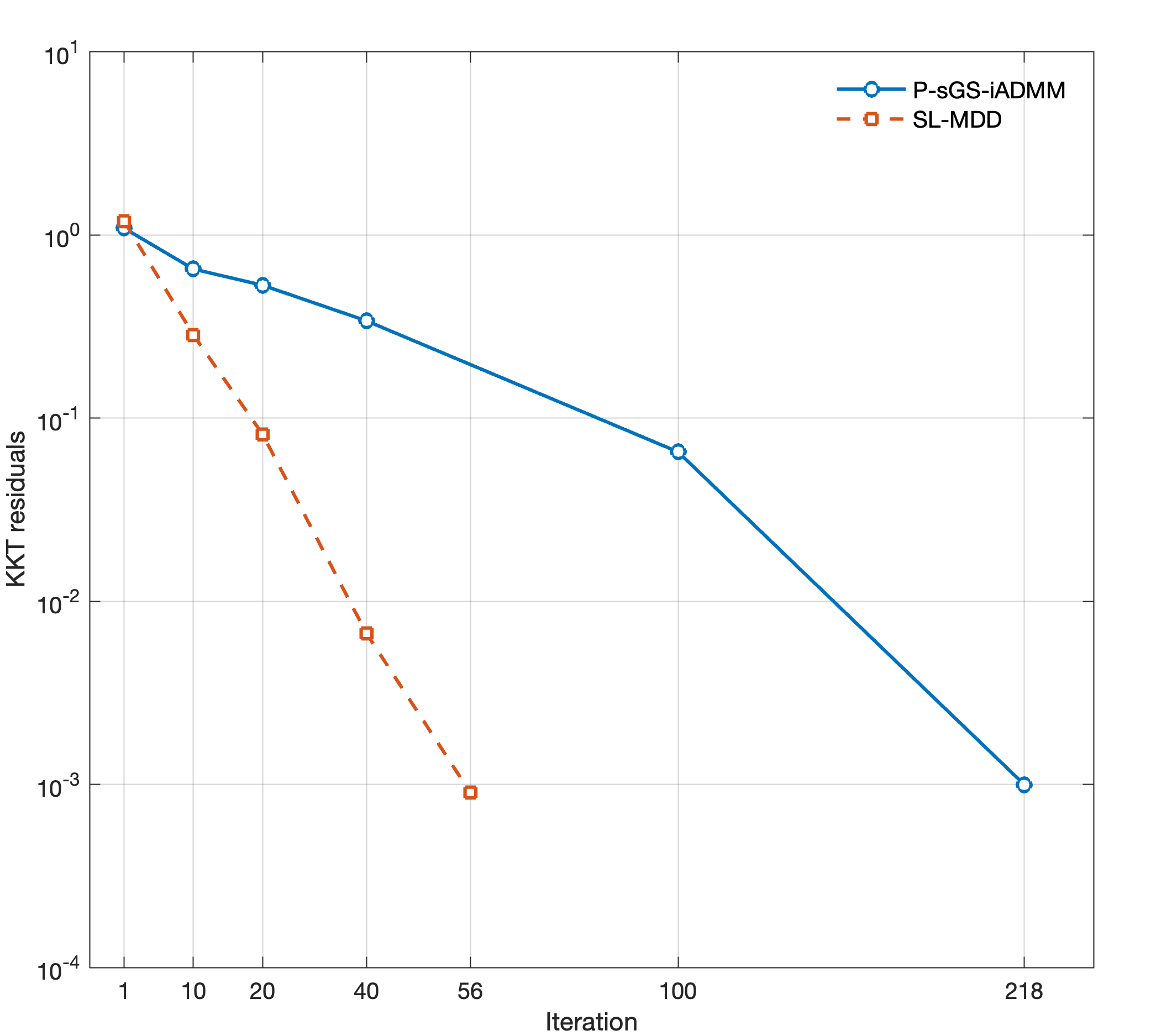}
    \\[2pt]
    \footnotesize(b) KKT residuals vs iteration
  \end{minipage}
  \hfill
  \begin{minipage}{0.32\textwidth}
    \centering
    \includegraphics[width=\linewidth, height=0.8\linewidth]{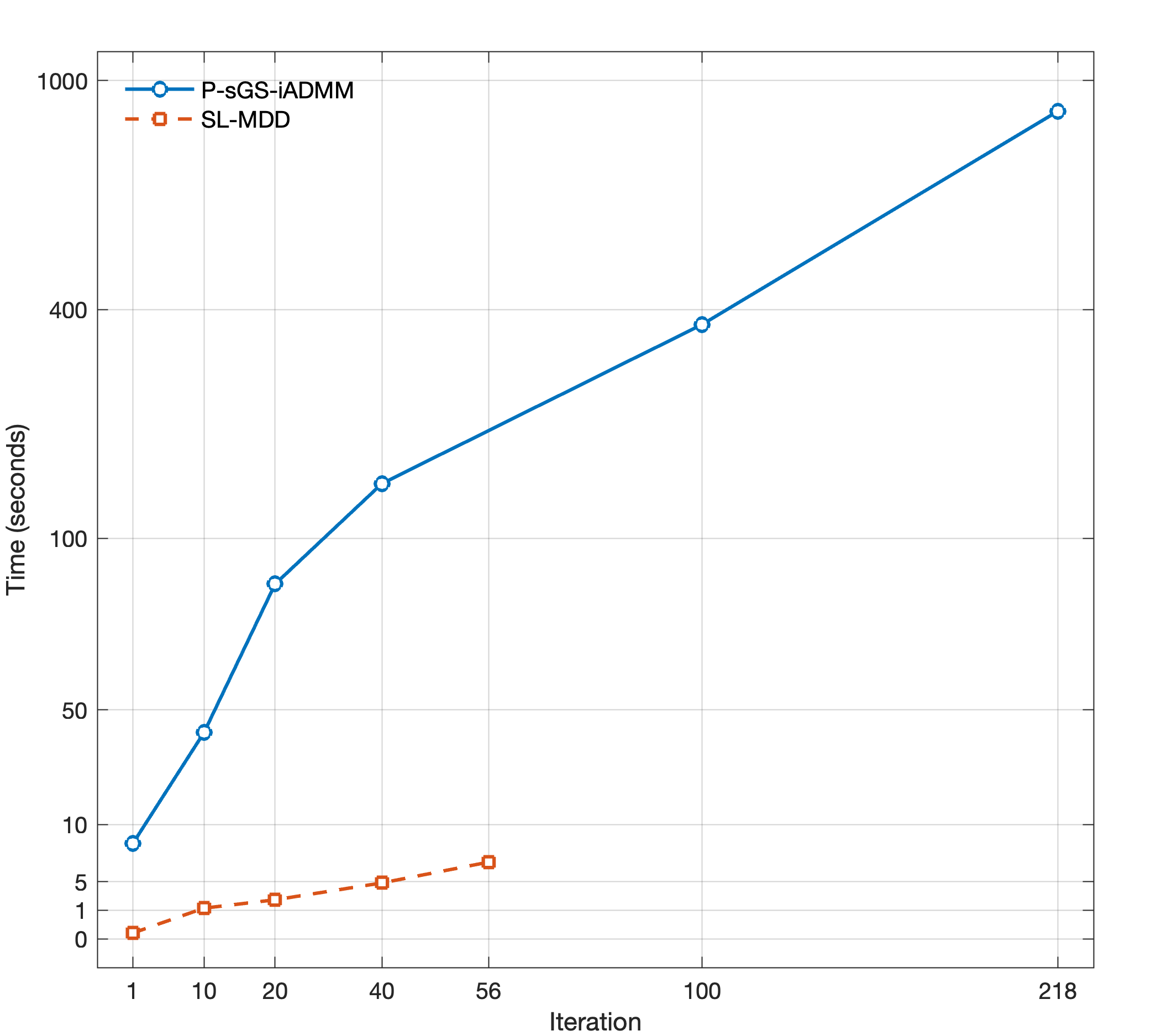}
    \\[2pt]
    \footnotesize(c) Parallel time vs iteration
  \end{minipage}
  \caption{Instance 7 (Experiment 1): $\tilde{d}=5, T=5$}
  \label{fig:Instance7}
\end{figure}

\begin{table}[htbp]
\footnotesize
\caption{Performance comparison on Experiment 1, Group 3}\label{tab:group3}
\centering
\begin{tabular}{ccccccccc}
\toprule
\multirow{2}{*}{Instances} & \multirow{2}{*}{$T$} & \multirow{2}{*}{$\tilde{d}$} & \multicolumn{2}{c}{Iter} & \multicolumn{2}{c}{Time (s)} & \multicolumn{2}{c}{Parallel Time (s)} \\
\cmidrule(r){4-5} \cmidrule(r){6-7} \cmidrule(r){8-9}
 & & & D & P & D & P & D & P \\
\midrule
        1 & 3 & 5    & 52    & 151   & 1.53      & 2.66     & 0.02      & 0.14      \\
        4 & 4 & 5    & 57    & 212   & 136.36   & 356.14   & 0.14     & 12.64      \\
        7 & 5 & 5    & 56    & 218   & 7425.55  & 73910.83 & 6.71      &919.79   \\
  \bottomrule
\end{tabular}
\end{table}

We begin by analyzing the evolution of function value and KKT residuals w.r.t. the iteration. As shown in Figures \ref{fig:Instance1}--\ref{fig:Instance7}, 
the proposed SL-MDD method reaches the prescribed termination criterion and satisfies the KKT conditions in substantially fewer iterations than P-sGS-iADMM.


	Regarding the estimated parallel computation time, the results in Tables \ref{tab:group1}--\ref{tab:group3} (where D denotes SL-MDD method and P denotes P-sGS-iADMM) compare the scalability of the two methods. For $T=3$ (Figures \ref{fig:Instance1}--\ref{fig:Instance3} and Table \ref{tab:group1}), as the variable dimension $\tilde{d}$ increases from $5$ to $200$ and $400$, the runtime for D grows gradually from $0.02$ s to $0.19$ s and $0.32$ s, whereas P increases from $0.14$ s to $2.60$ s and $3.45$ s. This performance difference becomes more apparent for $T=4$ (Figures \ref{fig:Instance4}--\ref{fig:Instance6} and Table \ref{tab:group2}); under the same dimensional scaling, the runtime for D remains efficient, increasing from $0.14$ s to $0.22$ s and $0.42$ s, while the time for P extends from $12.64$ s to $112.11$ s and $164.61$ s. Additionally, when fixing $\tilde{d}=5$ and varying $T \in \{3, 4, 5\}$ (Figures \ref{fig:Instance1},\ref{fig:Instance4},\ref{fig:Instance7} and Table \ref{tab:group3}), the runtime for D increases from $0.02$~s to $0.14$ s and $6.71$ s, compared to $0.14$ s, $12.64$ s and $919.79$ s for P.

Overall, Tables~\ref{tab:group1}--\ref{tab:group3} indicate that the SL-MDD method is more efficient than P-sGS-iADMM in both iteration count and runtime, with particularly pronounced improvements in higher-dimensional instances and multi-stage settings.


\subsection{MSP with exponential-$\ell_1$ composite objective}
Next, we consider problem \eqref{MSP} with an exponential-$\ell_1$ composite objective combining $\ell_1$-regularization and an exponential utility term \cite{pedersen2003utility}. For all stages $t=1,\dots,T$, we set $\theta_t\big(x_t(\xi_{[t]}^{k_{[t]}})\big)=\|x_t(\xi_{[t]}^{k_{[t]}})\|_1$, $h_t\big(x_t(\xi_{[t]}^{k_{[t]}})\big)=\exp\big(0.2\sum_{i=1}^d x_{t,i}(\xi_{[t]}^{k_{[t]}})\big)$, and enforce the box constraints $\mathcal{K}_t=\{x_t\in\mathbb{R}^d\mid \bar{l}_t\le x_t\le \bar{u}_t\}$ with $\bar{l}_t=-2$ and $\bar{u}_t=10$. 
To instantiate the problem data, the discrete scenario tree is constructed with $N_t=40$ branches per stage, where the underlying realizations are sampled via $\xi_2 \sim \mathcal{U}[0,1]$ and $\xi_t \sim \mathcal{U}[\xi_{t-1},\xi_{t-1}+1]$ for $t \geq 3$. Specifically, at stage $1$, we set $A_1 \sim \mathcal{U}[0,1]^{\ell\times d}$ and $b_1=4$. For stages $t \ge 2$, the parameters depend on scenario realizations such that $B_t(\xi_{[t]})=\xi_{[t]}B$, $A_t(\xi_{[t]})=\xi_{[t]}A$, and $b_t(\xi_{[t]})=10\,\xi_{[t]}$, where the entries of $A \in \mathbb{R}^{\ell\times d}$ and $B \in \mathbb{R}^{\ell\times d}$ are drawn i.i.d. from $\mathcal{U}[0,1]$ and $\mathcal{U}[0,2]$, respectively. Finally, we evaluate both algorithms on 7 instances of problem \eqref{MSP} using the same dimensions $d$ and stages $T$ as in Section \ref{sec:exp1}.
\begin{figure}[htbp]
  \centering
  \begin{minipage}{0.32\textwidth}
    \centering
    \includegraphics[width=\linewidth, height=0.8\linewidth]{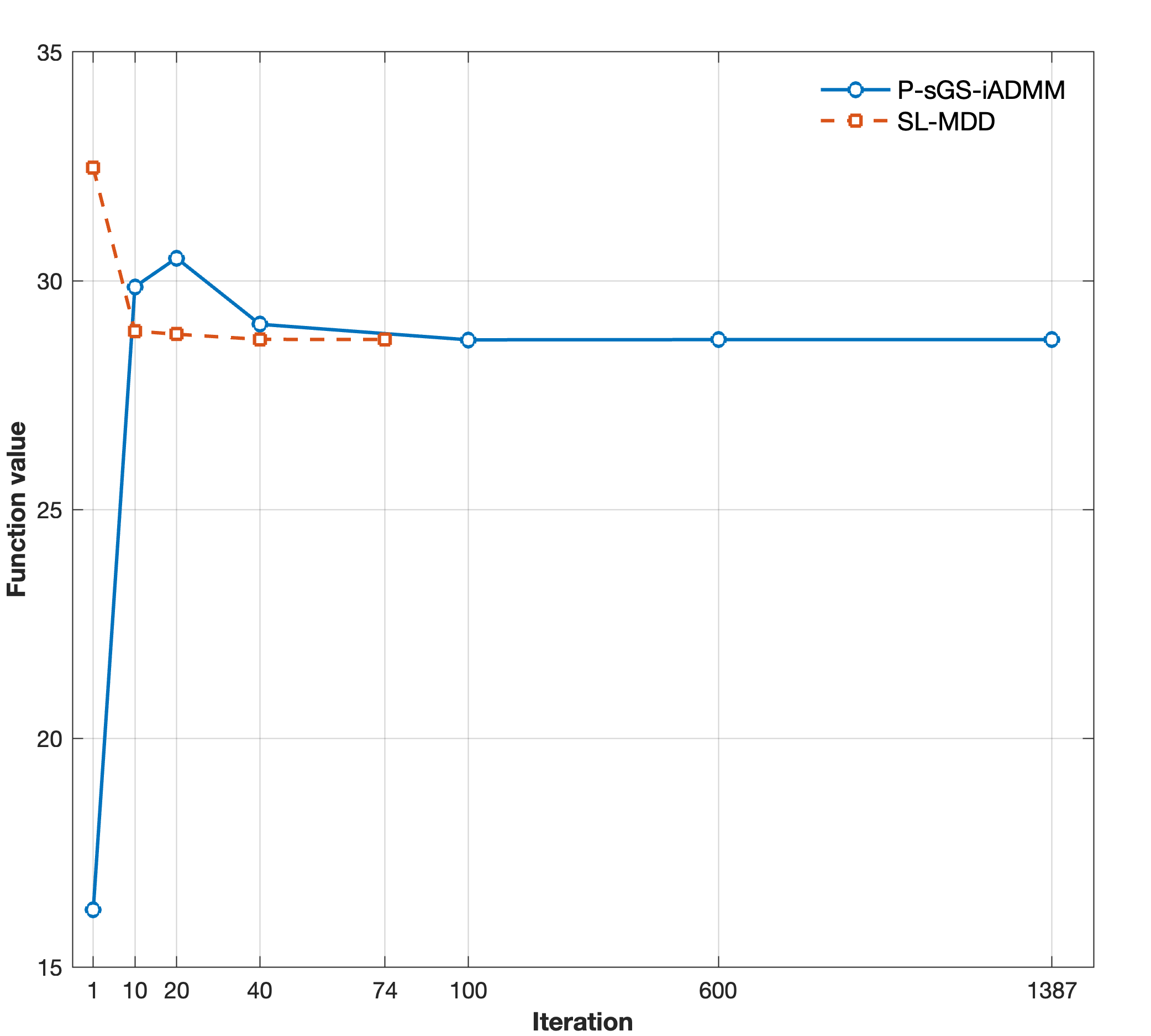}
    \\[2pt]
    \footnotesize(a) Function value vs iteration
  \end{minipage}
  \hfill
  \begin{minipage}{0.32\textwidth}
    \centering
    \includegraphics[width=\linewidth, height=0.8\linewidth]{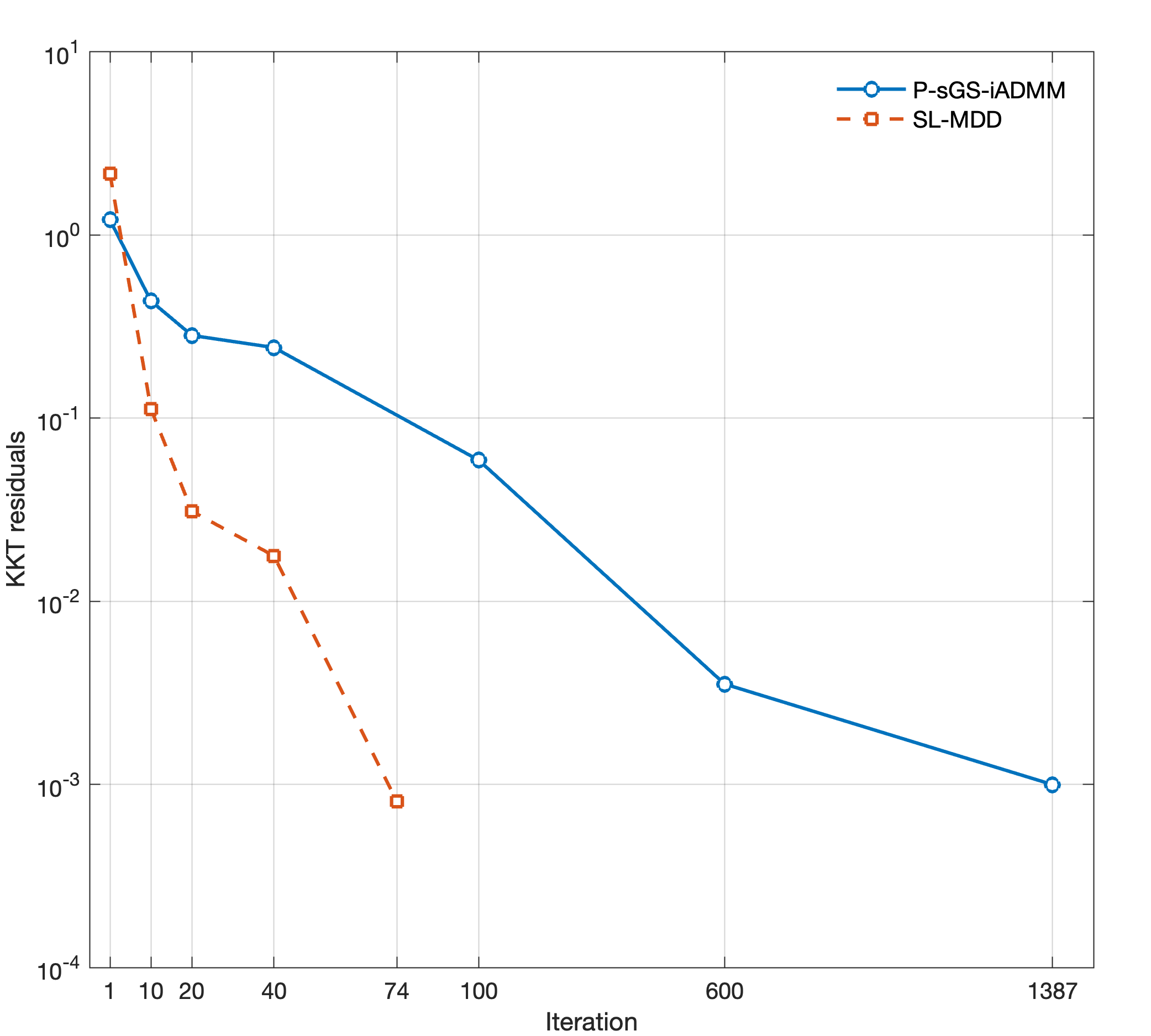}
    \\[2pt]
    \footnotesize(b) KKT residuals vs iteration
  \end{minipage}
  \hfill
  \begin{minipage}{0.32\textwidth}
    \centering
    \includegraphics[width=\linewidth, height=0.8\linewidth]{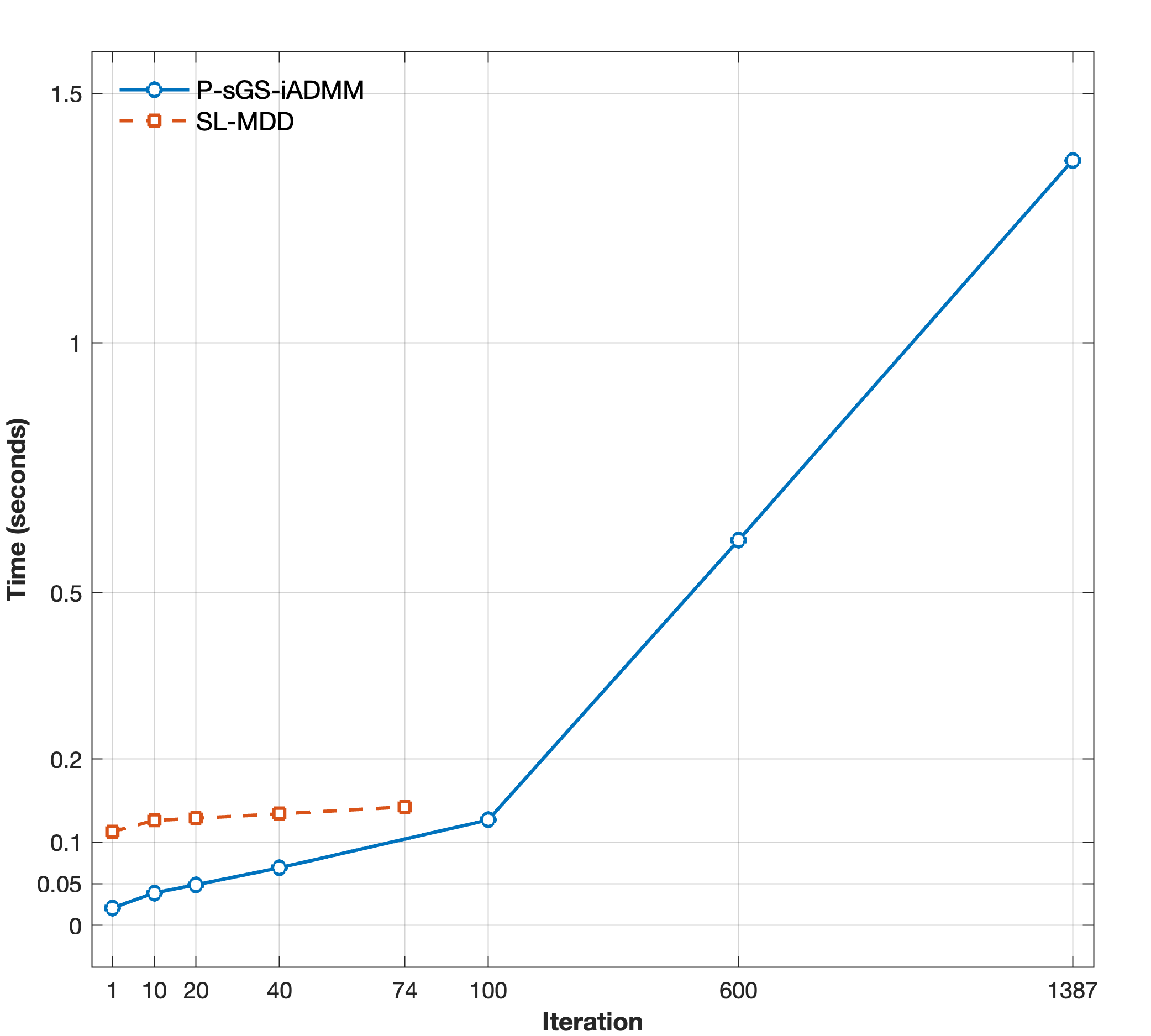}
    \\[2pt]
    \footnotesize(c) Parallel time vs iteration
  \end{minipage}
  \caption{Instance 1 (Experiment 2): $d=5, T=3$}
  \label{ffig:Instance1}
\end{figure}
\begin{figure}[htbp]
  \centering
  \begin{minipage}{0.32\textwidth}
    \centering
    \includegraphics[width=\linewidth, height=0.8\linewidth]{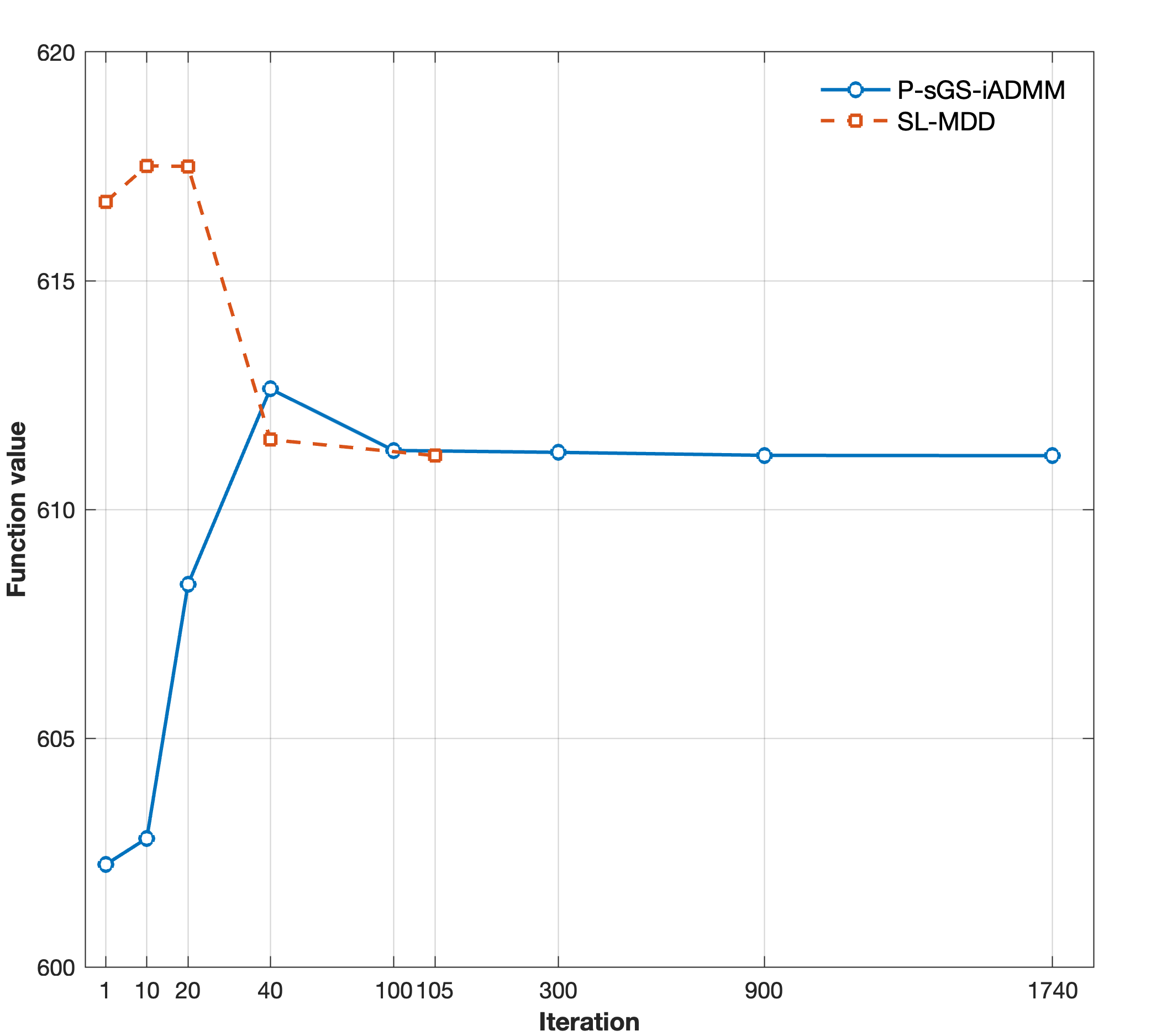}
    \\[4pt]
    \footnotesize(a) Function value vs iteration
  \end{minipage}
  \hfill
  \begin{minipage}{0.32\textwidth}
    \centering
    \includegraphics[width=\linewidth, height=0.8\linewidth]{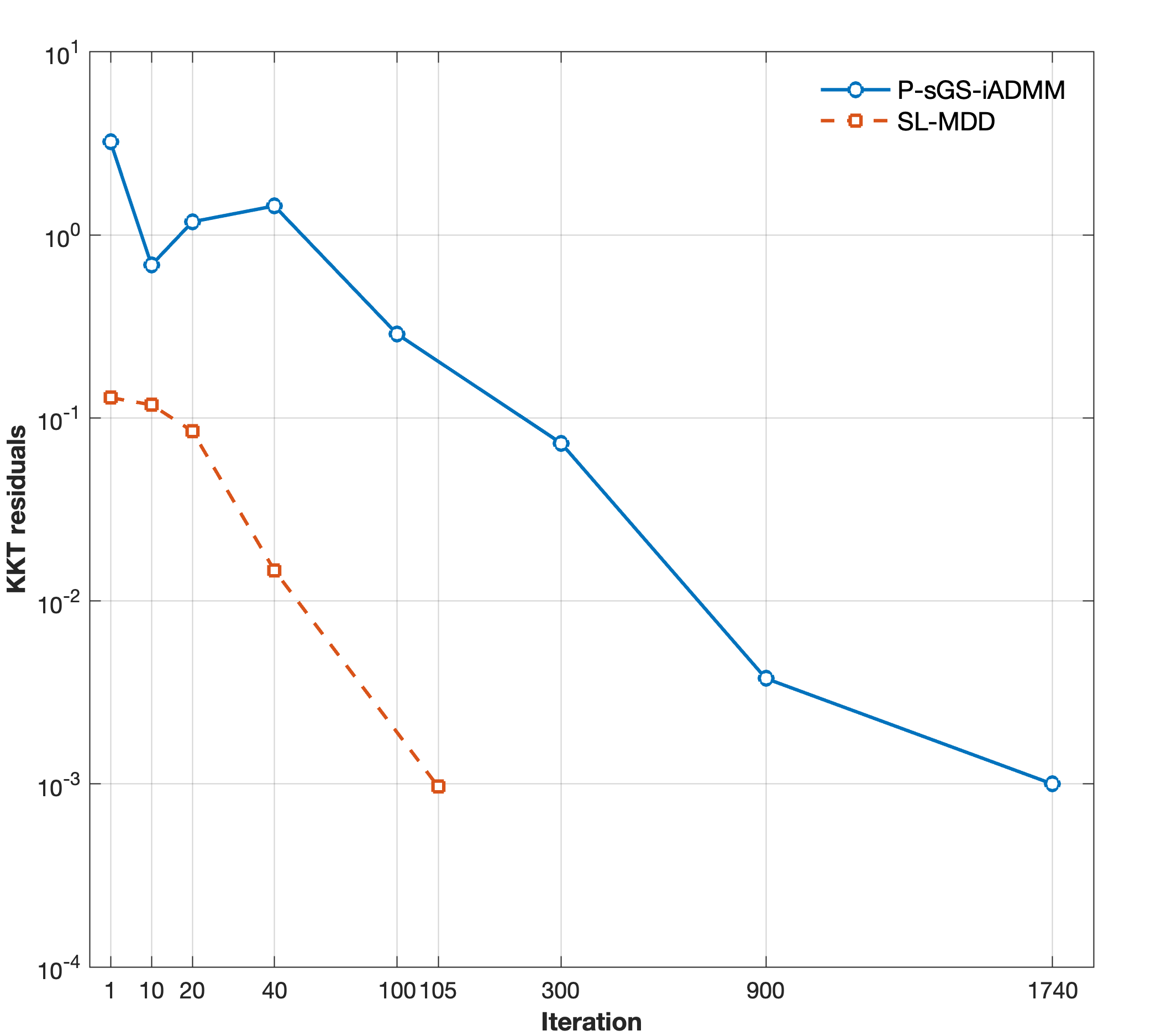}
    \\[2pt]
    \footnotesize(b) KKT residuals vs iteration
  \end{minipage}
  \hfill
  \begin{minipage}{0.32\textwidth}
    \centering
    \includegraphics[width=\linewidth, height=0.8\linewidth]{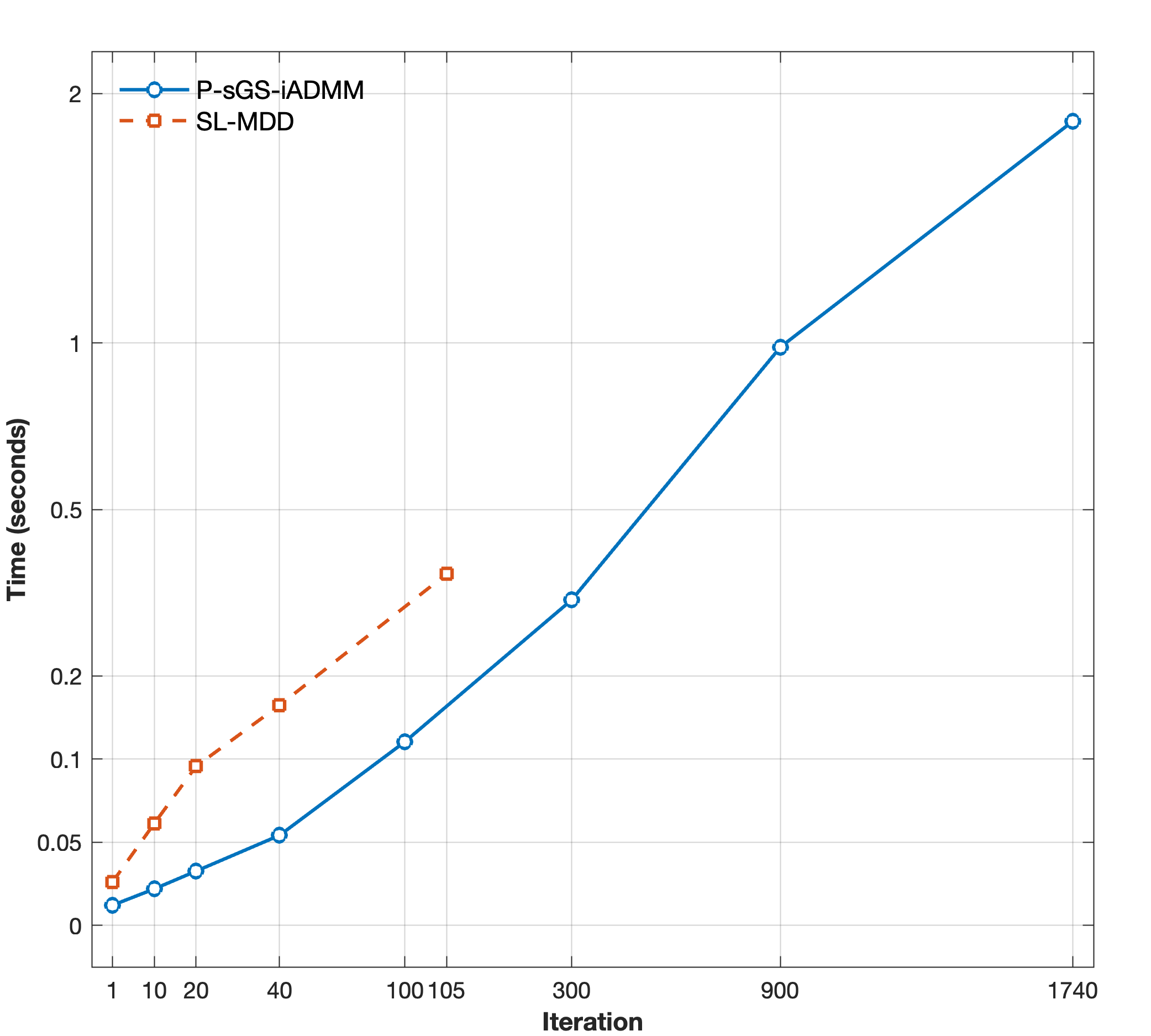}
    \\[2pt]
    \footnotesize(c) Parallel time vs iteration
  \end{minipage}
  \caption{Instance 2 (Experiment 2): $d=200, T=3$}
  \label{ffig:Instance2}
\end{figure}
\begin{figure}[htbp]
  \centering
  \begin{minipage}{0.32\textwidth}
    \centering
    \includegraphics[width=\linewidth, height=0.8\linewidth]{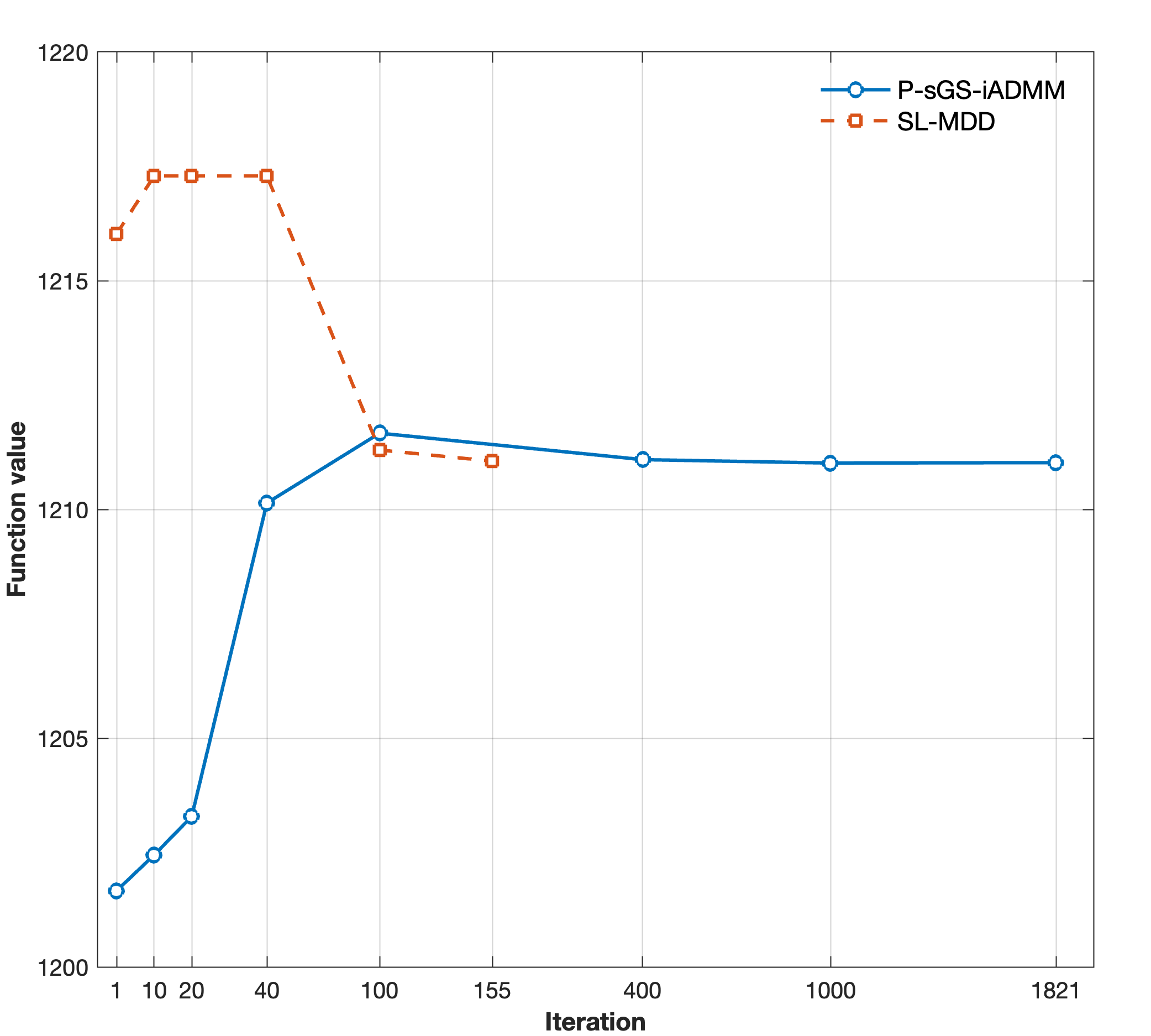}
    \\[2pt]
    \footnotesize(a) Function value vs iteration
  \end{minipage}
  \hfill
  \begin{minipage}{0.32\textwidth}
    \centering
    \includegraphics[width=\linewidth, height=0.8\linewidth]{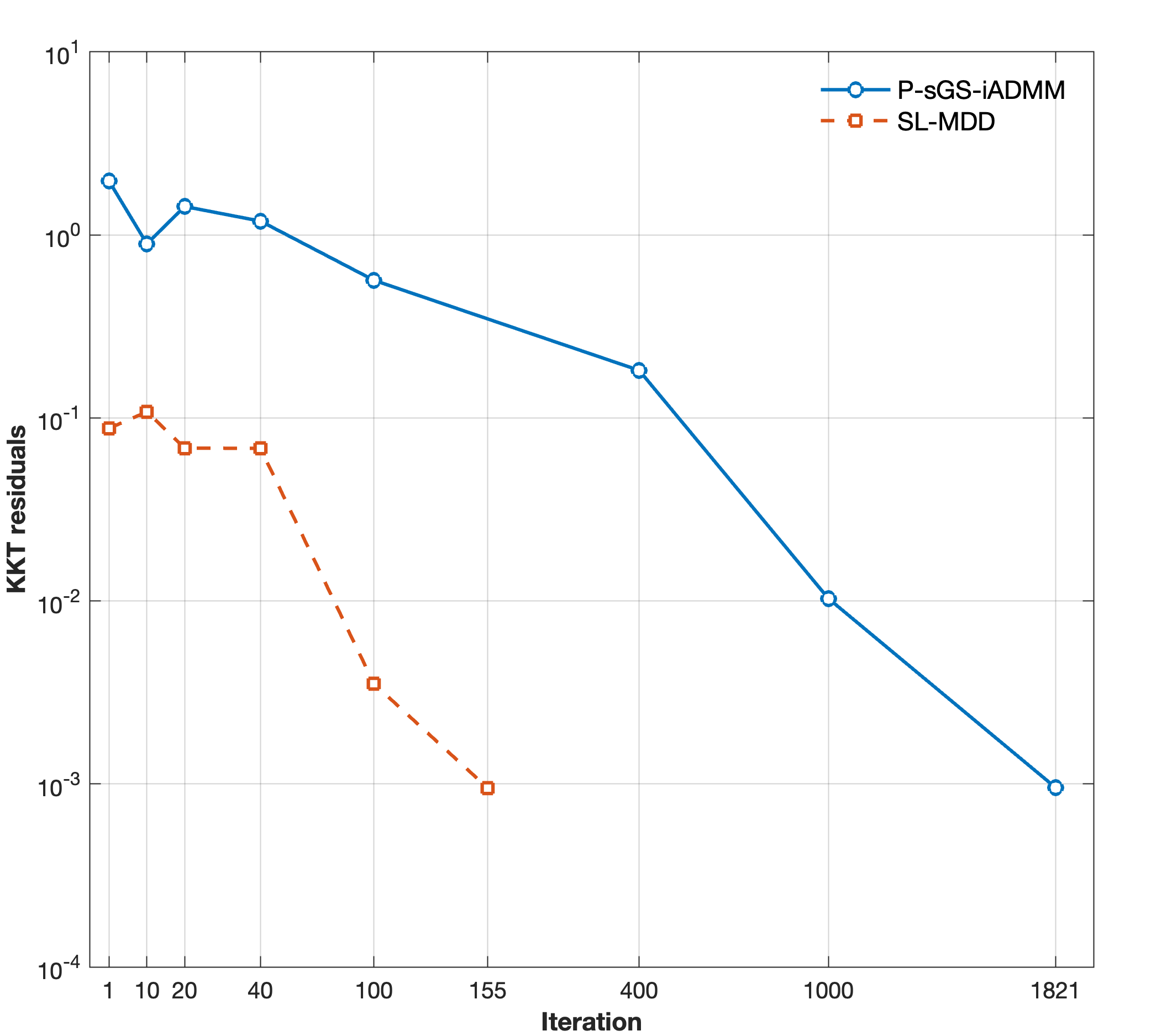}
    \\[2pt]
    \footnotesize(b) KKT residuals vs iteration
  \end{minipage}
  \hfill
  \begin{minipage}{0.32\textwidth}
    \centering
    \includegraphics[width=\linewidth, height=0.8\linewidth]{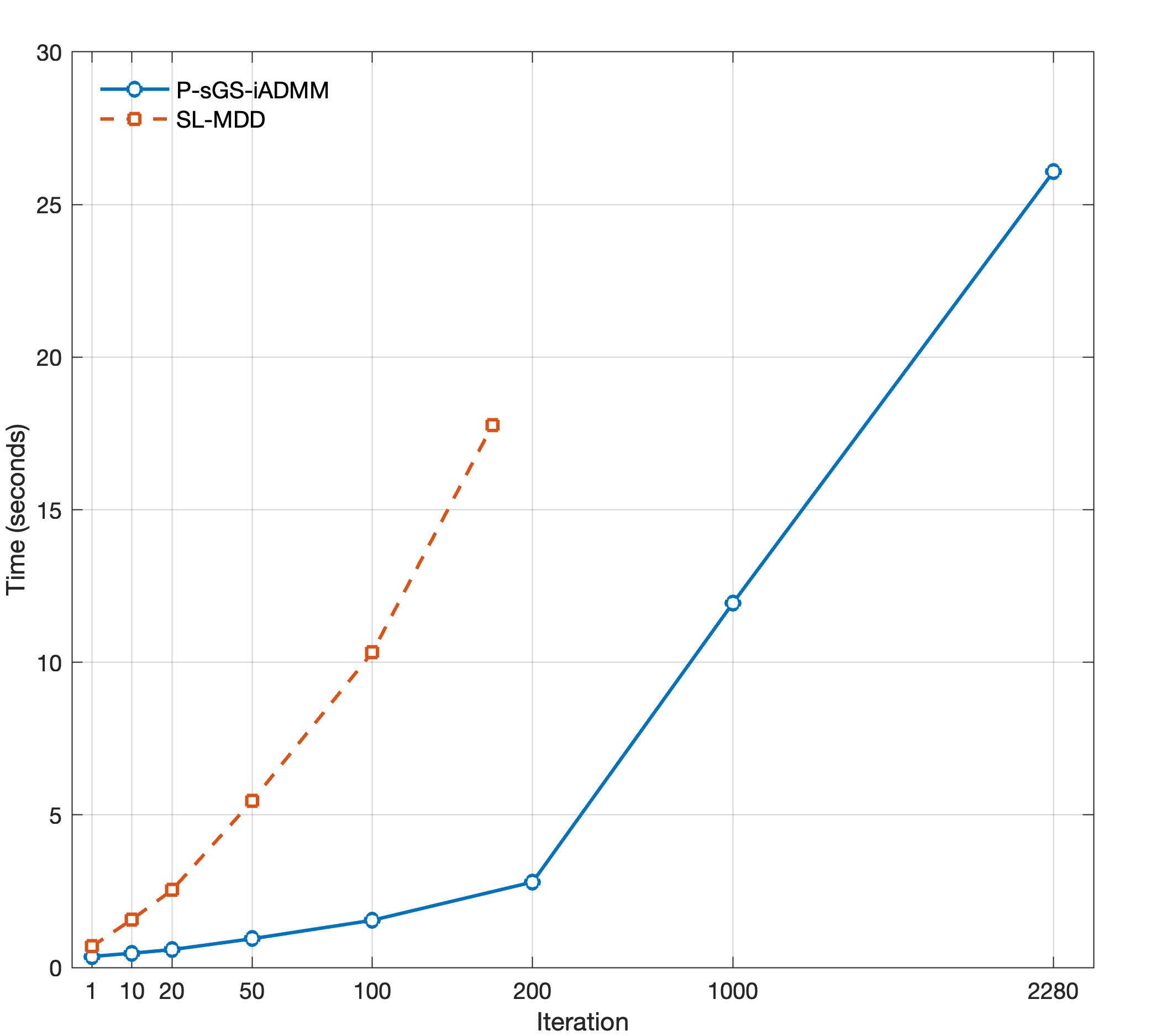}
    \\[2pt]
    \footnotesize(c) Parallel time vs iteration
  \end{minipage}
  \caption{Instance 3 (Experiment 2): $d=400, T=3$}
  \label{ffig:Instance3}
\end{figure}
\begin{table}[htbp]
\footnotesize
\caption{Performance comparison on Experiment 2, Group 1}\label{tab:performance4}
\centering
\begin{tabular}{ccccccccc}
\toprule
\multirow{2}{*}{Instances} & \multirow{2}{*}{$T$} & \multirow{2}{*}{$d$} & \multicolumn{2}{c}{Iter} & \multicolumn{2}{c}{Time (s)} & \multicolumn{2}{c}{Parallel Time (s)} \\
\cmidrule(r){4-5} \cmidrule(r){6-7} \cmidrule(r){8-9}
 & & & D & P & D & P & D & P \\
\midrule
         1 & 3 & 5    & 74    & 1387  & 2.66      & 34.89     &0.14     & 1.37      \\
        2 & 3 & 200  &  105    & 1740  & 12.66    & 148.05    & 0.38     & 1.89     \\
        3 & 3 & 400  &  155    & 1821  &  30.75   & 392.51   & 0.86      & 4.21    \\\bottomrule
\end{tabular}
\end{table}
\begin{figure}[htbp]
  \centering
  \begin{minipage}{0.32\textwidth}
    \centering
    \includegraphics[width=\linewidth, height=0.8\linewidth]{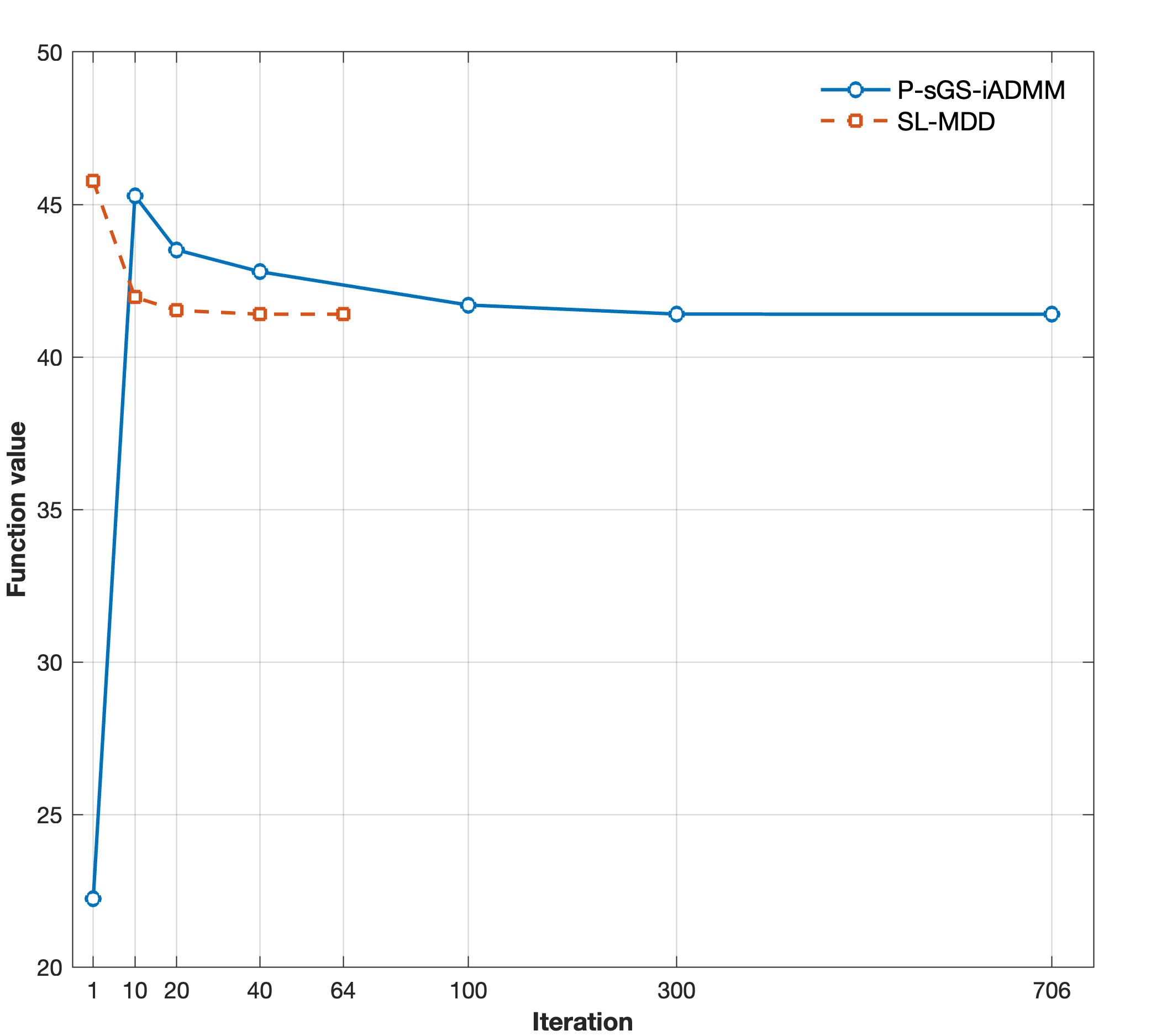}
    \\[2pt]
    \footnotesize(a) Function value vs iteration
  \end{minipage}
  \hfill
  \begin{minipage}{0.32\textwidth}
    \centering
    \includegraphics[width=\linewidth, height=0.8\linewidth]{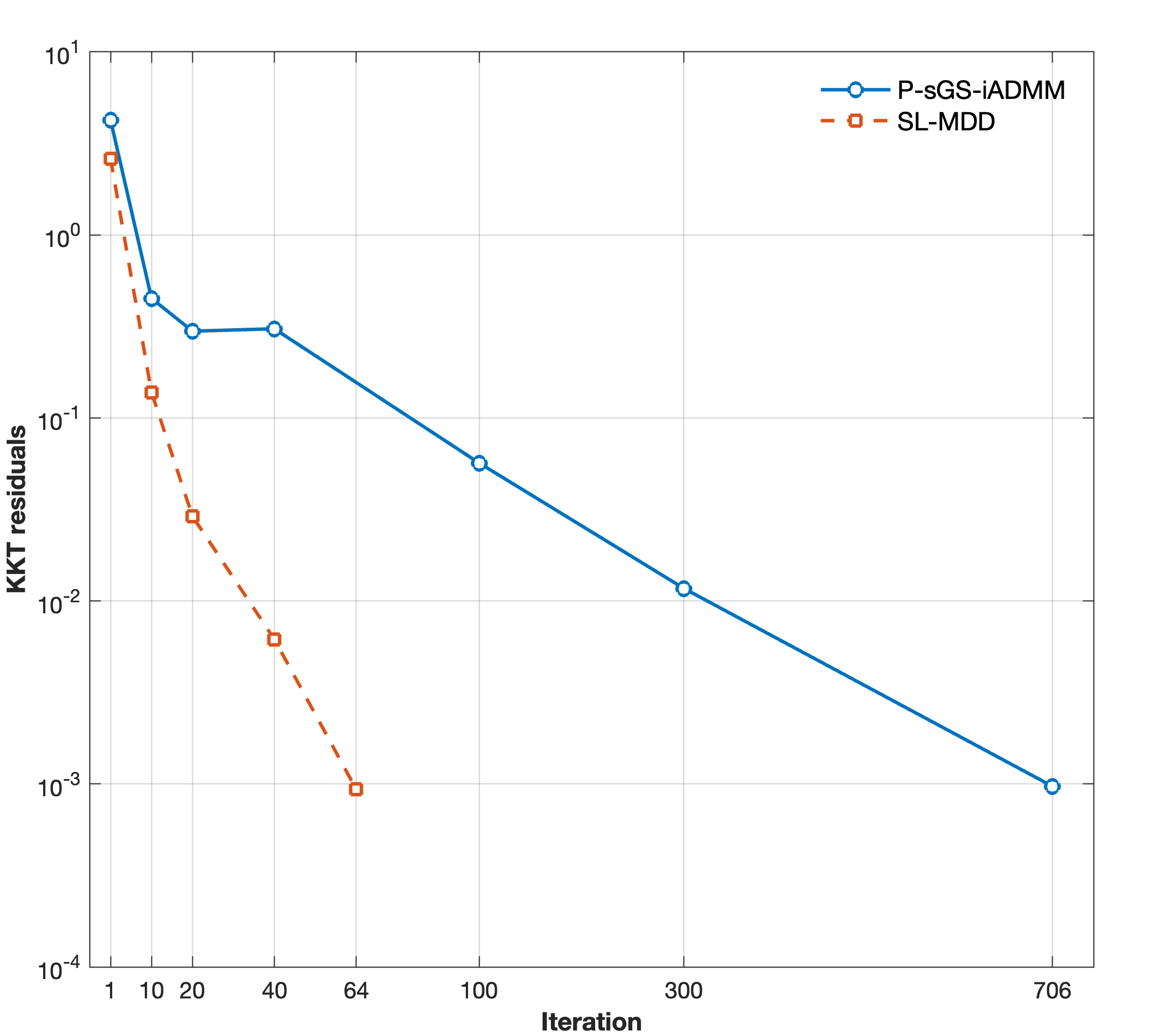}
    \\[2pt]
    \footnotesize(b) KKT residuals vs iteration
  \end{minipage}
  \hfill
  \begin{minipage}{0.32\textwidth}
    \centering
    \includegraphics[width=\linewidth, height=0.8\linewidth]{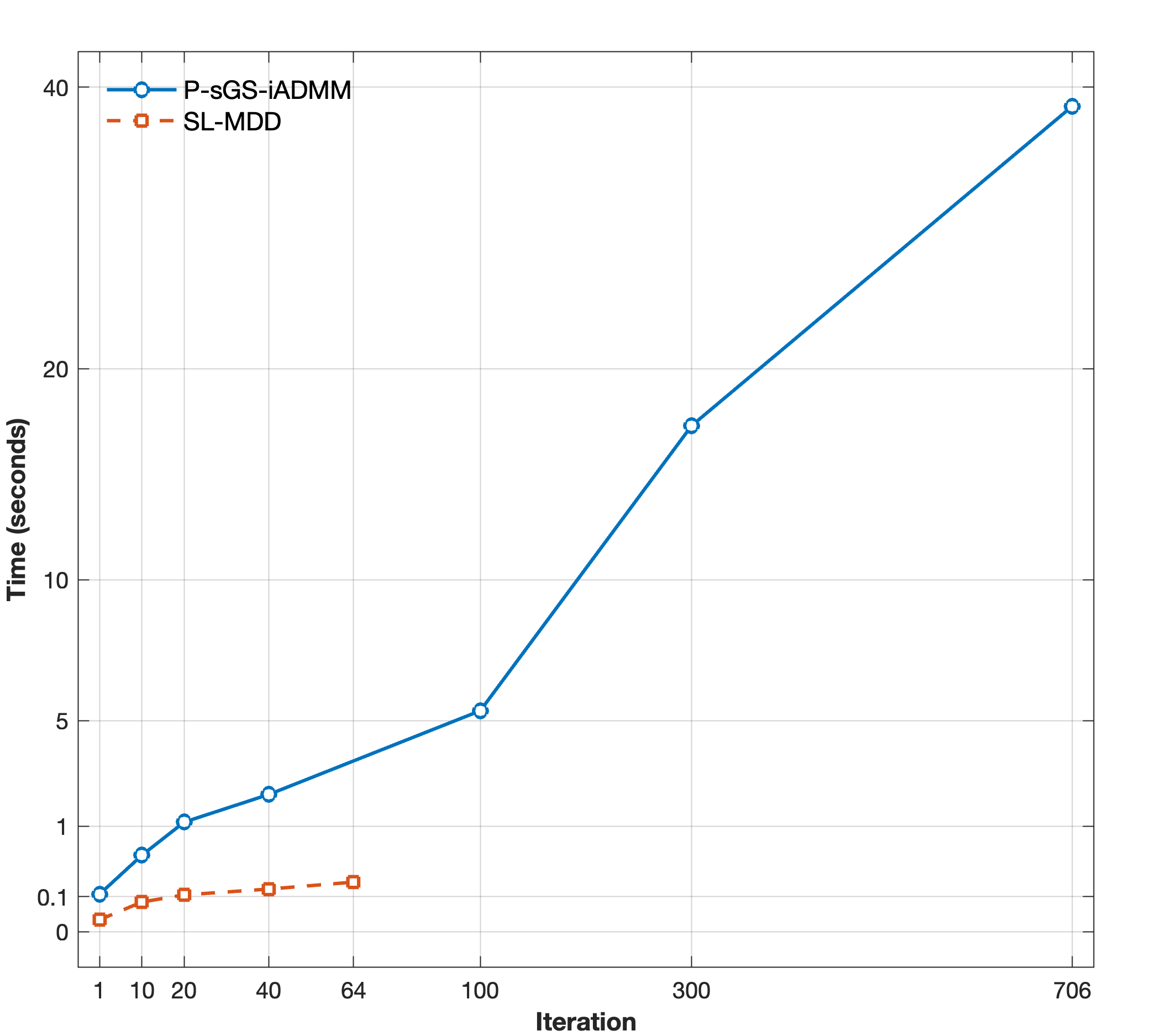}
    \\[2pt]
    \footnotesize(c) Parallel time vs iteration
  \end{minipage}
  \caption{Instance 4 (Experiment 2): $d=5, T=4$}
  \label{ffig:Instance4}
\end{figure}

\begin{figure}[htbp]
  \centering
  \begin{minipage}{0.32\textwidth}
    \centering
    \includegraphics[width=\linewidth, height=0.8\linewidth]{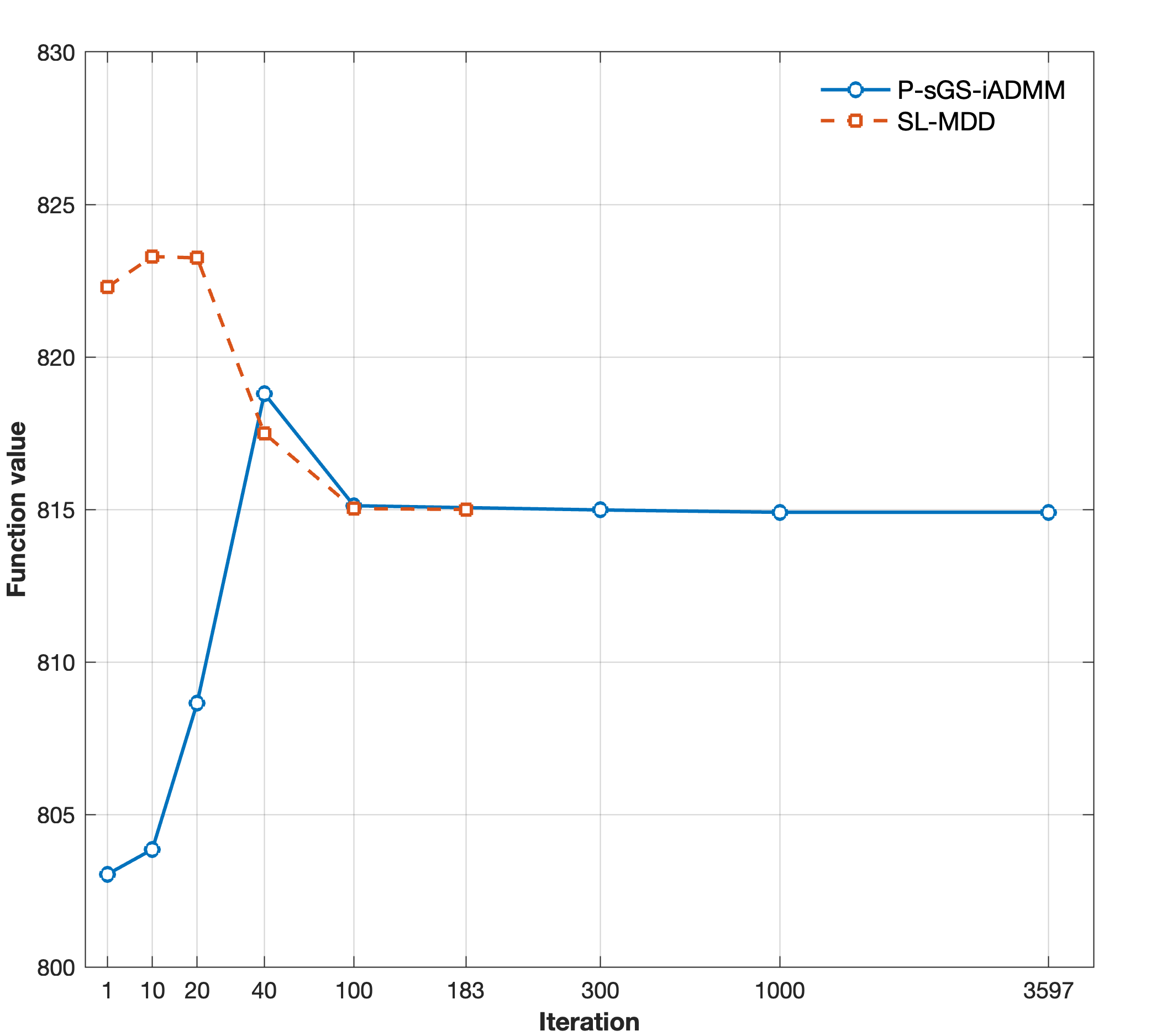}
    \\[2pt]
    \footnotesize(a) Function value vs iteration
  \end{minipage}
  \hfill
  \begin{minipage}{0.32\textwidth}
    \centering
    \includegraphics[width=\linewidth, height=0.8\linewidth]{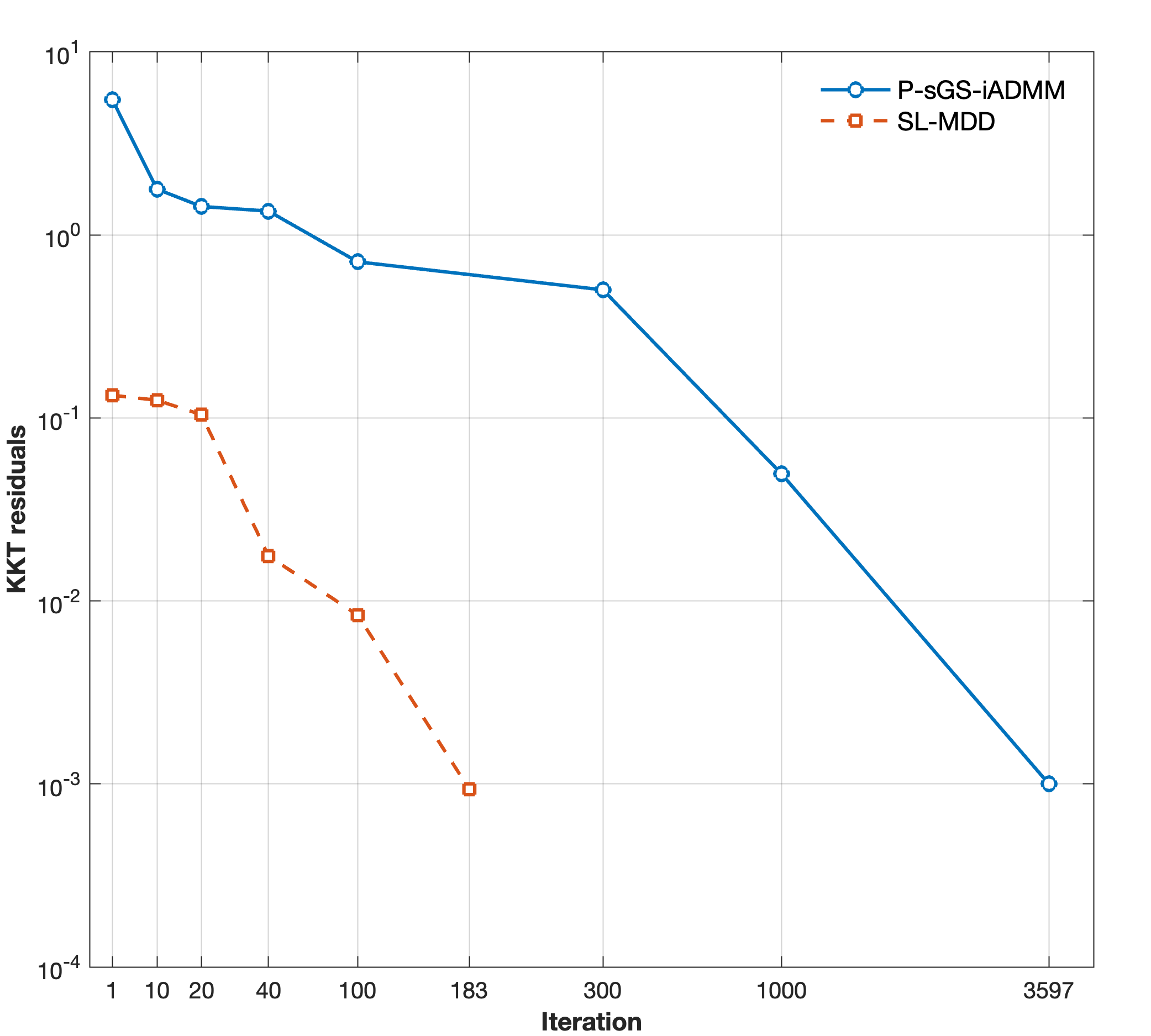}
    \\[2pt]
    \footnotesize(b) KKT residuals vs iteration
  \end{minipage}
  \hfill
  \begin{minipage}{0.32\textwidth}
    \centering
    \includegraphics[width=\linewidth, height=0.8\linewidth]{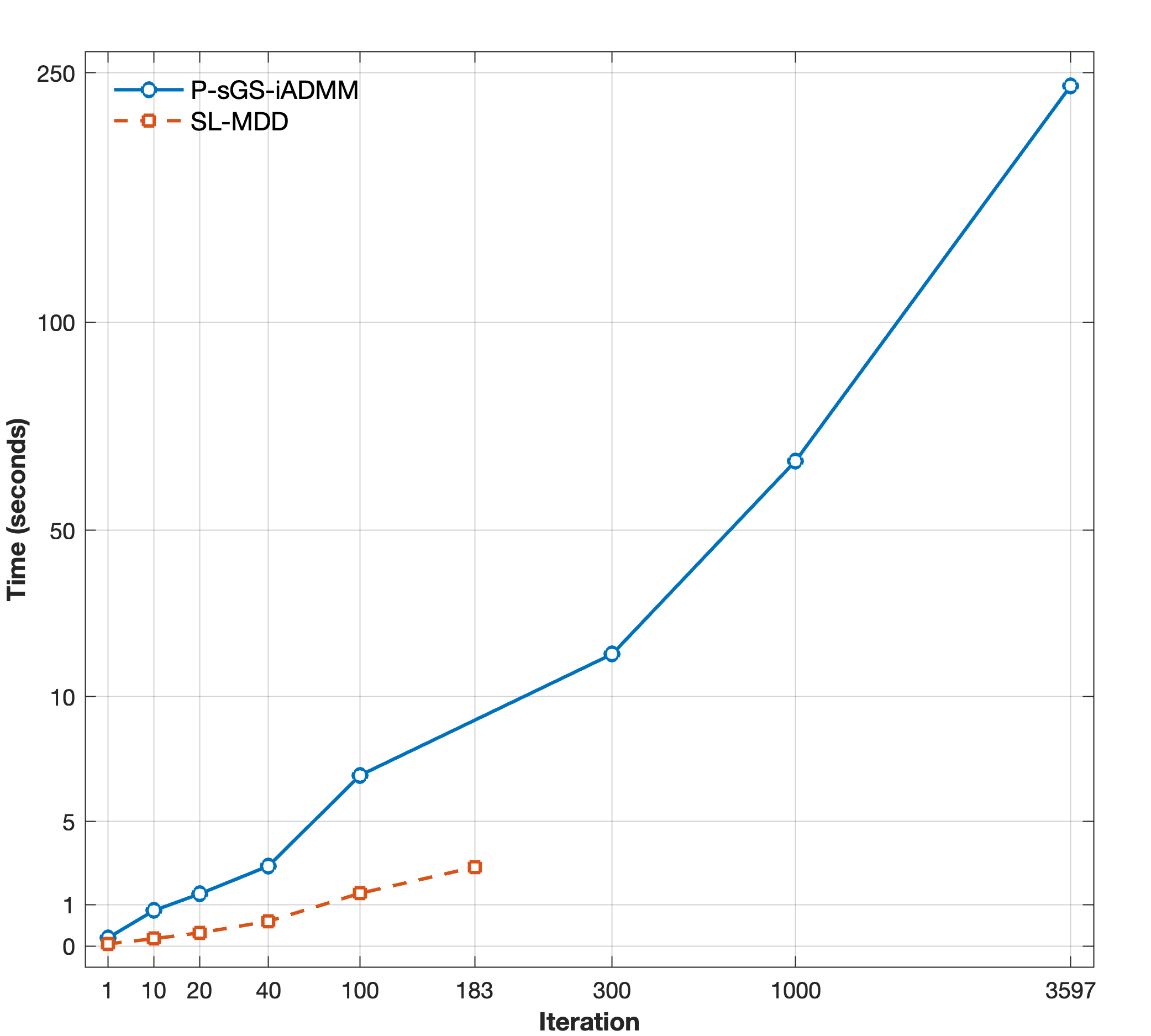}
    \\[2pt]
    \footnotesize(c) Parallel time vs iteration
  \end{minipage}
  \caption{Instance 5 (Experiment 2): $d=200, T=4$}
  \label{ffig:Instance5}
\end{figure}

\begin{figure}[htbp]
  \centering
  \begin{minipage}{0.32\textwidth}
    \centering
    \includegraphics[width=\linewidth, height=0.8\linewidth]{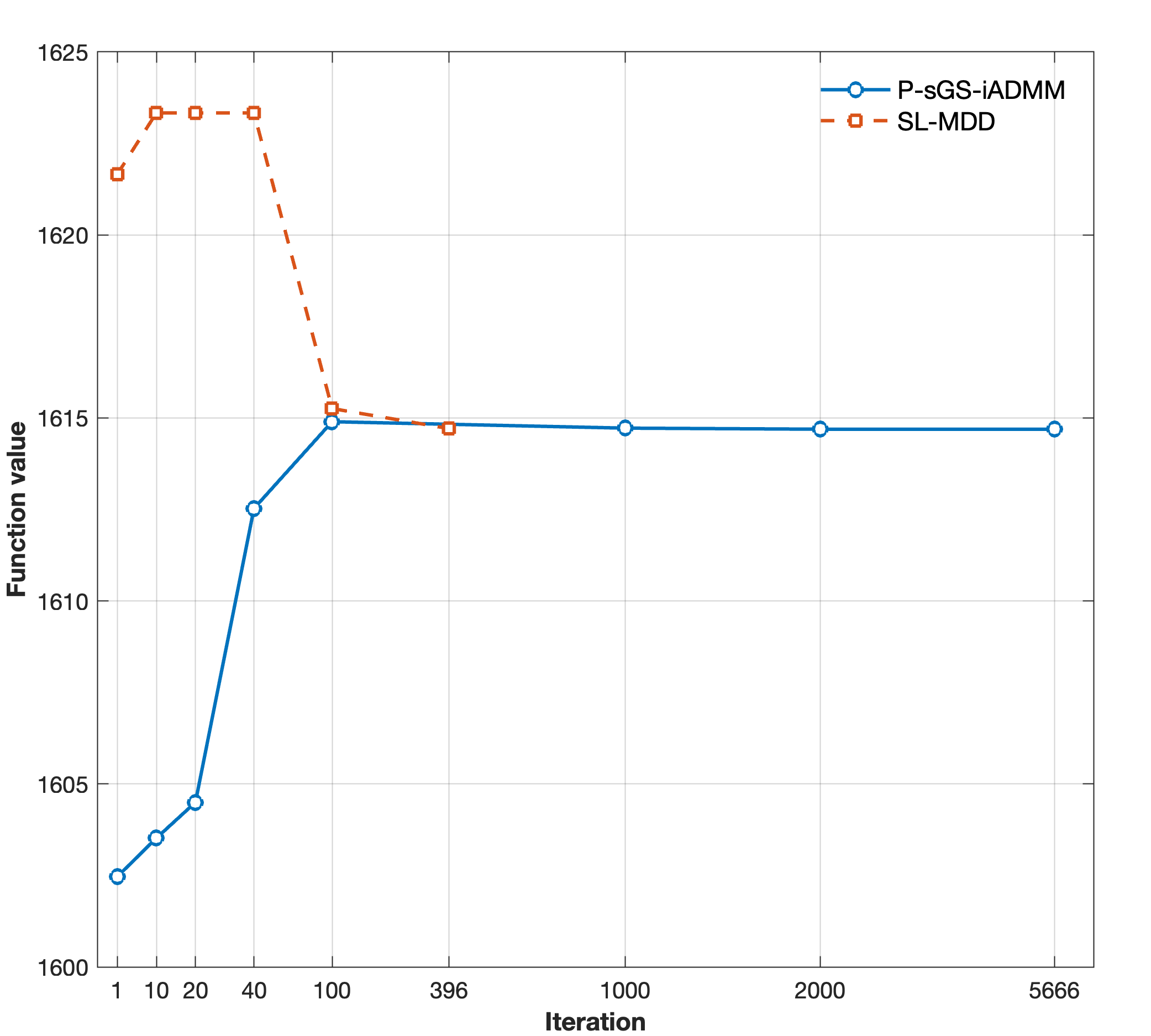}
    \\[2pt]
    \footnotesize(a) Function value vs iteration
  \end{minipage}
  \hfill
  \begin{minipage}{0.32\textwidth}
    \centering
    \includegraphics[width=\linewidth, height=0.8\linewidth]{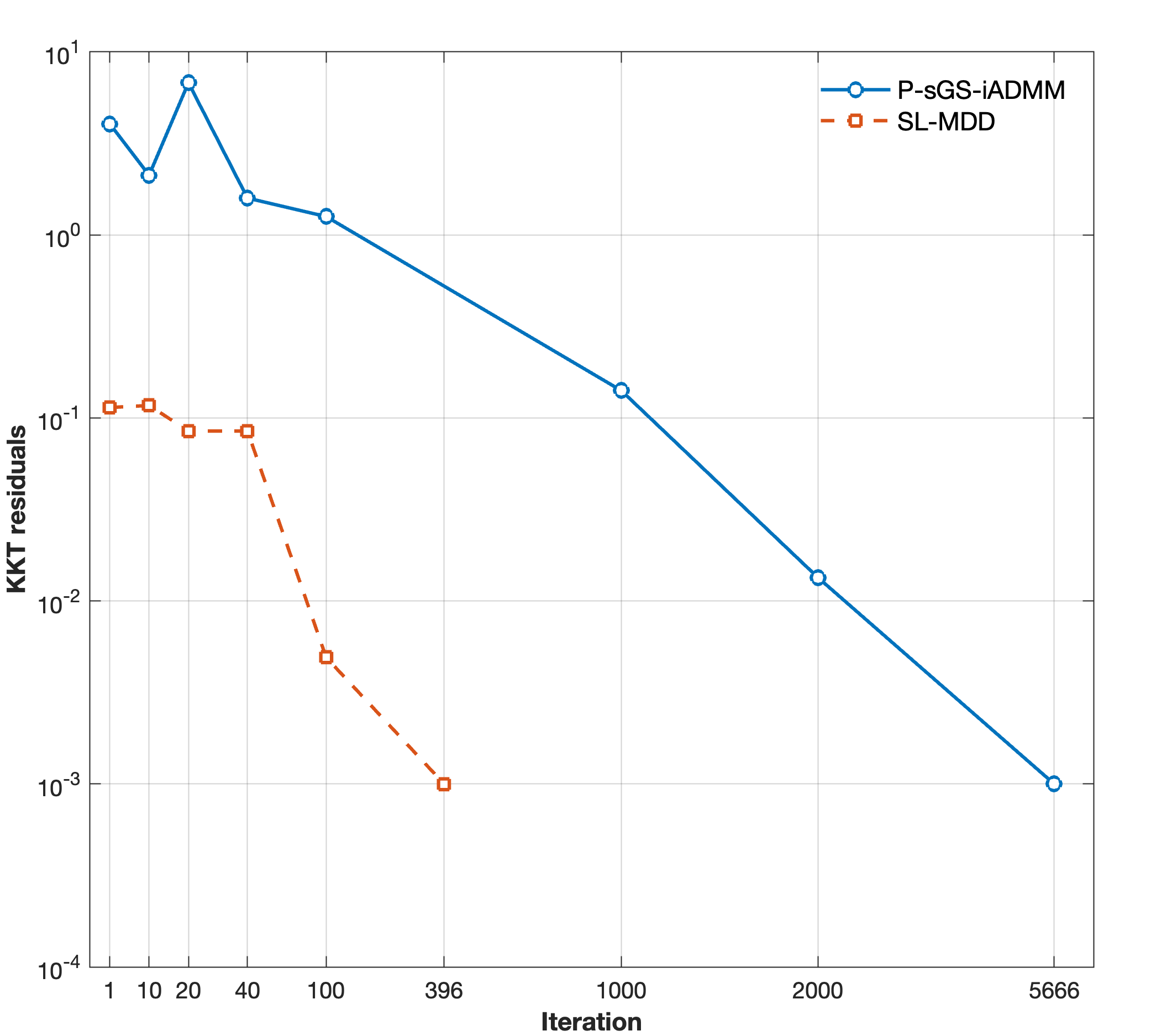}
    \\[2pt]
    \footnotesize(b) KKT residuals vs iteration
  \end{minipage}
  \hfill
  \begin{minipage}{0.32\textwidth}
    \centering
    \includegraphics[width=\linewidth, height=0.8\linewidth]{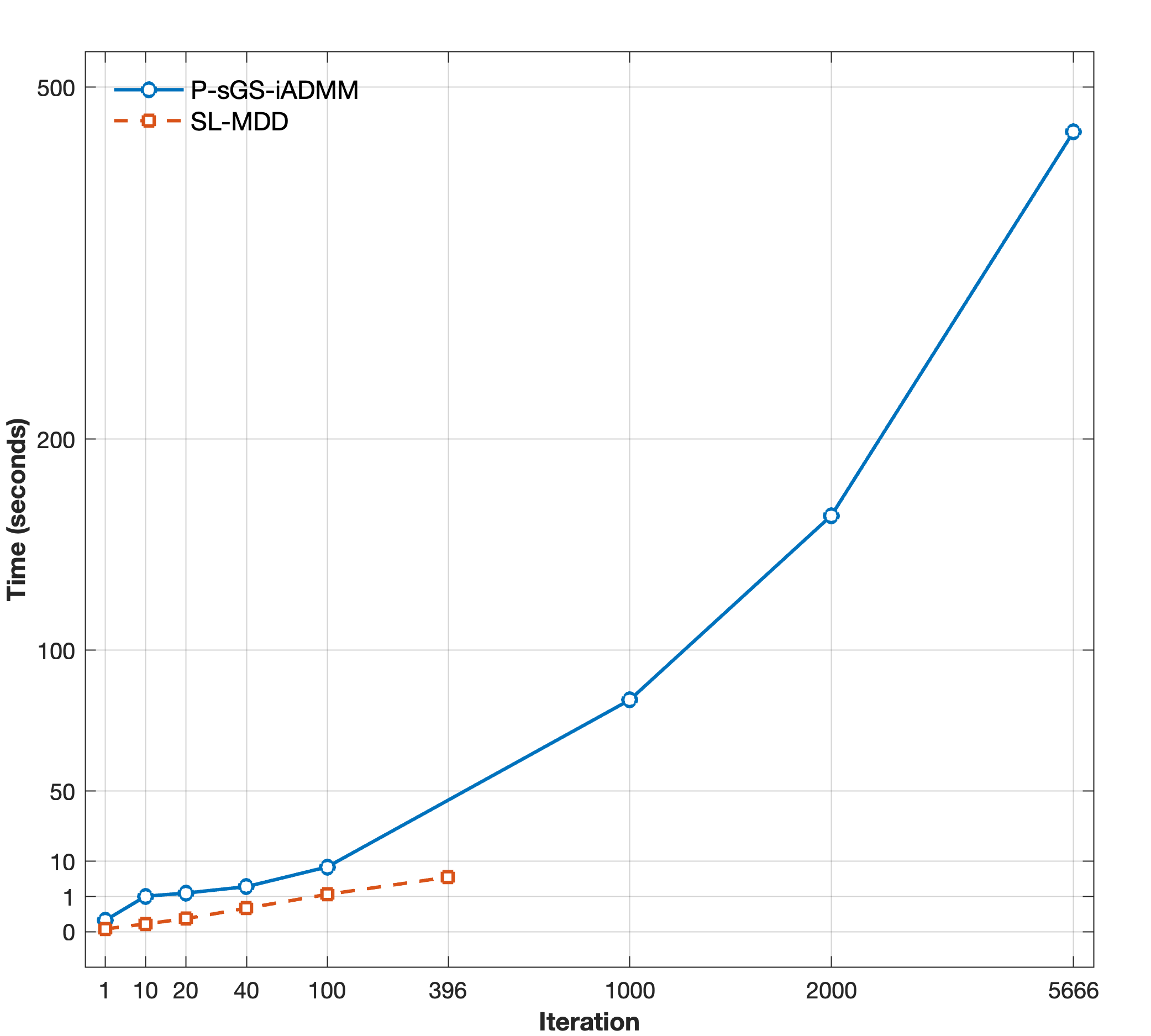}
    \\[2pt]
    \footnotesize(c) Parallel time vs iteration
  \end{minipage}
  \caption{Instance 6 (Experiment 2): $d=400, T=4$}
  \label{ffig:Instance6}
\end{figure}

\begin{table}[htbp]
\footnotesize
\caption{Performance comparison on Experiment 2, Group 2}\label{tab:performance5}
\centering
\begin{tabular}{ccccccccc}
\toprule
\multirow{2}{*}{Instances} & \multirow{2}{*}{$T$} & \multirow{2}{*}{$d$} & \multicolumn{2}{c}{Iter} & \multicolumn{2}{c}{Time (s)} & \multicolumn{2}{c}{Parallel Time (s)} \\
\cmidrule(r){4-5} \cmidrule(r){6-7} \cmidrule(r){8-9}
 & & & D & P & D & P & D & P \\
\midrule
        4 & 4 & 5    & 64    & 706   &  217.11   &1391.10   & 0.29     & 38.65      \\
        5 & 4 & 200  & 183    &  3597  &  1516.01 & 15566.38 & 2.80    &  242.08    \\
        6 & 4 & 400  & 396    & 5666 & 3763.09 & 56070.33 & 5.98   & 461.09   \\
\bottomrule
\end{tabular}
\end{table}

\begin{figure}[htbp]
  \centering
  \begin{minipage}{0.32\textwidth}
    \centering
    \includegraphics[width=\linewidth, height=0.8\linewidth]{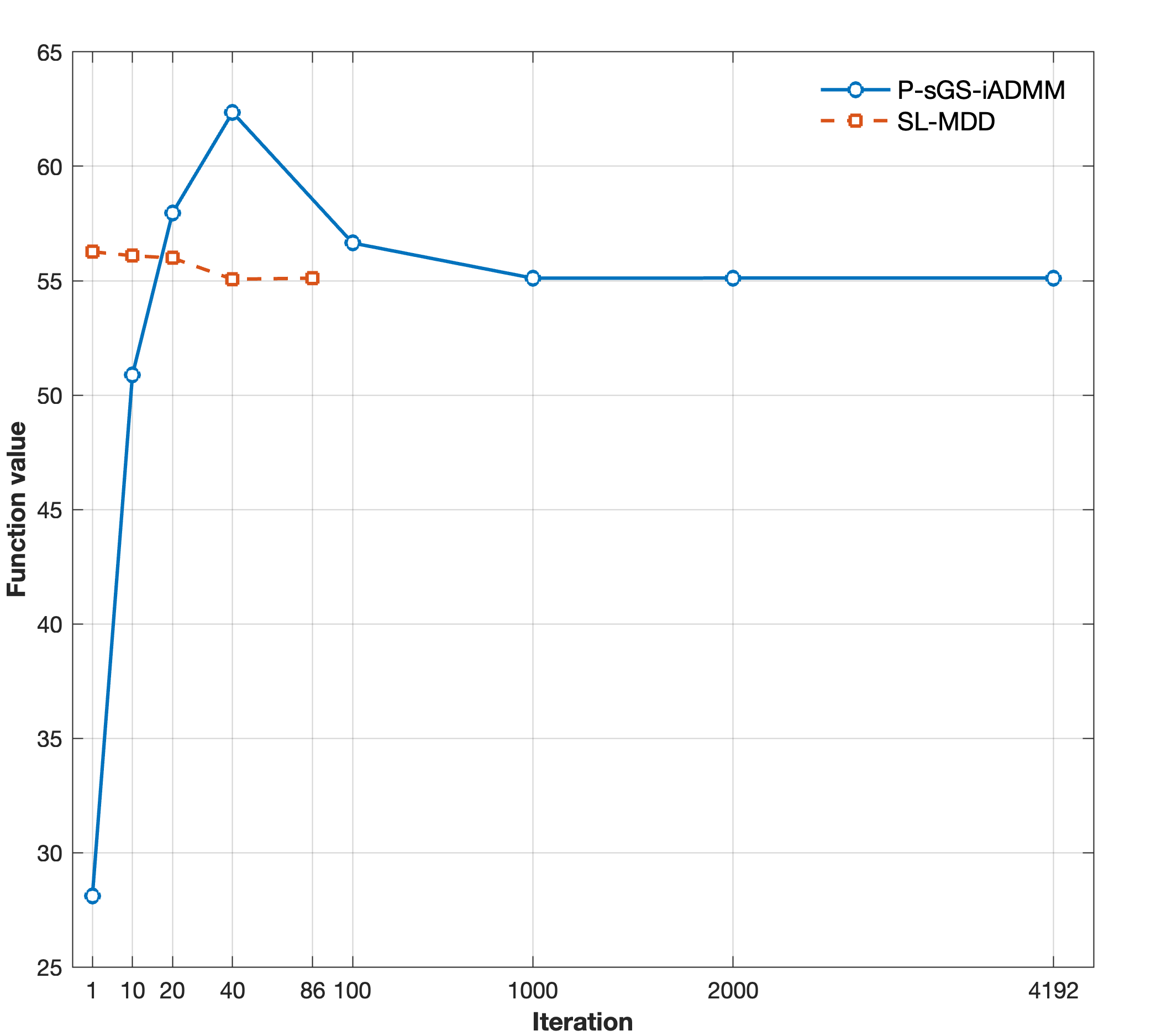}
    \\[2pt]
    \footnotesize(a) Function value vs iteration
  \end{minipage}
  \hfill
  \begin{minipage}{0.32\textwidth}
    \centering
    \includegraphics[width=\linewidth, height=0.8\linewidth]{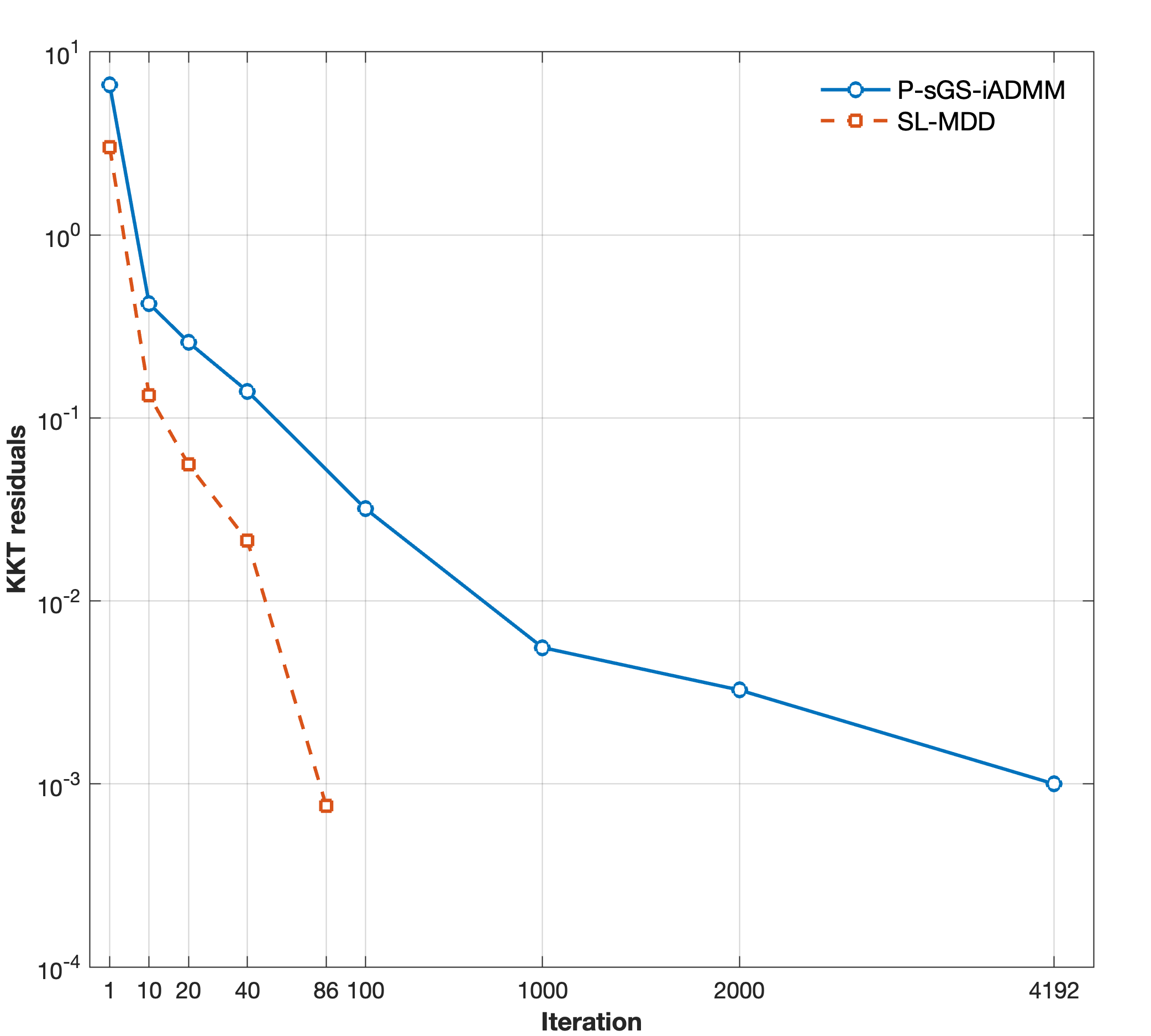}
    \\[2pt]
    \footnotesize(b) KKT residuals vs iteration
  \end{minipage}
  \hfill
  \begin{minipage}{0.32\textwidth}
    \centering
    \includegraphics[width=\linewidth, height=0.8\linewidth]{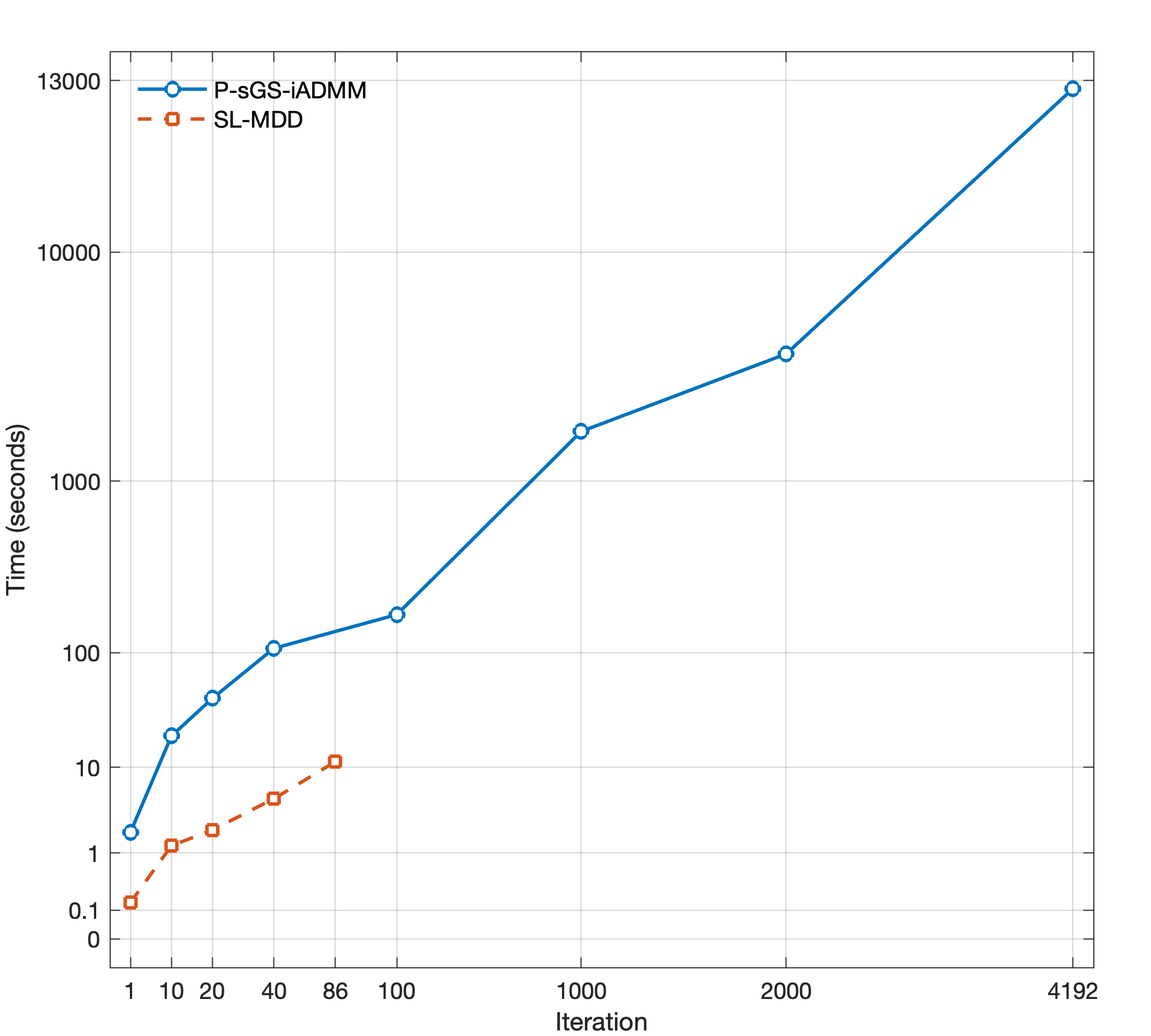}
    \\[2pt]
    \footnotesize(c) Parallel time vs iteration
  \end{minipage}
  \caption{Instance 7 (Experiment 2): $d=5, T=5$}
  \label{ffig:Instance7}
\end{figure}

\begin{table}[htbp]
\footnotesize
\caption{Performance comparison on Experiment 2, Group 3}\label{tab:performance6}
\centering
\begin{tabular}{ccccccccc}
\toprule
\multirow{2}{*}{Instances} & \multirow{2}{*}{$T$} & \multirow{2}{*}{$d$} & \multicolumn{2}{c}{Iter} & \multicolumn{2}{c}{Time (s)} & \multicolumn{2}{c}{Parallel Time (s)} \\
\cmidrule(r){4-5} \cmidrule(r){6-7} \cmidrule(r){8-9}
 & & & D & P & D & P & D & P \\
\midrule
1 & 3 & 5    & 74    & 1387  & 2.66      & 34.86     &0.14     & 1.37      \\
4 & 4 & 5    & 64    & 706   &  217.11   &1391.10   & 0.29     & 38.65      \\
7 & 5 & 5    & 86    & 4192   & 14329.40  & 415328.52 & 14.36    & 12860.03   \\
\bottomrule
\end{tabular}
\end{table}

For the case of exponential--$\ell_1$ composite objective functions, the results reported in Tables \ref{tab:performance4}--\ref{tab:performance6} and Figures \ref{ffig:Instance1}--\ref{ffig:Instance7} further illustrate the computational advantage of the SL-MDD method. In particular, Figures \ref{ffig:Instance1}--\ref{ffig:Instance7} show that it attains an optimal solution satisfying the KKT conditions in substantially fewer iterations than P-sGS-iADMM. Moreover, as the problem dimension $d$ and the number of stages $T$ increase, the trend in computational time remains consistent with that observed in Section \ref{sec:exp1}, indicating that the SL-MDD method maintains a faster performance compared to P-sGS-iADMM.

In summary, for both quadratic objective functions and the more general exponential and $\ell_1$-norm composite objective functions, the SL-MDD method exhibits better scalability and computational efficiency, making it practically useful for MSP with $T = 3$--$5$ and high-dimensional variables.

\section{Conclusions}\label{sec:conclusions}


In this paper, we develop a single-loop minorized dual decomposition (SL-MDD) framework for deterministic-equivalent nonsmooth multi-stage stochastic programming. The proposed method is derived from the intrinsic stage-wise and scenario-wise structure of MSP. At each iteration, a minorized problem is constructed, its restricted Wolfe dual is derived, and one sGS-iADMM sweep is performed to generate the next iterate. In this way, the minorization step, the dual reformulation, and the block splitting procedure are integrated into a unified single-loop recursion, rather than being organized as an outer-inner scheme for repeatedly solving successive subproblems.

From a computational perspective, the key feature of the proposed framework is that it reorganizes the stage/scenario coupling on the dual side and thereby yields a computationally favorable block structure. In particular, the $w$-updates reduce to structured linear-system solves, the $v$-updates reduce to proximal mappings, and the $z$-updates reduce to projection-type computations. The remaining coupled $y$-updates can be further decoupled by Lemma~\ref{lem:key}, which completes the decomposition of one iteration and makes the resulting subproblems suitable for stage-wise and scenario-wise parallel implementation.

For the three-stage case, we first show that the proposed SL-MDD method admits an equivalent single-loop minorized dual inexact semi-proximal ADMM representation, which is used only as a proof-oriented analytical device for convergence analysis. Based on this equivalent representation, we establish the global convergence of the generated iterates. The same proof pattern is then carried out for the general multi-stage setting, yielding the corresponding global convergence theorem. Numerical experiments illustrate the computational viability of the proposed framework and its favorable scaling behavior with respect to the stage-wise and scenario-wise structure of MSP instances.


\bibliographystyle{plain}\bibliography{reference}
\end{document}